%% file: IntroCCT.tex
\documentclass[12pt,reqno]{amsproc}
\usepackage[pdfauthor   = {Sebastian\ Posur},
            pdftitle    = {Constructive\ category\ theory},
            pdfsubject  = {},
            pdfkeywords = {},
            bookmarks=true,
            bookmarksopen=true,
            pagebackref=true,
            hyperindex=true,
            colorlinks=true,
            linkcolor=blue,
            citecolor=blue,
            filecolor=blue,
            urlcolor=blue,
            ]{hyperref}

\include{header}

\author{Sebastian Posur}
\thanks{The author is supported by Deutsche Forschungsgemeinschaft (DFG) grant SFB-TRR 195: \emph{Symbolic Tools in Mathematics and their Application}}
\address{Department of mathematics, University of Siegen, 57072 Siegen, Germany}
\email{\href{mailto:Sebastian Posur <sebastian.posur@uni-siegen.de>}{sebastian.posur@uni-siegen.de}}

\begin{document}

\title[Methods of constructive category theory]{Methods of constructive category theory}

\begin{abstract}
We give an introduction to constructive category theory by answering two guiding computational questions.
The first question is: how do we compute the set of all natural transformations between two finitely presented functors like $\Ext$ and $\Tor$
over a commutative coherent ring $R$? We give an answer by introducing category constructors that enable us to build up
a category which is both suited for performing explicit calculations and equivalent to the category of all finitely presented functors.
The second question is:
how do we determine the differentials on the pages of a spectral sequence associated to a filtered cochain complex
only in terms of operations directly provided by the axioms of an abelian category?
Its answer relies on a constructive method for performing diagram chases
based on a calculus of relations within an arbitrary abelian category.
\end{abstract}


\maketitle

\tableofcontents

\section*{Introduction}
Basic algorithms in computer algebra provide answers for basic mathematical questions.
The Gaussian algorithm computes solutions of a given linear system over a field $k$.
The Euclidean algorithm computes the gcd of elements in an Euclidean domain.
Buchberger's algorithm \cite{Buch} computes Gröbner bases of (homogeneous) ideals $I = \langle f_1, \dots, f_r \rangle$ in a (graded) polynomial ring $R \coloneqq k[ x_1, \dots, x_n ]$ for $r, n \in \N$,
allowing us to answer many basic questions\footnote{For learning how to answer these questions computationally, we refer the reader to \cite{GP}.}: when do two representatives of elements in the residue class ring $k[ x_1, \dots, x_n ]/ I$ define the same element?
How to find a finite set of generators of the ideal $k[ x_1, \dots, x_m ] \cap I$ for $m < n$?
How to find a generating set for the syzygies of the given generators of $I$?
Generalizations of Buchberger's algorithm \cite{Green99}
provide answers to similar questions for some non-commutative rings, like finite dimensional quotients of path algebras.

In this article,
we demonstrate a strategy
that uses these basic algorithms as building blocks 
for answering a more high-level mathematical question:
\begin{enumerate}
 \item How do we compute the set of all natural transformations between two finitely presented functors over a commutative coherent\footnote{A commutative ring is coherent if kernels of $R$-module homomorphisms between finitely generated free $R$-modules are themselves finitely generated.} ring $R$?
\end{enumerate}
Examples of finitely presented functors over such rings are given by $\Ext^i(M,-)$ and $\Tor_i(M,-)$ for a finitely presented $R$-module $M$ and $i \in \N_0$.

The first section of this article is dedicated to answering this question.
The main idea is to use a constructive formulation of category theory.
We regard a category $\AC$ as a computational entity on whose objects and morphisms 
we can operate by algorithms. For example, composition of morphisms is an algorithm that takes two morphisms
$\alpha: A \rightarrow B$, $\beta: B \rightarrow C$ as input and outputs a new morphism $\alpha \cdot \beta: A \rightarrow C$.
Equality of morphisms is an algorithm that takes two morphisms $\alpha: A \rightarrow B$, $\alpha': A \rightarrow B$,
and outputs $\mathtt{true}$ if $\alpha$ and $\alpha'$ are equal, $\mathtt{false}$ otherwise.

The basic algorithms of computer algebra can now be used to render concrete instances of categories computable in the above sense.
For example, if we regard the quotient ring
$R = k[ x_1, \dots, x_n ]/ I$ as a category with a single object whose morphisms are given by the elements of $R$ and composition
by ring multiplication,
then deciding equality of morphisms is the same as deciding equality of ring elements, which is algorithmically realized by Buchberger's algorithm.
As another example, Gaussian elimination serves as an algorithm to realize the computation of kernels in a computational model of finite dimensional vector spaces.

Once reinterpreted in purely categorical terms, we can forget about the internal functioning of the basic algorithms
and start building up algorithms that solely rely on category theory specific notions.
In this way, we will be able to answer the more high-level mathematical question stated above,
i.e., we end up with an algorithmic strategy for computing sets of natural transformations between finitely presented functors.

Moreover, once we get used to the idea of using purely categorical notions as building blocks of
our algorithms, we can ask further questions that are founded on this idea:
\begin{enumerate}
 \setcounter{enumi}{1}
 \item How do we construct morphisms that are claimed to exist by homological algebra, like the differentials on the pages of a spectral sequence associated to a filtered cochain complex, only in terms of operations directly provided by the axioms of an abelian category, like computing kernels or cokernels?
\end{enumerate}

The second section of this article deals with this second question
and its answer relies on the introduction of the concept of \emph{generalized morphisms} \cite{barhabil}.
They provide a key tool for a constructive treatment of homological algebra
that let us compute with spectral sequences in the end.

Our constructive treatment of category theory
has been implemented within a software project called \CapPkg \cite{CAP-project}, which consists of a collection of \GAP \cite{GAP4} packages.
To reveal the feasibility of a direct computer implementation of all the outlined ideas within this article,
we make use of a more constructive language of mathematics (see, for example, \cite{MRRConstructiveAlgebra}).
Concretely, this means that we make an intuitive use of terms like \emph{data types} and \emph{algorithms}
instead of sets and functions,
and treat the notion of equality between elements of data types
as an extra datum that has to be provided by an explicit algorithm
(instead of being inherently available like in the case of sets).
Since this more constructive language encompasses classical mathematics,
all given constructions and theorems are also valid classically.

We assume a classical understanding of basic notions in category theory:
categories, functors, natural transformations, and equivalences of categories.

\input{IntroCCT_chap1.tex}
\input{IntroCCT_chap2.tex}

\input{IntroCCT.bbl}

\end{document}

%% file: header.tex
\usepackage[utf8]{inputenc} 
\usepackage[T1]{fontenc}
\usepackage{a4wide}
\usepackage{lmodern}
\usepackage[english]{babel}
\usepackage{mathrsfs}
\usepackage{etex}
\usepackage{mathtools}
\usepackage{latexsym}
\usepackage{amssymb}
\usepackage{amsthm}
\usepackage{amsmath}
\usepackage{caption}
\usepackage{mathabx}

\usepackage{colortbl}
\usepackage{tensor}

\usepackage[all]{xy}
\usepackage{verbatim}
\usepackage{listings}
\usepackage{fancyvrb}

\usepackage{graphicx}

\usepackage[dvipsnames]{xcolor}
\usepackage{accents} 
\usepackage{enumerate}
\usepackage{wrapfig}
\usepackage{tikz}
\usepackage{tikz-cd}
\usetikzlibrary{automata,shapes,arrows,matrix,backgrounds,positioning,plotmarks,calc,patterns,matrix,decorations.pathreplacing,decorations.pathmorphing,decorations.text,decorations.markings}
\usepackage[colorinlistoftodos,shadow]{todonotes}

\usepackage{multirow}
\usepackage{mdwlist}

\usepackage{stmaryrd}
\usepackage{mathdots} 

\usepackage{fancyvrb}

\usepackage[toc,page]{appendix}
\usepackage{float}

\usepackage{extarrows}
\usepackage{pdflscape}
\usepackage{rotating}
\usepackage{etex}
\newtheoremstyle{mytheoremstyle} 
    {5pt}                    
    {5pt}                    
    {\itshape}                   
    {\parindent}                           
    {\bf}                   
    {.}                          
    {.5em}                       
    {}  

\theoremstyle{mytheoremstyle}

\newtheorem{theorem}{Theorem}[section]

\newtheorem{lemma}[theorem]{Lemma}

\newtheorem{corollary}[theorem]{Corollary}

\newtheoremstyle{mytdefintionstyle} 
    {5pt}                    
    {5pt}                    
    {\rm}                   
    {\parindent}                           
    {\bf}                   
    {.}                          
    {.5em}                       
    {}  

\theoremstyle{remark}
\newtheorem{remark}[theorem]{Remark}

\theoremstyle{mytdefintionstyle}
\newtheorem{definition}[theorem]{Definition}

\newtheorem{example}[theorem]{Example}

\newtheorem{construction}[theorem]{Construction}

\newtheoremstyle{exmp_contd} 
{\topsep} {\topsep}%
{\upshape}
{}
{\bfseries}
{}
{ }
{\thmname{#1}\,\thmnumber{ #2}\thmnote{#3}\enspace(continued)}

\theoremstyle{exmp_contd}

\usepackage{xspace}
\input{pre.tex}

\tikzset{round left paren/.style={ncbar=0.5cm,out=120,in=-120}}
\tikzset{round right paren/.style={ncbar=0.5cm,out=60,in=-60}}
\newcolumntype{C}[1]{>{\centering\arraybackslash$}p{#1}<{$}}
\newlength{\mycolwd}
\usepackage{array, xcolor}
\definecolor{lightgray}{gray}{0.8}
\newcolumntype{L}{>{\raggedleft}p{0.28\textwidth}}
\newcolumntype{R}{p{0.8\textwidth}}

\definecolor{ctcolor}{gray}{0.95}

\definecolor{ctucolor}{gray}{0.85}

\makeatletter
\newcommand{\thickhline}{%
    \noalign {\ifnum 0=`}\fi \hrule height 1pt
    \futurelet \reserved@a \@xhline
}
\newcolumntype{"}{@{\hskip\tabcolsep\vrule width 1pt\hskip\tabcolsep}}
\makeatother

\usepackage{enumitem}
\newlist{theoremenumerate}{enumerate}{1}
\setlist[theoremenumerate]{label=(\arabic{theoremenumeratei}), ref=\thetheorem.(\arabic{theoremenumeratei}),noitemsep}

%% file: pre.tex
\definecolor{ExQ}{HTML}{0000FF}
\definecolor{Dec}{HTML}{E07B00}

\newcommand{\CapPkg}{\textsc{Cap}\xspace}

\DeclareMathOperator{\Proj}{Proj}

\newcommand{\AC}{\mathbf{A}}
\newcommand{\BC}{\mathbf{B}}

\newcommand{\firstproj}[2]{ \tensor*[_{#1}]{\begin{bmatrix}1\\0\end{bmatrix}}{_{#2}} }
\newcommand{\secondproj}[2]{ \tensor*[_{#1}]{\begin{bmatrix}0\\1\end{bmatrix}}{_{#2}} }
\newcommand{\bmatrowind}[4]{ \tensor*[_{#1}]{\begin{bmatrix}{#2} & {#3}\end{bmatrix}}{_{#4}} }

\newcommand{\pmatrow}[2]{ \begin{pmatrix}{#1} & {#2} \end{pmatrix} }
\newcommand{\pmatcol}[2]{ \begin{pmatrix}{#1} \\ {#2} \end{pmatrix} }

\newcommand{\N}{\mathbb{N}}
\newcommand{\Z}{\mathbb{Z}}

\newcommand{\Q}{\mathbb{Q}}

\newcommand{\Set}{\mathbf{Set}}

\newcommand{\AsCat}{\mathcal{C}}
\newcommand{\Free}{\mathrm{Free}}

\newcommand{\CH}{\mathrm{H}}

\newcommand{\SVec}{\mathrm{Rows}}

\newcommand{\Ab}{\mathbf{Ab}}

\newcommand{\op}{\mathrm{op}}

\newcommand{\Obj}{\mathrm{Obj}}
\DeclareMathOperator{\ObjC}{\Obj_{\mathbf{C}}}

\newcommand{\Freyd}{\mathcal{A}}

\newcommand{\im}{\mathrm{im}}

\newcommand{\tr}{\mathrm{tr}}

\newcommand{\id}{\mathrm{id}}

\DeclareMathOperator{\Hom}{\mathrm{Hom}}

\DeclareMathOperator{\HomC}{\Hom_{\mathbf{C}}}

\DeclareMathOperator{\kernel}{\mathrm{ker}}
\DeclareMathOperator{\weakkernel}{\mathrm{WeakKernel}}
\DeclareMathOperator{\WeakPullback}{\mathrm{WeakPullback}}

\DeclareMathOperator{\gkernel}{\mathrm{gker}}
\DeclareMathOperator{\cokernel}{\mathrm{coker}}

\DeclareMathOperator{\domain}{\mathrm{dom}}
\DeclareMathOperator{\image}{\mathrm{im}}

\DeclareMathOperator{\defect}{\mathrm{def}}
\DeclareMathOperator{\generalizedimage}{\mathrm{gim}}

\DeclareMathOperator{\Tor}{\mathrm{Tor}}
\DeclareMathOperator{\Ext}{\mathrm{Ext}}

\DeclareMathOperator{\Rel}{\mathrm{Rel}}
\DeclareMathOperator{\Span}{\mathrm{Span}}

\DeclareMathOperator{\G}{\mathbf{G}}

\newcommand{\emb}{\mathrm{emb}}
\newcommand{\proj}{\mathrm{proj}}

\newcommand{\fpmodl}{\text{-}\mathrm{fpmod}}

\newcommand{\fpgrmodl}{\text{-}\mathrm{fpgrmod}}

\newcommand{\Rows}{\mathrm{Rows}}

\newcommand{\fp}{\mathrm{fp}}

\newcommand{\GAP}{\texttt{GAP}\xspace}

\newcommand{\Source}{\mathrm{Source}}
\newcommand{\Range}{\mathrm{Range}}

\newcommand{\KernelObject}{\mathrm{KernelObject}}
\newcommand{\KernelEmbedding}{\mathrm{KernelEmbedding}}
\newcommand{\WeakKernelEmbedding}{\mathrm{WeakKernelEmbedding}}

\newcommand{\KernelLift}{\mathrm{KernelLift}}
\newcommand{\WeakKernelLift}{\mathrm{WeakKernelLift}}
\newcommand{\CokernelColift}{\mathrm{CokernelColift}}

\newcommand{\CokernelObject}{\mathrm{CokernelObject}}

\newcommand{\CokernelProjection}{\mathrm{CokernelProjection}}

\newcommand{\bfindex}[1]{{\textbf{#1}}{\index{{#1}}}}

%% file: IntroCCT_chap1.tex
\section{Category constructors}

In this section, we make use of the concept of category constructors
in order to build up a category equivalent to the category of finitely presented functors.
Simply put, a category constructor is an operation that produces a category from some given input:
\begin{center}
    \begin{tikzpicture}[label/.style={postaction={
      decorate,
      decoration={markings, mark=at position .5 with \node #1;}}},
      mylabel/.style={thick, draw=black, align=center, minimum width=0.5cm, minimum height=0.5cm,fill=white}]
      \coordinate (r) at (8,0);
      \node[mylabel] (A) {some input};
      \node[mylabel] (B) at ($(A)+(r)$) {a category};
      \draw[shorten >=0.2em,shorten <=0.2em,->,thick] (A) --node[above]{category constructor} (B);
    \end{tikzpicture}
\end{center}
For example, we can regard a ring $R$ as a single object category $\AsCat(R)$ whose morphisms are given by the elements
of $R$ and composition is given by ring multiplication. This defines a category constructor:
\begin{center}
    \begin{tikzpicture}[label/.style={postaction={
      decorate,
      decoration={markings, mark=at position .5 with \node #1;}}},
      mylabel/.style={thick, draw=black, align=center, minimum width=0.5cm, minimum height=0.5cm,fill=white}]
      \coordinate (r) at (8,0);
      \node[mylabel] (A) {$R$ a ring};
      \node[mylabel] (B) at ($(A)+(r)$) {$R$ regarded as a category};
      \draw[shorten >=0.2em,shorten <=0.2em,->,thick] (A) --node[above]{$\AsCat(-)$} (B);
    \end{tikzpicture}
\end{center}
The input of a category constructor can of course itself consist of a category, for example in the case of taking the opposite category:
\begin{center}
    \begin{tikzpicture}[label/.style={postaction={
      decorate,
      decoration={markings, mark=at position .5 with \node #1;}}},
      mylabel/.style={thick, draw=black, align=center, minimum width=0.5cm, minimum height=0.5cm,fill=white}]
      \coordinate (r) at (6,0);
      \node[mylabel] (A) {$\AC$ a category};
      \node[mylabel] (B) at ($(A)+(r)$) {$\AC^{\op} \coloneqq $ opposite of $\AC$};
      \draw[shorten >=0.2em,shorten <=0.2em,->,thick] (A) --node[above]{$(-)^{\op}$} (B);
    \end{tikzpicture}
\end{center}

Other important examples of category constructors introduced in this section are:
\begin{itemize}
 \item the additive closure $\AC \mapsto \AC^{\oplus}$, see Subsection \ref{subsection:additive_closure}, which turns an Ab-category $\AC$ into an additive one,
 \item the Freyd category $\AC \mapsto \Freyd( \AC )$, see Subsection \ref{subsection:Freyd_category}, which equips an additive category $\AC$ with cokernels.
\end{itemize}

An iterative application of category constructors can lead to intriguing results.
Let $R$ be a commutative coherent ring,
$R\fpmodl$ the category of finitely presented $R$-modules,
and $\fp( R\fpmodl, \Ab )$ the category of all finitely presented functors $R\fpmodl \rightarrow \Ab$ (where $\Ab$ denotes the category of abelian groups),
i.e., functors that arise as cokernels of representable functors.
Triggering a cascade of category constructors yields an equivalence
\[
 \fp( R\fpmodl, \Ab ) \simeq \Freyd( \Freyd( \AsCat(R)^{\oplus} )^{\op} ).
\]
Thus, knowing how to compute homomorphism sets within $\Freyd( \Freyd( \AsCat(R)^{\oplus} )^{\op} )$ 
allows us to compute homomorphism sets between finitely presented functors.

In order to carry out this plan, we start at the lowest level $\AsCat(R)$ of this cascade
and analyze how algorithms at the current level give rise to algorithms on the next level until we end up with algorithms for dealing with the top level.

\subsection{Computable categories}

As a very first step we need to introduce categories from a constructive point of view.
We will see that it is worthwhile to pay special attention to the
classically trivial notion of \emph{equality of morphisms}.

\begin{definition}\label{definition:computable_category}
 A \textbf{category} $\AC$ consists of the following data:
 \begin{enumerate}
  \item A data type $\Obj_{\AC}$ 
        (\textbf{objects}).
  \item Depending on $A,B \in \Obj_{\AC}$, a data type $\Hom_{\AC}(A,B)$ (\textbf{morphisms}),
        each equipped with an equivalence relation $=$ (\textbf{equality}).
  \item An algorithm that computes for given $A,B,C \in \Obj_{\AC}$, $\alpha \in \Hom_{\AC}(A,B)$, and $\beta \in \Hom_{\AC}(B,C)$
        a morphism $\alpha \cdot \beta \in \Hom_{\AC}(A,C)$
        (\textbf{composition}).
        For $D \in \Obj_{\AC}$ and $\gamma \in \Hom_{\AC}(C,D)$, we require
        \[
         (\alpha \cdot \beta)\cdot \gamma = \alpha \cdot ( \beta \cdot \gamma ) \text{~(\textbf{associativity})}.
        \]
  \item An algorithm that constructs for given $A \in \Obj_{\AC}$ a morphism $\id_A \in \Hom_{\AC}(A,A)$
        (\textbf{identities}). For $B, C \in \Obj_{\AC}$, $\beta \in \Hom_{\AC}(B,A)$, $\gamma \in \Hom_{\AC}(A,C)$,
        we require 
        \[
         \beta \cdot \id_A = \beta \text{\hspace{1em} and \hspace{1em}} \id_A \cdot \gamma = \gamma.
        \]
  \end{enumerate}
\end{definition}

We give several examples of categories
that will quickly lead us into the realm of computationally undecidable problems.

\begin{example}\label{example:monoid_as_category}
 Every monoid $(M, 1, \cdot)$ gives rise to a category $\AsCat( M )$,
 consisting of a single object $\ast$,
 whose morphisms $\ast \stackrel{m}{\longrightarrow} \ast$
 are given by the elements $m \in M$.
 Composition is induced by multiplication in $M$,
 the identity is given by $1 \in M$.
 Equality of morphisms is simply equality of elements.
\end{example}

\begin{example}\label{example:free_monoid_as_category}
 Let $\Sigma$ be a finite alphabet, say $\Sigma \coloneqq \{ a, b, c, d, e \}$.
 All words built up from $\Sigma$,
 i.e., the elements of the free monoid $\Free( \Sigma )$ on $\Sigma$,
 together with concatenation of words form the morphisms of the single object category $\AsCat( \Free( \Sigma ) )$
 with the empty word as the identity.
 In this example, equality of morphisms is given by comparing words letter by letter.
\end{example}

\begin{example}\label{example:undecidable_monoid}
 We may alter the notion of equality in the previous Example \ref{example:free_monoid_as_category}
 without altering the other defining data.
 For a given finite set $R \subseteq \Free( \Sigma ) \times \Free( \Sigma )$, we may choose
 the equality in our category as the monoid equivalence relation generated by $R$,
 i.e., the smallest equivalence relation containing $R$ that is also a submonoid of $\Free( \Sigma ) \times \Free( \Sigma )$.
 For example, we could choose the monoid equivalence relation generated by
 \begin{align*}
  ac &= ca \\
  bc &= cb \\
  ce &= eca \\
  ad &= da \\
  bd &= db \\
  de &= edb \\
  cca &= ccae.
 \end{align*}
 We can still perform compositions as in $\AsCat( \Free( \Sigma ) )$,
 but the question of deciding whether two morphisms are equal
 w.r.t.\ the concrete monoid equivalence relation above
 is computationally unsolvable \cite{CollinsWordProblem}.
\end{example}

The previous example highlights the enormous importance of equality in a constructive setup
and motivates the next definition which singles out those categories for which the
classically trivial proposition
\[
 \forall \alpha, \beta \in \Hom_{\AC}(A,B): \alpha = \beta  \vee \alpha \neq \beta
\]
can be realized algorithmically.

\begin{definition}
 A category $\AC$ is called \textbf{computable}
 if we have an algorithm that decides for given $A,B \in \Obj_{\AC}$, $\alpha, \beta \in \Hom_{\AC}(A,B)$
 whether $\alpha = \beta$ or $\alpha \neq \beta$.
\end{definition}

\begin{example}
The category associated to the free monoid as described in Example \ref{example:free_monoid_as_category}
is computable if we can decide equality of elements in the given alphabet $\Sigma$. 
\end{example}

The following example generalizes the example of a free monoid
to ``a free monoid with multiple objects'', i.e., a free category.

\begin{example}[Free categories]\label{example:free_categories}
 A quiver $Q$ is a directed graph (with finitely many vertices and edges) that is allowed to contain loops and multiple edges.
 Edges $\alpha$ in $Q$ are usually called arrows
 and are depicted by
 \[
  a \stackrel{\alpha}{\longrightarrow} b,
 \]
 and we set $\Source( \alpha ) \coloneqq a$ and $\Range( \alpha ) \coloneqq b$.
 Arrows $\alpha$ and $\beta$ are called composable if
 \[
  \Range( \alpha ) = \Source( \beta ).
 \]
 Given two nodes $a,b \in Q$, 
 a path from $a$ to $b$ is a finite sequence
 of arrows $\alpha_1, \dots, \alpha_n$ such that 
 any two consecutive arrows are composable,
 and $\Source( \alpha_1 ) = a$, $\Range( \alpha_n ) = b$.
 If $a = b$, we allow $n = 0$ and call it the empty path.
 
 Taking the set of nodes in $Q$ as our objects,
 and taking paths from $a$ to $b$ as our morphisms,
 we can create a category whose composition is given by concatenation of paths.
 The empty paths now come in handy as identities.
 Equality for morphisms is given by comparing two paths arrow-wise.
 If we start with the quiver that contains a single node denoted by $\ast$
 and five loops
 \begin{center}
  \begin{tikzpicture}[transform shape,mylabel/.style={thick, align=center, minimum width=0.5cm, minimum height=0.5cm,fill=white}]
          \coordinate (r) at (8,0);
          \coordinate (rr) at (1,0);
          \coordinate (d) at (0,-1);
          \node(base) {$\ast$};
          \draw[->,thick] (base) to [out=90+20,in=90-20,looseness=15] node[above] {$a$} (base);
          \draw[->,thick] (base) to [out=90+20+72,in=90-20+72,looseness=15] node[above] {$b$} (base);
          \draw[->,thick] (base) to [out=90+20+144,in=90-20+144,looseness=15] node[left] {$c$} (base);
          \draw[->,thick] (base) to [out=90+20+216,in=90-20+216,looseness=15] node[right,xshift=-0.1em] {$d$} (base);
          \draw[->,thick] (base) to [out=90+20+288,in=90-20+288,looseness=15] node[above] {$e$} (base);
  \end{tikzpicture}
\end{center}
then this category recovers Example \ref{example:free_monoid_as_category}.
Free categories are computable whenever we can decide equality of arrows.
\end{example}

\subsection{Ab-categories}

We are mostly interested in categories that admit additional structure.
An Ab-category is a category $\AC$ for which
all homomorphism data types come equipped with the structure of abelian groups
such that this structure is compatible with the composition in $\AC$.
We spell this out explicitly.

\begin{definition}
  An \textbf{Ab-category} is a category $\AC$ for which we have:
  \begin{enumerate}
  \item An algorithm that computes for given $A,B \in \Obj_{\AC}$, $\alpha, \beta \in \Hom_{\AC}(A,B)$
        a morphism $\alpha + \beta \in \Hom_{\AC}(A,B)$ (\textbf{addition}).
  \item An algorithm that constructs for given $A,B \in \Obj_{\AC}$ a morphism $0 \in \Hom_{\AC}(A,B)$
        (\textbf{zero morphism}).
  \item An algorithm that computes for given $A,B \in \Obj_{\AC}$, $\alpha \in \Hom_{\AC}(A,B)$ 
        a morphism $-\alpha \in \Hom_{\AC}(A,B)$ (\textbf{additive inverse}).
  \item For all $A,B \in \Obj_{\AC}$, the given data turn $\Hom_{\AC}(A,B)$ into an abelian group.
  \item Composition with morphisms both from left and right in $\AC$ becomes a bilinear map.
\end{enumerate}
\end{definition}

\begin{example}[Rings as categories]\label{example:ring_as_category}
 Analogous to Example \ref{example:monoid_as_category},
 every ring $R$ gives rise to an Ab-category $\AsCat(R)$, i.e.,
 we identify ring multiplication with composition\footnote{
   Note that we defined composition in a category as \emph{pre}composition,
   and not as \emph{post}composition. It is common to regard a ring $R$
   as a category with postcomposition being identified with ring multiplication,
   and also to use the symbol $R$ in order to refer to that category.
   Thus, our category $\AsCat(R)$ equals the category $R^{\op}$.
 }.
 
 Analogous to Example \ref{example:undecidable_monoid},
 every two-sided ideal $I \subseteq R$
 lets us alter the notion of equality in this category by considering
 two morphisms as equal if and only if they are equal as elements in the quotient ring $R/I$.
 We denote this category by $\AsCat(R, I)$ and remark that it is equivalent to $\AsCat( R/I )$.
\end{example}

\begin{example}
Let $k$ be a field.
Let us consider the Ab-category $\AsCat(R, I)$
associated to the commutative polynomial ring
$R \coloneqq k[x_1, \dots, x_n]$ for $n \in \N$
and an ideal $I \subseteq R$ generated by finitely many given elements.
It is a great triumph of computer algebra that this category is indeed computable provided we can decide equality in $k$.
The decidability test for equality in $R/I$ can then be executed using the theory of Gröbner bases.
For a detailed account on Gröbner bases, see, e.g., \cite{CLO}.
\end{example}

To a graded ring, we can attach an Ab-category
having more than a single object.

\begin{example}[Graded rings as categories]\label{example:graded_ring_as_category}
 A $\Z$-graded ring is a
 ring $R$ together with a decomposition
 \[
  R = \bigoplus_{d \in \Z} R_d
 \]
 into a direct sum of abelian groups
 such that $R_d \cdot R_e \subseteq R_{d+e}$ for all $d,e \in \Z$.
 We can form an Ab-category out of these data as follows:
 \begin{enumerate}
  \item Objects are given by $\Z$.
  \item For $a,b \in \Z$, homomorphisms from $a$ to $b$ are given by elements in $R_{b-a}$.
  \item Composition is multiplication in $R$.
 \end{enumerate}
 We denote this category by $\AsCat( \bigoplus_{d \in \Z} R_d )$.
 For a graded ideal $I \subseteq R$,
 we alter the notion of equality analogously to Example \ref{example:ring_as_category} 
 in order to obtain an Ab-category $\AsCat( \bigoplus_{d \in \Z} R_d, I )$.
\end{example}

\begin{example}[Path algebras as categories]\label{example:path_algebra_as_category}
 We can linearize Example \ref{example:free_categories} as follows.
 Let $k$ be a commutative ring.
 Again, we take as objects the nodes in our quiver $Q$,
 but now, we allow as morphisms from a node $a$ to a node $b$
 any formal $k$-linear combination of paths from $a$ to $b$.
 Extending concatenation of paths $k$-bilinearly,
 we obtain in this way an Ab-category that we denote by $\AsCat( k, Q )$.
 
 Note that the morphisms from $a$ to $b$ in $\AsCat( k, Q )$ identify with 
 all elements in the path algebra $k[Q]$ that start at $a$ and end at $b$.
 In this sense, our category only stores the uniform elements of $k[Q]$.
 Ideals $I \subseteq k[Q]$ generated by uniform elements let us alter the notion of equality
 analogously to Example \ref{example:ring_as_category}, and we denote the corresponding category
 by $\AsCat( k, Q, I )$.
\end{example}

\subsection{Additive closure}\label{subsection:additive_closure}

In this subsection, we want to introduce a categorical concept that
grasps the idea of forming matrices whose entries consist of morphisms
in an underlying Ab-category.

\begin{definition}
  An \textbf{additive category} is an Ab-category $\AC$ for which we have:
  \begin{theoremenumerate}
  \setcounter{enumi}{9}
  \item An algorithm that computes for a given finite (possibly empty) list of objects $A_1, \dots, A_n$ in $\Obj_{\AC}$ (for $n \in \N_0$)
        an object $\bigoplus_{i=1}^n A_i \in \Obj_{\AC}$ (\textbf{direct sum}).
        If we are additionally given an integer $j \in \{ 1 \dots n\}$,
        we furthermore have algorithms for computing 
        morphisms $\pi_j \in \Hom_{\AC}(\bigoplus_{i=1}^n A_i, A_j)$ (\textbf{direct sum projection})
        and $\iota_j: \Hom_{\AC}(A_j, \bigoplus_{i=1}^n A_i)$ (\textbf{direct sum injection}).
  \item\label{definition_item:add_identities} The identities
        \begin{itemize}
         \item $\sum_{i=1}^n \pi_i \cdot \iota_i = \id_{\bigoplus_{i=1}^n A_i}$,
         \item $\iota_i \cdot \pi_i = \id_{A_i}$,
         \item $\iota_i \cdot \pi_j = 0$,
        \end{itemize}
        hold for all $i,j = 1, \dots, n$, $i \not= j$.
 \end{theoremenumerate}
\end{definition}
Direct sums in $\AC$ and matrices having morphisms in $\AC$ as its entries are closely linked as follows:
given a morphism
\[
 \bigoplus_{i=1}^m A_i \stackrel{\alpha}{\longrightarrow} \bigoplus_{j=1}^n B_j
\]
between direct sums in $\AC$, we can form the matrix of morphisms
\[
 \big( A_s \stackrel{\iota_s}{\longrightarrow} \bigoplus_{i=1}^m A_i \stackrel{\alpha}{\longrightarrow} \bigoplus_{j=1}^n B_j \stackrel{\pi_t}{\longrightarrow} B_t \big)_{st}.
\]
Conversely, any matrix of morphisms
\[
 \big( A_s \stackrel{\alpha_{st}}{\longrightarrow} B_t \big)_{ \substack{s = 1, \dots, m\\t = 1, \dots, n}}
\]
defines a morphism
\[
 \sum_{\substack{s = 1, \dots, m\\t = 1, \dots, n}}(\pi_s \cdot \alpha \cdot \iota_t): \bigoplus_{i=1}^m A_i {\longrightarrow} \bigoplus_{j=1}^n B_j.
\]
Both constructions are mutually inverse thanks to the equational identities \ref{definition_item:add_identities} that hold for direct sums.

If an Ab-category $\AC$ does not yet admit direct sums,
it is easy to construct its additive closure
by employing exactly the philosophy of thinking of morphisms
between direct sums as matrices.
We will now show this construction as an example of a \emph{category constructor}.

\begin{construction}
 Let $\AC$ be an Ab-category.
 We construct its \textbf{additive closure} $\AC^{\oplus}$
 as follows:
 an object in $\AC^{\oplus}$ is given by an integer $m \geq 0$ and a list
 \[
  (A_1, \dots, A_m)
 \]
 of objects $A_i \in \AC$ for $i = 1, \dots, m$.
 We think of this list as formally representing the object
 \[
  \bigoplus_{i = 1}^{m} A_i.
 \]
 A morphism from one such list
 $(A_1, \dots, A_m)$ to another $(B_1, \dots, B_n)$ is given by a matrix
 \[
\begin{pmatrix}
\alpha_{11} & \dots & \alpha_{1n} \\
\vdots & \ddots & \vdots \\
\alpha_{m1} & \dots & \alpha_{mn} \\
\end{pmatrix}
 \]
 consisting of morphisms $\alpha_{ij}: A_i \longrightarrow B_j$ in $\AC$.
 Now, composition can be defined by the usual formula for matrix multiplication,
 and matrices with identity morphisms on the diagonal and zero morphisms off-diagonal
 serve as identities in this category.
 Equality for morphisms is checked entrywise.
\end{construction}

\begin{remark}
 $\AC^{\oplus}$ is computable if and only if $\AC$ is.
\end{remark}

It is quite easy to check that $\AC^{\oplus}$ is indeed an additive category.
Futhermore, we can always view $\AC$ as a subcategory of $\AC^{\oplus}$
by identifying an object in $A \in \AC$ as a list with a single element $(A)$.
The empty list $()$ defines a zero object in $\AC^{\oplus}$, i.e., an object whose identity morphism
equals the zero morphism.

\begin{example}\label{example:rows_k}
 If $k$ is a field, then 
 the objects in $\AsCat(k)^{\oplus}$ (see Example \ref{example:ring_as_category})
 are simply given by natural numbers $\N_0$,
 and a morphism from $m \in \N_0$ to $n \in \N_0$
 is an $m \times n$ matrix with entries in $k$.
 
 The map $m \mapsto k^{1 \times m}$ and the identification of elements in $k^{m \times n}$
 with $k$-linear maps $k^{1 \times m} \longrightarrow k^{1 \times n}$
 gives rise to an equivalence of categories 
 between $\AsCat(k)^{\oplus}$ and the category of all finite dimensional 
 $k$-vector spaces.
 We set 
 \[
  \SVec_{k} \coloneqq \AsCat(k)^{\oplus},
 \]
 since we think of the objects $n \in \N_0$ as the vector spaces $k^{1 \times n}$ of rows.
 From a computational point of view, $\SVec_{k}$ often serves as a workhorse:
 due to the power of Gaussian elimination, whenever we can reduce a problem in another category
 to linear algebra, we can try and solve it within $\SVec_{k}$.
\end{example}

\begin{example}\label{example:rows_R}
 More generally, if $R$ is a ring,
 then objects in $\AsCat(R)^{\oplus}$
 identify with row modules $R^{1 \times n}$ for $n \in \N_0$,
 and every $R$-module homomorphism $R^{1 \times m} \longrightarrow R^{1 \times n}$ is given by a matrix in $R^{m \times n}$.
 But since not every $R$-module is free in general,
 $\AsCat(R)^{\oplus}$ is only equivalent to a subcategory of the category of all finitely generated $R$-modules.
 We set 
 \[
 \Rows_R \coloneqq \AsCat(R)^{\oplus}.
 \]
\end{example}

If $\AC$ has more than just a single object, then compositionality of morphisms
in $\AC^{\oplus}$ relies on more than just matching numbers of columns and rows.

\begin{example}\label{example:additive_closure_graded}
 If we take the additive closure of the category $\AsCat( \bigoplus_{d\in\Z}R_d )$
 introduced in Example \ref{example:graded_ring_as_category}, then
 we get a category whose objects
 can be seen as finite lists of integers.
 A morphism from such a list $(m_1, \dots, m_s)$ to another list $(n_1, \dots, n_t)$ with $s,t \in \N_0$
 is given by a matrix 
 \[(a_{ij})_{\substack{i = 1, \dots, s \\ j = 1, \dots, t } }\]
 with homogeneous entries
 in $R$ whose degrees satisfy
 \begin{equation}\label{eq:degrees}
\deg( a_{ij} ) + m_i = n_j  
 \end{equation}
 whenever $a_{ij} \neq 0$.
 
 As an example, let $k$ be a field and $R = k[x,y]$ be the $\Z$-graded polynomial ring with $\deg(x) = \deg(y) = 1$.
 Then
 \begin{center}
         \begin{tikzpicture}[label/.style={postaction={
          decorate,
          decoration={markings, mark=at position .5 with \node #1;}},
          mylabel/.style={thick, draw=none, align=center, minimum width=0.5cm, minimum height=0.5cm,fill=white}}]
          \coordinate (r) at (2.5,0);
          \coordinate (u) at (0,2);
          \node (A) {$(0,1)$};
          \node (B) at ($(A)+(r)$) {$(2)$};
          \draw[->,thick] (A) --node[above]
          {$\pmatcol{xy}{x+y}$} 
          (B);
          \end{tikzpicture}
\end{center}
is an example of a morphism in $\AsCat( \bigoplus_{d\in\Z}R_d )$.
Note that the matrix alone does not determine the source and range of this morphism, since, for example 
 \begin{center}
         \begin{tikzpicture}[label/.style={postaction={
          decorate,
          decoration={markings, mark=at position .5 with \node #1;}},
          mylabel/.style={thick, draw=none, align=center, minimum width=0.5cm, minimum height=0.5cm,fill=white}}]
          \coordinate (r) at (2.5,0);
          \coordinate (u) at (0,2);
          \node (A) {$(-1,0)$};
          \node (B) at ($(A)+(r)$) {$(1)$};
          \draw[->,thick] (A) --node[above]
          {$\pmatcol{xy}{x+y}$} 
          (B);
          \end{tikzpicture}
\end{center}
is also a valid example of a morphism.
If we fix the matrix and the source/range in the first example
and forget its range/source
 \begin{center}
         \begin{tikzpicture}[label/.style={postaction={
          decorate,
          decoration={markings, mark=at position .5 with \node #1;}},
          mylabel/.style={thick, draw=none, align=center, minimum width=0.5cm, minimum height=0.5cm,fill=white}}]
          \coordinate (r) at (2.5,0);
          \coordinate (u) at (0,2);
          \node (A) {$(0,1)$};
          \node (B) at ($(A)+(r)$) {$(?)$};
          \node (A2) at ($(A)+3*(r)$) {$(?,?)$};
          \node (B2) at ($(A2)+(r)$) {$(2)$};
          \draw[->,thick] (A) --node[above]
          {$\pmatcol{xy}{x+y}$} 
          (B);
          
          \draw[->,thick] (A2) --node[above]
          {$\pmatcol{xy}{x+y}$} 
          (B2);
    \end{tikzpicture}
\end{center}
then Equation \ref{eq:degrees} makes it possible to reconstruct the missing information.
However, such a reconstruction is not possible in general: the $s \times t$ zero matrix
defines a valid morphism between any two objects $(m_1, \dots, m_s)$ and $(n_1, \dots, n_t)$.
\end{example}

\begin{example}\label{example:additive_closure_k_Q}
 Similarly, taking the additive closure of the category $\AsCat( k, Q )$
 introduced in Example \ref{example:path_algebra_as_category},
 we get a category whose objects are finite lists of nodes in $Q$,
 and morphisms 
 from a list $(v_1, \dots, v_s)$ to $(w_1, \dots, w_s)$
 are matrices 
 \[(a_{ij})_{\substack{i = 1, \dots, s \\ j = 1, \dots, t } }\]
 whose entries consist of uniform elements in the path algebra $k[Q]$,
 where $a_{ij}$ is either zero or starts at $v_i$ and ends at $w_j$.
\end{example}

\subsection{Homomorphism structures}

The question of how to describe the homomorphisms between two objects ``as a whole''
is just as important as the decidability problem of equality for two individual morphisms.
Classically, one could restrict the attention to so-called locally small categories,
which are categories $\AC$ in which the members of the family $\Hom_{\AC}(A,B)$
can all be interpreted as objects in $\Set$, the category of sets.
This enables us to view $\Hom$ as a functor
\[
 \Hom: \AC^{\op} \times \AC \rightarrow \Set.
\]
For our constructive approach, we will simply generalize
this point of view and axiomatize those features that we need
from a $\Hom$-functor to make computational use of it.
But before we do this, we state the definition of a functor
within our constructive setup.

\begin{definition}\label{definition:functor}
 A \bfindex{functor} $F$ between two categories $\AC$ and $\BC$ consists of the following data:
 \begin{enumerate}
  \item An algorithm that computes for given $A \in \Obj_{\AC}$ an object  $F(A) \in \Obj_{\BC}$.
  \item An algorithm that computes for given $A,B \in \Obj_{\AC}$, $\alpha \in \HomC(A,B)$
        a morphism $F( \alpha ) \in \Hom_{\BC}(F(A),F(B))$.
        This algorithm needs to be compatible with the notion of equality for morphisms.
  \item For $A \in \ObjC$, $F( \id_A ) = \id_{F(A)}$.
  \item For $A,B,C \in \ObjC$, $\alpha \in \Hom_{\AC}( A, B )$, $\beta \in \Hom_{\AC}( B, C )$, we have
        \[F( \alpha \cdot \beta ) = F( \alpha ) \cdot F( \beta ).\]
 \end{enumerate}
\end{definition}

\begin{remark}
 Note that since we did not impose a notion of equality on
 the data type $\Obj_{\AC}$, it is not meaningful to declare
 the operation of $F$ on objects to be compatible with equality
 like we did in the case of morphisms.
\end{remark}

\begin{definition}\label{definition:homomorphism_strucutre}
 Let $\AC$, $\BC$ be categories.
 A \textbf{$\BC$-homomorphism structure} for $\AC$ consists of the following data:
 \begin{enumerate}
  \item An object $1 \in \BC$ called the \textbf{distinguished object}.
  \item A functor $H: \AC^{\op} \times \AC \rightarrow \BC$.
  \item A bijection $\nu: \Hom_{\AC}( A,B ) \xrightarrow{\sim} \Hom_{\BC}( 1, H( A,B ) )$ natural in $A,B \in \AC$,
        i.e, 
        \[
        \nu( \alpha \cdot X \cdot \beta ) = \nu( X ) \cdot H(\alpha, \beta) 
        \]
        for all composable triples of
        morphisms $\alpha, X, \beta$.
 \end{enumerate}
 Moreover, if we are in the context of Ab-categories,
 we also impose the condition that $H$ is a bilinear functor, i.e., acts linearly on morphisms in each component.
\end{definition}

\begin{example}\label{example:hom_structure_for_svec}
 Let $k$ be a field. We are going to describe
 a homomorphism structure for $\Rows_k$ (see Example \ref{example:rows_k})
 that is inspired by the fact that
 $\Rows_k$ is equivalent to the category of finite dimensional $k$-vector spaces and that
 linear maps between two given finite dimensional vector spaces form themselves a finite dimensional vector space.
 
 In the language of homomorphism structures, we can construct a $\SVec_k$-homomorphism structure for $\SVec_k$.
 We define a functor $H$ on objects (which are simply elements in $\N_0$) by multiplication of natural numbers,
 and on morphisms (which are matrices) by
 \[
  H( \alpha, \beta ) \coloneqq \alpha^{\tr} \otimes \beta,
 \]
 where $(-)^{\tr}$ is transposition and $\otimes$ denotes the Kronecker product.
 As a distinguished object, we take the natural number $1 \in \N_0$.
 Now, for given $m, n \in \N_0$,
 any morphism from $1$ to $mn$, i.e.,
 any row vector $(a_i)_{i = 1, \dots mn}$, can be interpreted
 as an $m \times n$ matrix by ``line-breaking'' after each $n$-entries.
 Conversely, every $m \times n$ matrix can be converted to such a row 
 by simply concatenating all rows.
 Thus, we have found a natural way to
 transfer ``vectors'' of $mn$, i.e., morphisms $1 \rightarrow mn$,
 into morphisms $m \rightarrow n$ in $\SVec_k$.
 
 So, note that it is not the object $mn \in \SVec_k$ alone that encodes $\Hom_{\SVec_k}(m,n)$,
 but it is the object $mn$ in the context of a homomorphisms structure that allows
 us to \emph{interpret} it as an encoding of homomorphisms from $m$ to $n$.
\end{example}

Next, we describe homomorphism structures for special cases of the Ab-categories
given in Examples \ref{example:ring_as_category}, \ref{example:graded_ring_as_category}, \ref{example:path_algebra_as_category}.

\begin{example}\label{example:hom_structure_for_R}
 Let $R$ be a \emph{commutative} ring.
 We can construct a $\AsCat(R)$-homomorphism structure for $\AsCat(R)$ (see Example \ref{example:ring_as_category}) as follows:
 the operation
 \[
  H: \AsCat(R)^{\op} \times \AsCat(R) \longrightarrow \AsCat(R): (\ast \stackrel{a}{\longleftarrow} \ast, \ast \stackrel{b}{\longrightarrow} \ast) \mapsto (\ast \stackrel{a \cdot b}{\longrightarrow} \ast)
 \]
 defines a bilinear functor
 due to the commutativity of $R$.
 For the distinguished object, we have no other choice but to take the unique object $\ast$ in $\AsCat(R)$.
 Finally, $\nu$ can be chosen as the identity on $\Hom_{\AsCat(R)}(\ast, \ast)$.
\end{example}

\begin{example}\label{example:hom_structure_for_graded_R}
 Let $k$ be a field and
 let $R$ be a $\Z$-graded $k$-algebra.
 If every $R_d$ is of finite $k$-dimension 
 with bases $\{r_d^1, \dots, r_d^{\dim_k(R_d)}\}$,
 then we may write for every $a,b,c \in \Z$ and $r \in R_{c}$, $s \in R_{b-(a+c)}$
 the $k$-linear operator
 \[
  R_a \longrightarrow R_{b}: x \mapsto r \cdot x \cdot s
 \]
 in terms of the given bases in order to obtain matrices $M_{a,b,r,s}$.
 This enables us to describe for $\AsCat( \bigoplus_{d \in \Z} R_d )$ (see Example \ref{example:graded_ring_as_category}) a $\SVec_k$-homomorphism structure
 with
 \[
  H( a, b ) \coloneqq \dim_k( R_{b-a} )
 \]
 for $a,b \in \Z$,
 and for $a',b' \in \Z$, $r \in R_{a-a'}$, $s \in R_{b'-b}$,
 \[
  H( a \stackrel{r}{\longleftarrow} a', b \stackrel{s}{\longrightarrow} b' ) \coloneqq M_{(b-a), (b'-a'), r, s}.
 \]
 The distinguished object is $1 \in \Rows_k$, and $\nu_{a,b}$ computes for an element $r \in R_{b-a}$ its list of coefficients w.r.t.\ the basis $\{r_{b-a}^1, \dots, r_{b-a}^{\dim_k(R_{b-a})}\}$.
 
 Moreover, if $R$ is commutative (but the $R_d$ not necessarily finite dimensional), 
 we could also construct a different homomorphism structure for $\AsCat( \bigoplus_{d \in \Z} R_d )$,
 namely a $\AsCat( \bigoplus_{d \in \Z} R_d )$-homomorphism structure
 with
 \[
  H( a, b ) \coloneqq b - a
 \]
 for $a,b \in \Z$ and for $a',b' \in \Z$, $r \in R_{a-a'}$, $s \in R_{b'-b}$,
 \[
  H( a \stackrel{r}{\longleftarrow} a', b \stackrel{s}{\longrightarrow} b' ) \coloneqq (b-a) \stackrel{r \cdot s}{\longrightarrow} (b'-a').
 \]
 This time, the distinguished object is $0 \in \AsCat( \bigoplus_{d \in \Z} R_d )$,
 and $\nu$ given by the identity
 \[
  \Hom_{\AsCat( \bigoplus_{d \in \Z} R_d )}( a, b ) = R_{b-a} = \Hom_{\AsCat( \bigoplus_{d \in \Z} R_d )}( 0, b - a ).
 \]
 So, we see that it is neither necessarily the case that $\BC$ is equivalent to $\AC$,
 nor that there is only a single homomorphism structure for a given category $\AC$.
\end{example}

\begin{example}\label{example:hom_structure_for_quiver}
 Let $k$ be a field and $Q$ be a quiver.
 If $Q$ is acyclic, then 
 the homomorphisms in $\AsCat(k,Q)$ from a vertex $v$ to a vertex $w$ form a finite dimensional
 $k$-vector space.
 Similarly to Example \ref{example:hom_structure_for_graded_R},
 this allows us to create an $\SVec_k$-homomorphism structure for
 $\AsCat( k, Q )$ with
 \[
  H( v, w ) \coloneqq \text{number of paths from $v$ to $w$.}
 \]
\end{example}

It is natural to ask how a structure that we have given to a category
may transfer to a category obtained by a category constructor.
We can indeed transfer homomorphism structures to the additive closure.

\begin{construction}\label{Construction:lift_of_hom_structure_to_additive_closure}
 Let $\AC$ be an Ab-category and $\BC$ be an additive category.
 Let furthermore $(H,1,\nu)$ be a $\BC$-homomorphism structure for $\AC$.
 Then we can extend $(H,1,\nu)$ to a $\BC$-homomorphisms structure $(H^{\oplus},1,\nu^{\oplus})$ for $\AC^{\oplus}$ by extending bilinearly
 \[
H^{\oplus}\left( (B_j)_j \stackrel{(\alpha_{ij})_{ij}}{\longleftarrow} (A_i)_i, (C_s)_s \stackrel{(\beta_{st})_{st}}{\longrightarrow} (D_t)_t \right)
\coloneqq
\bigoplus_{j,s}H(B_j,C_s) \stackrel{\big(H(\alpha_{ij},\beta_{st})\big)_{(js)(it)}}{\longrightarrow} \bigoplus_{i,t}H(A_i,D_t).
 \]
 The natural isomorphism $\nu^{\oplus}$ is defined via the composition of natural isomorphisms
 \begin{align*}
  \Hom_{\AC^{\oplus}}( (B_j)_j, (C_s)_s ) &\simeq \bigoplus_{j,s}\Hom_{\AC^{\oplus}}( B_j,C_s ) \\
  &\simeq \bigoplus_{j,s}\Hom_{\BC}( 1, H(B_j,C_s) ) \\
  &\simeq \Hom_{\BC}\big( 1, H^{\oplus}((B_j)_j,(C_s)_s) \big). \\
 \end{align*}
\end{construction}

\begin{remark}\label{remark:Ab_to_additive_for_hom_structure}
 We can also use Construction \ref{Construction:lift_of_hom_structure_to_additive_closure}
 in the case when $\BC$ is an Ab-category that is not necessarily additive
 by first applying the full embedding $\BC \hookrightarrow \BC^{\oplus}$
 in order to obtain a $\BC^{\oplus}$-homomorphism structure for $\AC$,
 and then proceed as described.
\end{remark}

\begin{example}
 Let $k$ be a field.
 Let $H$ denote the $\AsCat(k)$-homomorphism structure of $\AsCat(k)$
 described in Example \ref{example:hom_structure_for_R}.
 Applying Construction \ref{Construction:lift_of_hom_structure_to_additive_closure}
 to $H$ (via Remark \ref{remark:Ab_to_additive_for_hom_structure})
 yields exactly the $\SVec_k$-homomorphism structure of $\SVec_k = \AsCat(k)^{\oplus}$
 that we described in Example \ref{example:hom_structure_for_svec}.
\end{example}

\subsection{Freyd category}\label{subsection:Freyd_category}

In this subsection, we introduce a further category constructor: the Freyd category \cite{FreydRep, BelFredCats}.
Freyd categories provide a unified approach to categories of finitely presented modules,
finitely presented graded modules, and finitely presented functors.

Let $R$ be a ring. 
Recall that a (left) $R$-module $M$ is called \textbf{finitely presented} if there exist
$a,b \in \N_0$ and an exact sequence
\begin{center}
         \begin{tikzpicture}[label/.style={postaction={
          decorate,
          decoration={markings, mark=at position .5 with \node #1;}},
          mylabel/.style={thick, draw=none, align=center, minimum width=0.5cm, minimum height=0.5cm,fill=white}}]
          \coordinate (r) at (2.5,0);
          \coordinate (u) at (0,2);
          \node (A) {$R^{1 \times b}$,};
          \node (B) at ($(A)-(r)$) {$R^{1 \times a}$};
          \node (C) at ($(B)-(r)$) {$M$};
          \node (D) at ($(C)-(r)$) {$0$};
          \draw[->,thick] (A) --node[above]{$\rho_M$} (B);
          \draw[->,thick] (B) -- (C);
          \draw[->,thick] (C) -- (D);
          \end{tikzpicture}
\end{center}
which is called a \textbf{presentation} of $M$.
Since $\rho_M$ is induced by a matrix with rows $r_1, \dots, r_b \in R^{1 \times a}$,
being finitely presented means nothing but the existence of an isomorphism
\[
 M \simeq R^{1 \times a}/ \langle r_1, \dots, r_b \rangle.
\]
Thus, we may think of a presentation as a way to store finitely
many relations $r_1, \dots, r_b$ that we would like to impose on an free module $R^{1 \times a}$.
Let $N$ be another finitely presented module with presentation
$\rho_N: R^{1 \times b'} \rightarrow R^{1 \times a'}$.
By the comparison theorem \cite{weihom}, we can lift any morphism $\mu: M \rightarrow N$
to a commutative diagram
\begin{center}
         \begin{tikzpicture}[label/.style={postaction={
          decorate,
          decoration={markings, mark=at position .5 with \node #1;}},
          mylabel/.style={thick, draw=none, align=center, minimum width=0.5cm, minimum height=0.5cm,fill=white}}]
          \coordinate (r) at (2.5,0);
          \coordinate (u) at (0,2);
          \node (A) {$R^{1 \times b}$};
          \node (B) at ($(A)-(r)$) {$R^{1 \times a}$};
          \node (C) at ($(B)-(r)$) {$M$};
          \node (D) at ($(C)-(r)$) {$0$};
          
          \node (A2) at ($(A) - (u)$) {$R^{1 \times b'}$};
          \node (B2) at ($(A2)-(r)$) {$R^{1 \times a'}$};
          \node (C2) at ($(B2)-(r)$) {$N$};
          \node (D2) at ($(C2)-(r)$) {$0$};
          \draw[->,thick] (A) --node[above]{$\rho_M$} (B);
          \draw[->,thick] (B) -- (C);
          \draw[->,thick] (C) -- (D);
          
          \draw[->,thick] (A2) --node[above]{$\rho_N$} (B2);
          \draw[->,thick] (B2) -- (C2);
          \draw[->,thick] (C2) -- (D2);
          
          \draw[->,thick] (A) -- (A2);
          \draw[->,thick] (B) -- (B2);
          \draw[->,thick] (C) --node[left]{$\mu$} (C2);
          \end{tikzpicture}
\end{center}
and conversely, any commutative diagram
\begin{center}
         \begin{tikzpicture}[label/.style={postaction={
          decorate,
          decoration={markings, mark=at position .5 with \node #1;}},
          mylabel/.style={thick, draw=none, align=center, minimum width=0.5cm, minimum height=0.5cm,fill=white}}]
          \coordinate (r) at (2.5,0);
          \coordinate (u) at (0,2);
          \node (A) {$R^{1 \times b}$};
          \node (B) at ($(A)-(r)$) {$R^{1 \times a}$};
          
          \node (A2) at ($(A) - (u)$) {$R^{1 \times b'}$};
          \node (B2) at ($(A2)-(r)$) {$R^{1 \times a'}$};
          \draw[->,thick] (A) --node[above]{$\rho_M$} (B);
          
          \draw[->,thick] (A2) --node[above]{$\rho_N$} (B2);
          
          \draw[->,thick] (A) -- (A2);
          \draw[->,thick] (B) -- (B2);
          \end{tikzpicture}
\end{center}
induces a morphism $\mu: M \rightarrow N$.
Moreover, such a $\mu$ is zero if and only if
we have a commutative diagram with exact rows
\begin{center}
         \begin{tikzpicture}[label/.style={postaction={
          decorate,
          decoration={markings, mark=at position .5 with \node #1;}},
          mylabel/.style={thick, draw=none, align=center, minimum width=0.5cm, minimum height=0.5cm,fill=white}}]
          \coordinate (r) at (2.5,0);
          \coordinate (u) at (0,2);
          \node (A) {$R^{1 \times b}$};
          \node (B) at ($(A)-(r)$) {$R^{1 \times a}$};
          \node (C) at ($(B)-(r)$) {$M$};
          \node (D) at ($(C)-(r)$) {$0$};
          
          \node (A2) at ($(A) - (u)$) {$R^{1 \times b'}$.};
          \node (B2) at ($(A2)-(r)$) {$R^{1 \times a'}$};
          \node (C2) at ($(B2)-(r)$) {$N$};
          \node (D2) at ($(C2)-(r)$) {$0$};
          \draw[->,thick] (A) --node[above]{$\rho_M$} (B);
          \draw[->,thick] (B) -- (C);
          \draw[->,thick] (C) -- (D);
          
          \draw[->,thick] (A2) --node[above]{$\rho_N$} (B2);
          \draw[->,thick] (B2) -- (C2);
          \draw[->,thick] (C2) -- (D2);
          
          \draw[->,thick,dashed] (B) -- (A2);
          \draw[->,thick] (B) -- (B2);
          \draw[->,thick] (C) --node[left]{$\mu$} (C2);
          \end{tikzpicture}
\end{center}

It follows that computing with finitely presented modules and their homomorphisms
can be replaced by computing with presentations (which are nothing but morphisms in the additive category $\Rows_R$,
see Example \ref{example:rows_R}),
and commutative squares involving presentations (which are simply commutative squares within $\Rows_R$)
considered up to an equivalence relation.
The concept of a Freyd category formalizes this calculus with $\Rows_R$
being replaced by an arbitrary additive category $\AC$.

\begin{construction}[Freyd categories]
Let $\AC$ be an additive category. We create
$\Freyd( \AC )$, the so-called \textbf{Freyd category} of $\AC$.
Its objects consist of morphisms
\[
 (A \stackrel{\rho_A}{\longleftarrow} R_A)
\]
in $\AC$. We think of such morphisms as formally representing
the cokernel of $\rho_A$. Note that neither $R_A$ nor $\rho_A$
do formally depend on $A$, however, we like to decorate these
objects with $A$ as an index and think of them as an encoding
for ``relations'' imposed on $A$.
A morphism between two objects in $\Freyd( \AC )$,
i.e., $(A \stackrel{\rho_A}{\longleftarrow} R_A)$ to $(B \stackrel{\rho_B}{\longleftarrow} R_B)$,
is given by a morphism
\[
 \alpha: A \longrightarrow B
\]
such that $\exists \rho_{\alpha}: R_A \longrightarrow R_B$
making the diagram 
\begin{center}
         \begin{tikzpicture}[label/.style={postaction={
          decorate,
          decoration={markings, mark=at position .5 with \node #1;}},
          mylabel/.style={thick, draw=none, align=center, minimum width=0.5cm, minimum height=0.5cm,fill=white}}]
          \coordinate (r) at (2.5,0);
          \coordinate (u) at (0,2);
          \node (A) {$A$};
          \node (B) at ($(A)+(r)$) {$R_A$};
          \node (C) at ($(A) - (u)$) {$B$};
          \node (D) at ($(B) - (u)$) {$R_B$};
          \draw[->,thick] (B) --node[above]{$\rho_A$} (A);
          \draw[->,thick] (D) --node[above]{$\rho_B$} (C);
          \draw[->,thick] (A) --node[left]{$\alpha$} (C);
          \draw[->,thick,dashed] (B) --node[right]{$\rho_{\alpha}$} (D);
          \end{tikzpicture}
\end{center}
commutative.
The equality of two morphisms 
$A \stackrel{\alpha}{\longrightarrow} B$, $A \stackrel{\alpha'}{\longrightarrow} B$ 
from $(A \stackrel{\rho_A}{\longleftarrow} R_A)$ to $(B \stackrel{\rho_B}{\longleftarrow} R_B)$
is defined by the existence of a $\lambda$ (called \textbf{witness for $\alpha$ and $\alpha'$ being equal}) 
rendering the diagram
\begin{center}
\begin{tikzpicture}[label/.style={postaction={
  decorate,
  decoration={markings, mark=at position .5 with \node #1;}},
  mylabel/.style={thick, draw=none, align=center, minimum width=0.5cm, minimum height=0.5cm,fill=white}}]
  \coordinate (r) at (3.5,0);
  \coordinate (u) at (0,2);
  \node (m) {$B$};
  \node (n) at ($(m)+(r)$) {$R_B$};
  \node (r1) at ($(m) + (u)$) {$A$};
  \draw[->,thick] (n) --node[above]{$\rho_B$} (m);
  \draw[->,thick,dashed] (r1) --node[above]{$\lambda$} (n);
  \draw[->,thick] (r1) --node[left]{$\alpha - \alpha'$} (m);
\end{tikzpicture}
\end{center}
commutative.
Composition and identity morphisms are inherited from $\AC$.
It is easy to check that the notion of equality for morphisms yields an equivalence relation compatible
with composition and identities.
\end{construction}

\begin{remark}
 Two commutative squares
 \begin{center}
  \begin{tabular}{ccc}
    \begin{tikzpicture}[label/.style={postaction={
            decorate,
            decoration={markings, mark=at position .5 with \node #1;}},
            mylabel/.style={thick, draw=none, align=center, minimum width=0.5cm, minimum height=0.5cm,fill=white}},
            baseline=(base)]
            \coordinate (r) at (2.5,0);
            \coordinate (u) at (0,2);
            \node (A) {$A$};
            \node (B) at ($(A)+(r)$) {$R_A$};
            \node (C) at ($(A) - (u)$) {$B$};
            \node (D) at ($(B) - (u)$) {$R_B$};
            \node (base) at ($(B) - 0.5*(u)$) {};
            \draw[->,thick] (B) --node[above]{$\rho_A$} (A);
            \draw[->,thick] (D) --node[above]{$\rho_B$} (C);
            \draw[->,thick] (A) --node[left]{$\alpha$} (C);
            \draw[->,thick,dashed] (B) --node[right]{$\rho_{\alpha}$} (D);
    \end{tikzpicture} 
          & 
          \hspace{1em}and \hspace{1em}
          &
    \begin{tikzpicture}[label/.style={postaction={
            decorate,
            decoration={markings, mark=at position .5 with \node #1;}},
            mylabel/.style={thick, draw=none, align=center, minimum width=0.5cm, minimum height=0.5cm,fill=white}},
            baseline=(base)]
            \coordinate (r) at (2.5,0);
            \coordinate (u) at (0,2);
            \node (A) {$A$};
            \node (B) at ($(A)+(r)$) {$R_A$};
            \node (C) at ($(A) - (u)$) {$B$};
            \node (D) at ($(B) - (u)$) {$R_B$};
            \node (base) at ($(B) - 0.5*(u)$) {};
            \draw[->,thick] (B) --node[above]{$\rho_A$} (A);
            \draw[->,thick] (D) --node[above]{$\rho_B$} (C);
            \draw[->,thick] (A) --node[left]{$\alpha$} (C);
            \draw[->,thick,dashed] (B) --node[right]{$\rho_{\alpha}'$} (D);
    \end{tikzpicture}  
  \end{tabular}
 \end{center}
are equal as morphisms in $\Freyd( \AC )$ with $0: A \longrightarrow R_B$ as a witness,
which is why we depict the arrows corresponding to $\rho_{\alpha}, \rho_{\alpha}'$ with a dashed line:
they merely need to exist, but do not otherwise contribute to the actual morphism.
\end{remark}

If $R\fpmodl$ denotes the category of finitely presented (left) $R$-modules,
then the discussion in the beginning of this subsection
can be summarized by the existence of an equivalence
\[
 R\fpmodl \simeq \Freyd( \Rows_R ).
\]
Note that the decisive feature of row modules $R^{1 \times a}$
that makes this equivalence work is their projectiveness as $R$-modules.
Thus, if we let $\Proj_R$ denote the full subcategory of the category of $R$-modules
spanned by all finitely presented projective modules,
and if $\AC$ is any full subcategory satisfying
\[
 \Rows_R \subseteq \AC \subseteq \Proj_R,
\]
we still have
\[
 R\fpmodl \simeq \Freyd( \AC ).
\]
If $k$ is a field and $Q$ a quiver,
then $\AsCat(k, Q)^{\oplus}$ (see Example \ref{example:additive_closure_k_Q})
identifies with the full additive subcategory of the category of modules over the path algebra $k[Q]$
generated by the projectives $k[Q]e_v$, where $e_v$ denotes the idempotent associated to the node $v \in Q$.
Since this subcategory contains $k[Q]$ and thus $\Rows_{k[Q]}$, we obtain an equivalence
\[
 k[Q]\fpmodl \simeq \Freyd( \AsCat(k, Q)^{\oplus} ).
\]
The discussion in this subsection neatly generalizes
to finitely presented graded modules.
If $R = \bigoplus_{d \in \Z}R_d$ is a $\Z$-graded ring,
then $\AsCat(\bigoplus_{d \in \Z}R_d)^{\oplus}$ (see Example \ref{example:additive_closure_graded})
identifies with the full additive subcategory of the category of graded $R$-modules 
generated by the shifts $R(d)$ for $d \in \Z$, i.e.,
by the graded modules with graded parts $R(d)_e \coloneqq R_{d+e}$
for all $e \in \Z$, and we again have an equivalence
\[
 R\fpgrmodl \simeq \Freyd( \AsCat(\bigoplus_{d \in \Z}R_d)^{\oplus} ),
\]
with $R\fpgrmodl$ denoting the category of finitely presented graded $R$-modules.

Thus, the abstract study of Freyd categories enables us to study all these computational
models of finitely presented modules in one go.

For an additive category $\AC$, let $\Hom( \AC^{\op}, \Ab )$ denote the category of contravariant additive functors
from $\AC$ into the category of abelian groups $\Ab$.
By Yoneda's lemma, the functor
\[
 \AC \longrightarrow \Hom( \AC^{\op}, \Ab ): A \mapsto (-,A)
\]
is full and faithful, where $(-,A)$ denotes the contravariant $\Hom$-functor. 
Thus, we can think of $\AC$ as the full subcategory of $\Hom( \AC^{\op}, \Ab )$
generated by all representable functors.
Again, by Yoneda's lemma, representable functors are projective objects in $\Hom( \AC^{\op}, \Ab )$,
and a straightforward generalization of the discussion in the beginning of this subsection shows
that we can identify $\Freyd( \AC )$ with the full subcategory of $\Hom( \AC^{\op}, \Ab )$
generated by so-called \emph{finitely presented functors}.
A functor $F: \AC^{\op} \longrightarrow \Ab$ is finitely presented if there exists $A, B \in \AC$
and $\alpha: A \rightarrow B$
and an exact sequence
\begin{center}
         \begin{tikzpicture}[label/.style={postaction={
          decorate,
          decoration={markings, mark=at position .5 with \node #1;}},
          mylabel/.style={thick, draw=none, align=center, minimum width=0.5cm, minimum height=0.5cm,fill=white}}]
          \coordinate (r) at (3.5,0);
          \coordinate (u) at (0,2);
          \node (A) {$(-,A)$};
          \node (B) at ($(A)-(r)$) {$(-,B)$};
          \node (C) at ($(B)-(r)$) {$F$};
          \node (D) at ($(C)-(r)$) {$0$};
          \draw[->,thick] (A) --node[above]{$(-,\alpha)$} (B);
          \draw[->,thick] (B) -- (C);
          \draw[->,thick] (C) -- (D);
          \end{tikzpicture}
\end{center}
in $\Hom( \AC^{\op}, \Ab )$, i.e., $F$ arises as the cokernel of a morphism between representable functors.
Analogously, one defines finitely presented covariant functors on $\AC$, and the category of all such functors
is equivalent to $\Freyd( \AC^{\op} )$.

\begin{example}\label{example:exts}
 If $\AC$ is an abelian category with enough projectives and $A \in \AC$, then 
 \[\Ext^i(A,-): \AC \rightarrow \Ab\]
 if finitely presented for all $i \geq 0$ \cite{A}.
 For example, in order to write $\Ext^1(A,-)$ as an object in $\Freyd( \AC^{\op} )$,
 take any short exact sequence
 \begin{center}
         \begin{tikzpicture}[label/.style={postaction={
          decorate,
          decoration={markings, mark=at position .5 with \node #1;}},
          mylabel/.style={thick, draw=none, align=center, minimum width=0.5cm, minimum height=0.5cm,fill=white}}]
          \coordinate (r) at (3,0);
          \coordinate (u) at (0,2);
          \node (A) {$0$};
          \node (B) at ($(A)-(r)$) {$\Omega^1(A)$};
          \node (C) at ($(B)-(r)$) {$P$};
          \node (D) at ($(C)-(r)$) {$A$};
          \node (E) at ($(D)-(r)$) {$0$};
          \draw[->,thick] (A) -- (B);
          \draw[->,thick] (B) -- (C);
          \draw[->,thick] (C) -- (D);
          \draw[->,thick] (D) -- (E);
          \end{tikzpicture}
\end{center}
 with $P$ projective. Then the morphism $(\Omega^1(A)\longrightarrow P)$
 considered as an object in $\Freyd( \AC^{\op} )$ corresponds to $\Ext^1(A,-)$.
 For higher $\Ext$s, we need to compute more steps of a projective resolution of $A$.
\end{example}

We have seen in this subsection that if we start with a ring $R$ and consider it as a single object category $\AsCat(R)$,
then we can apply a cascade of category constructors
\[
 \Freyd( \Freyd( \AsCat(R)^{\oplus} )^{\op} )
\]
and end up with a category equivalent to finitely presented functors on finitely presented modules over $R$.
Thus, the question of how to compute with finitely presented functors now reduces to the understanding of
how to compute with Freyd categories.

\subsection{Computing with Freyd categories}

We explain how to perform several explicit constructions within Freyd categories, like computing cokernels, kernels, lifts along mono\-morphisms,
and homomorphism structures.
For details about the correctness of these constructions, we refer the reader to \cite{PosFreyd}.

\subsubsection{Equality of morphisms}

Being computable for $\AC$ does by no means imply computability of $\Freyd( \AC )$.
We specify the decisive algorithmic feature of $\AC$ that turns $\Freyd( \AC )$ into a computable category.

\begin{definition}
 We say a category $\AC$ has \textbf{decidable lifts}
 if we have an algorithm that takes as an input a cospan
 \[
  A \stackrel{\alpha}{\longrightarrow} B \stackrel{\gamma}{\longleftarrow} C
 \]
 and either outputs a lift $\lambda: A \rightarrow C$ rendering the diagram
 \begin{center}
    \begin{tikzpicture}[label/.style={postaction={
        decorate,
        decoration={markings, mark=at position .5 with \node #1;}},
        mylabel/.style={thick, draw=none, align=center, minimum width=0.5cm, minimum height=0.5cm,fill=white}},
        baseline=(base)]
        \coordinate (r) at (2.5,0);
        \coordinate (u) at (0,2);
        \node (A) {$A$};
        \node (B) at ($(A)-(u)$) {$B$};
        \node (C) at ($(B) + (r)$) {$C$};
        \draw[->,thick] (A) --node[left]{$\alpha$} (B);
        \draw[->,thick] (C) --node[above]{$\gamma$} (B);
        \draw[->,thick,dashed] (A) --node[above,yshift=0.3em]{$\lambda$} (C);
      \end{tikzpicture}    
 \end{center}
 commutative, or disproves the existence of such a lift.
 \end{definition}

Clearly, whenever an additive category $\AC$ has decidable lifts,
we are able to decide equality in $\Freyd( \AC )$.
 
\begin{example}\label{example:lifts_and_gauss}
Let $k$ be a field with decidable equality of elements. 
Then, the category $\SVec_k$ has decidable lifts:
a cospan in $\SVec_k$ is nothing but a pair of matrices $\alpha, \gamma$ over $k$
having the same number of columns, and we can decide whether there
exists a matrix $\lambda$ over $k$ such that
$\lambda \cdot \gamma = \alpha$ using Gaussian elimination.
\end{example}

\begin{example}
 The following class of examples is vital for constructive algebraic geometry.
 Let $k$ be a field with decidable equality of elements.
 For 
 \[
 R = k[x_1, \dots, x_n]/I,
 \]
 where $I \subseteq k[x_1, \dots, x_n]$ is an ideal,
 Gröbner basis techniques imply that $\Rows_R$ has decidable lifts.
 Moreover, if
 $\mathfrak{p} \subseteq k[x_1, \dots, x_n]/I$ is a prime ideal,
 then for the localization
 \[
 R = (k[x_1, \dots, x_n]/I)_{\mathfrak{p}},
 \]
 $\Rows_R$ has decidable lifts.
 A general algorithm proving this fact can be found in \cite{PosLinSys}.
 Computing lifts in more specialized cases of such rings are treated for example in \cite{BL}
 or \cite{GP}.
\end{example}

We can employ homomorphism structures for making lifts decidable.

\begin{lemma}\label{lemma:lifts_via_hom}
 Let $\AC$ have a $\BC$-homomorphism structure $(H, \nu, 1 )$.
 For a given cospan $A \stackrel{\alpha}{\longrightarrow} B \stackrel{\gamma}{\longleftarrow} C$ in $\AC$,
 there exists a lift
 \begin{center}
    \begin{tikzpicture}[label/.style={postaction={
        decorate,
        decoration={markings, mark=at position .5 with \node #1;}},
        mylabel/.style={thick, draw=none, align=center, minimum width=0.5cm, minimum height=0.5cm,fill=white}},
        baseline=(base)]
        \coordinate (r) at (2.5,0);
        \coordinate (u) at (0,2);
        \node (A) {$A$};
        \node (B) at ($(A)-(u)$) {$B$};
        \node (C) at ($(B) + (r)$) {$C$};
        \draw[->,thick] (A) --node[left]{$\alpha$} (B);
        \draw[->,thick] (C) --node[above]{$\gamma$} (B);
        \draw[->,thick,dashed] (A) --node[above,yshift=0.3em]{$\lambda$} (C);
      \end{tikzpicture}    
 \end{center}
 in $\AC$ if and only if there exists a lift
 \begin{center}
    \begin{tikzpicture}[label/.style={postaction={
        decorate,
        decoration={markings, mark=at position .5 with \node #1;}},
        mylabel/.style={thick, draw=none, align=center, minimum width=0.5cm, minimum height=0.5cm,fill=white}},
        baseline=(base)]
        \coordinate (r) at (4.5,0);
        \coordinate (u) at (0,2);
        \node (A) {$1$};
        \node (B) at ($(A)-(u)$) {$H(A,B)$};
        \node (C) at ($(B) + (r)$) {$H(A,C)$};
        \draw[->,thick] (A) --node[left]{$\nu( \alpha )$} (B);
        \draw[->,thick] (C) --node[above]{$H(A,\gamma)$} (B);
        \draw[->,thick,dashed] (A) --node[above,yshift=0.3em]{$\lambda'$} (C);
      \end{tikzpicture}    
 \end{center}
 in $\BC$. In other words, we can decide lifts in $\AC$ whenever we can decide lifts in $\BC$.
\end{lemma}
\begin{proof}
It is easy to see that
\[
 \nu: \Hom_{\AC}( A, C ) \longrightarrow \Hom( 1, H(A,C) )
\]
induces a bijection between lifts of the former system and lifts of the latter,
since, by naturality, we have
\begin{align*}
 \nu( \alpha ) = \nu( \lambda \cdot \gamma ) = \nu( \lambda ) \cdot H( A, \gamma ).
\end{align*}
\end{proof}

\begin{example} 
 Let $k$ be a field with decidable equality and let $Q$ be an acyclic quiver.
 Then the $\Rows_k$-homomorphism structure of $\AsCat(k,Q)^{\oplus}$
 described in Example \ref{example:hom_structure_for_quiver},
 the statement in Lemma \ref{lemma:lifts_via_hom},
 and the decidability of lifts in $\Rows_k$ (Example \ref{example:lifts_and_gauss})
 imply the decidability of lifts in $\AsCat(k, Q)^{\oplus}$.
 
 The same holds for $\Z$-graded $k$-algebras $R = \bigoplus_{d \in \Z} R_d$
 with finite dimensional degree-parts, see Example \ref{example:hom_structure_for_graded_R}.
\end{example}

\subsubsection{Cokernels}

Just as the additive closure turns an Ab-category into an additive one,
Freyd categories endow additive categories with cokernels.

\begin{definition}
 Let $\AC$ be an additive category.
 Given $A,B \in \Obj_{\AC}$, $\alpha \in \Hom_{\AC}(A,B)$, 
 a cokernel of $\alpha$ consists of the following data:
 \begin{enumerate}
  \item An object $\CokernelObject( \alpha )$ (\textbf{cokernel object}), also denoted by $\cokernel( \alpha )$,
        and a morphism
        \[
         \CokernelProjection(\alpha) \in \Hom_{\AC}( B, \CokernelObject( \alpha )) \text{\hspace{1em}(\textbf{cokernel projection})}
        \]
        such that $\alpha \cdot \CokernelProjection(\alpha) = 0$.
  \item An algorithm that computes for given $T \in \Obj_{\AC}$, $\tau \in \Hom_{\AC}(B,T)$
        such that $\alpha \cdot \tau = 0$ 
        a morphism 
        \[
          \CokernelColift( \alpha, \tau ) \in \Hom_{\AC}(\cokernel( \alpha ), T) \text{\hspace{1em}(\textbf{cokernel colift})}
        \]
        such that 
        \[
        \CokernelProjection( \alpha ) \cdot \CokernelColift( \alpha, \tau ) = \tau,
        \]
        where $\CokernelColift( \alpha, \tau )$ is uniquely determined (up to equality of morphisms) by this property.
 \end{enumerate}
\end{definition}

\begin{example}
Let $R$ be a ring and let $\rho: R^{1 \times b} \longrightarrow R^{1 \times a}$
be an $R$-module homomorphism.
Then $\cokernel( \rho ) \in R\fpmodl$ is mapped to an object in $\Freyd( \Rows_R )$
via the equivalence
\[
 R\fpmodl \simeq \Freyd( \Rows_R ),
\]
and this object is given, up to isomorphism, by the morphism $\rho$
itself.
In this sense, taking the cokernel of a morphism
between two row modules is a completely formal act.

\end{example}

Every morphism in $\Freyd( \AC )$ has a cokernel
by means of the following construction, whose proof of correctness can be found in
\cite[Section 3.1]{PosFreyd}.
\begin{construction}
 The following algorithm creates cokernel projections in $\Freyd( \AC )$:
 \begin{center}
 \begin{tabular}{ccc}
    \begin{tikzpicture}[label/.style={postaction={
      decorate,
      decoration={markings, mark=at position .5 with \node #1;}},
      mylabel/.style={thick, draw=none, align=center, minimum width=0.5cm, minimum height=0.5cm,fill=white}},
      baseline=(base)]
      \coordinate (r) at (2.5,0);
      \coordinate (u) at (0,2);
      \node (A) {$A$};
      \node (B) at ($(A)+(r)$) {$R_A$};
      \node (C) at ($(A) - (u)$) {$B$};
      \node (D) at ($(B) - (u)$) {$R_B$};
      \node (base) at ($(A) - 0.5*(u)$) {};
      \draw[->,thick] (B) --node[above]{$\rho_A$} (A);
      \draw[->,thick] (D) --node[above]{$\rho_B$} (C);
      \draw[->,thick] (A) --node[left]{$\alpha$} (C);
      \draw[->,thick,dashed] (B) --node[right]{$\rho_{\alpha}$} (D);
    \end{tikzpicture}
    &
      $\stackrel{\CokernelProjection}{\longmapsto}$
    &
    \begin{tikzpicture}[label/.style={postaction={
      decorate,
      decoration={markings, mark=at position .5 with \node #1;}},
      mylabel/.style={thick, draw=none, align=center, minimum width=0.5cm, minimum height=0.5cm,fill=white}},
      baseline=(base)]
      \coordinate (r) at (2.5,0);
      \coordinate (u) at (0,2);
      \node (A) {$B$};
      \node (B) at ($(A)+(r)$) {$R_B$};
      \node (C) at ($(A) - (u)$) {$B$};
      \node (D) at ($(B) - (u)$) {$R_B \oplus A$};
      \node (base) at ($(A) - 0.5*(u)$) {};
      \draw[->,thick] (B) --node[above]{$\rho_B$} (A);
      \draw[->,thick] (D) --node[above]{$\begin{pmatrix}\rho_B \\ \alpha\end{pmatrix}$} (C);
      \draw[->,thick] (A) --node[left]{$\id_B$} (C);
      \draw[->,thick,dashed] (B) --node[right]{$\begin{pmatrix}\id_{R_B} & 0\end{pmatrix}$} (D);
    \end{tikzpicture}
 \end{tabular}    
 \end{center}
 Moreover, for any morphism
 \begin{center}
   \begin{tikzpicture}[label/.style={postaction={
      decorate,
      decoration={markings, mark=at position .5 with \node #1;}},
      mylabel/.style={thick, draw=none, align=center, minimum width=0.5cm, minimum height=0.5cm,fill=white}},
      baseline=(base)]
      \coordinate (r) at (2.5,0);
      \coordinate (u) at (0,2);
      \node (A) {$B$};
      \node (B) at ($(A)+(r)$) {$R_B$};
      \node (C) at ($(A) - (u)$) {$T$};
      \node (D) at ($(B) - (u)$) {$R_T$};
      \node (base) at ($(A) - 0.5*(u)$) {};
      \draw[->,thick] (B) --node[above]{$\rho_B$} (A);
      \draw[->,thick] (D) --node[above]{$\rho_T$} (C);
      \draw[->,thick] (A) --node[left]{$\tau$} (C);
      \draw[->,thick,dashed] (B) --node[right]{$\rho_{\tau}$} (D);
    \end{tikzpicture}
 \end{center}
  and any witness $A \stackrel{\lambda}{\longrightarrow} R_T$ for the composition
 \begin{center}
  \begin{tikzpicture}[label/.style={postaction={
      decorate,
      decoration={markings, mark=at position .5 with \node #1;}},
      mylabel/.style={thick, draw=none, align=center, minimum width=0.5cm, minimum height=0.5cm,fill=white}},
      baseline=(base)]
      \coordinate (r) at (2.5,0);
      \coordinate (u) at (0,2);
      \node (A) {$A$};
      \node (B) at ($(A)+(r)$) {$R_A$};
      \node (C) at ($(A) - (u)$) {$T$};
      \node (D) at ($(B) - (u)$) {$R_T$};
      \node (base) at ($(A) - 0.5*(u)$) {};
      \draw[->,thick] (B) --node[above]{$\rho_A$} (A);
      \draw[->,thick] (D) --node[above]{$\rho_T$} (C);
      \draw[->,thick] (A) --node[left]{$\alpha \cdot \tau$} (C);
      \draw[->,thick,dashed] (B) --node[right]{$\rho_{\alpha} \cdot \rho_{\tau}$} (D);
    \end{tikzpicture}
 \end{center}
    being equal to zero in $\Freyd( \AC )$, we can construct a cokernel colift:
 \begin{center}
    \begin{tikzpicture}[label/.style={postaction={
      decorate,
      decoration={markings, mark=at position .5 with \node #1;}},
      mylabel/.style={thick, draw=none, align=center, minimum width=0.5cm, minimum height=0.5cm,fill=white}},
      baseline=(base)]
      \coordinate (r) at (2.5,0);
      \coordinate (u) at (0,2);
      \node (A) {$B$};
      \node (B) at ($(A)+(r)$) {$R_B \oplus A$};
      \node (C) at ($(A) - (u)$) {$T$};
      \node (D) at ($(B) - (u)$) {$R_T$};
      \node (base) at ($(A) - 0.5*(u)$) {};
      \draw[->,thick] (B) --node[above]{$\begin{pmatrix}\rho_B \\ \alpha\end{pmatrix}$} (A);
      \draw[->,thick] (D) --node[above]{$\rho_T$} (C);
      \draw[->,thick] (A) --node[left]{$\tau$} (C);
      \draw[->,thick,dashed] (B) --node[right]{$\begin{pmatrix}\rho_{\tau} \\ \lambda\end{pmatrix}$} (D);
    \end{tikzpicture}    
 \end{center}
\end{construction}

\subsubsection{Kernels}

Unlike cokernels, kernels in $\Freyd( \AC )$, if they exist,
cannot be constructed formally but only with the help of additional algorithms
in $\AC$.
\begin{definition}
 Let $\AC$ be an additive category.
 Given $A,B \in \Obj_{\AC}$, $\alpha \in \Hom_{\AC}(A,B)$, 
 a \textbf{kernel} of $\alpha$ consists of the following data:
 \begin{enumerate}
 \item  An object $\KernelObject( \alpha ) \in \Obj_{\AC}$ (\textbf{kernel object}), also denoted by $\kernel( \alpha )$,
        and a morphism
        \[\KernelEmbedding( \alpha ) \in \Hom_{\AC}(\KernelObject( \alpha ), A) \text{\hspace{1em}(\textbf{kernel embedding})}\]
        such that $\KernelEmbedding( \alpha ) \cdot \alpha = 0$.
  \item An algorithm that computes for given $T \in \Obj_{\AC}$, $\tau \in \Hom_{\AC}(T,A)$
        such that $\tau \cdot \alpha = 0$
        a morphism 
        \[
        \KernelLift( \alpha, \tau ) \in \Hom_{\AC}(T,\KernelObject( \alpha )) \text{\hspace{1em}(\textbf{kernel lift})}
        \]
        such that
        \[ \KernelLift( \alpha, \tau ) \cdot \KernelEmbedding( \alpha ) = \tau\]
        where $\KernelLift( \alpha, \tau )$ is uniquely determined (up to equality of morphisms) by this property.
  \end{enumerate}
\end{definition}

\begin{remark}\label{remark:freyd_kernels}
 Let $R$ be a ring.
Assume that we can produce for every $R$-module homomorphism
of the form $\rho: R^{1 \times b} \longrightarrow R^{1 \times a}$
another $R$-module homomorphism
\[
 \kappa: R^{1 \times c} \longrightarrow R^{1 \times b}
\]
whose image spans the kernel of $\rho$ as an $R$-module.
Then, by using such a procedure twice,
we are able to construct an exact sequence
\begin{center}
         \begin{tikzpicture}[label/.style={postaction={
          decorate,
          decoration={markings, mark=at position .5 with \node #1;}},
          mylabel/.style={thick, draw=none, align=center, minimum width=0.5cm, minimum height=0.5cm,fill=white}},baseline = (A)]
          \coordinate (r) at (2.5,0);
          \coordinate (u) at (0,2);
          \node (A) {$R^{1 \times c'}$};
          \node (B) at ($(A)-(r)$) {$R^{1 \times c}$};
          \node (C) at ($(B)-(r)$) {$R^{1 \times b}$};
          \node (D) at ($(C)-(r)$) {$R^{1 \times a}$};
          \draw[->,thick] (A) --node[above]{$\kappa'$} (B);
          \draw[->,thick] (B) --node[above]{$\kappa$} (C);
          \draw[->,thick] (C) --node[above]{$\rho$}  (D);
          \end{tikzpicture}
\end{center}
in which $\kappa'$ is a finite presentation of the kernel of $\rho$.
\end{remark}

Abstracting the procedure $\rho \mapsto \kappa$
from $\Rows_R$ to an arbitrary additive category $\AC$
leads to the notion of a weak kernel,
which is defined exactly like a kernel,
but we drop the uniqueness assumption of the kernel lift.
\begin{definition}
 Let $\AC$ be an additive category.
 Given $A,B \in \Obj_{\AC}$, $\alpha \in \Hom_{\AC}(A,B)$, 
 a \textbf{weak kernel} of $\alpha$ consists of the following data:
 \begin{enumerate}
 \item  An object $\weakkernel( \alpha ) \in \Obj_{\AC}$ (\textbf{weak kernel object})
        and a morphism
        \[\WeakKernelEmbedding( \alpha ) \in \Hom_{\AC}(\weakkernel( \alpha ), A) \text{\hspace{1em}(\textbf{weak kernel embedding})}\]
        such that $\WeakKernelEmbedding( \alpha ) \cdot \alpha = 0$.
  \item An algorithm that computes for given $T \in \Obj_{\AC}$, $\tau \in \Hom_{\AC}(T,A)$
        such that $\tau \cdot \alpha = 0$
        a morphism 
        \[
        \WeakKernelLift( \alpha, \tau ) \in \Hom_{\AC}(T,\weakkernel( \alpha )) \text{\hspace{1em}(\textbf{weak kernel lift})}\]
        such that
        \[ \WeakKernelLift( \alpha, \tau ) \cdot \WeakKernelEmbedding( \alpha ) = \tau.\]
  \end{enumerate}
\end{definition}

\begin{example}
 We unravel the definition of a weak kernel in the concrete case
 where $R$ is a ring and $\AC = \Rows_R$.
 So, given a matrix $R^{1 \times b} \stackrel{\rho}{\longrightarrow} R^{1 \times a}$,
 i.e., a morphism in $\Rows_R$,
 a weak kernel of $\rho$ consists of
 \begin{enumerate}
  \item an object $R^{1 \times c}$,
  \item a matrix $R^{1 \times c} \stackrel{\kappa}{\longrightarrow} R^{1 \times b}$ such that $\kappa \cdot \rho = 0$,
  \item and for every other matrix $R^{1 \times t} \stackrel{\tau}{\longrightarrow} R^{1 \times b}$ such that $\tau \cdot \rho = 0$,
  we can find a lift $R^{1 \times t} \stackrel{u(\tau)}{\longrightarrow} R^{1 \times c}$
  making the diagram
  \begin{center}
    \begin{tikzpicture}[label/.style={postaction={
            decorate,
            decoration={markings, mark=at position .5 with \node #1;}},
            mylabel/.style={thick, draw=none, align=center, minimum width=0.5cm, minimum height=0.5cm,fill=white}}]
            \coordinate (r) at (2.5,0);
            \coordinate (u) at (0,1.5);
            \node (K) {$R^{1 \times c}$};
            \node (A) at ($(K)+(r) - 0.5*(u)$) {$R^{1 \times b}$};
            \node (B) at ($(A)+(r)$) {$R^{1 \times a}$};
            \node (T) at ($(K) - (u)$) {$R^{1 \times t}$};
            \draw[->,thick] (K) --node[above]{$\kappa$} (A);
            \draw[->,thick] (A) --node[above]{$\rho$} (B);
            \draw[->,thick,dashed] (T) --node[left]{$u(\tau)$} (K);
            \draw[->,thick] (T) --node[below]{$\tau$} (A);
    \end{tikzpicture}
  \end{center}
  commutative.
  In matrix terms, this means that the rows of $\kappa$
  have to span the row kernel (also called \textbf{syzygies}) of $\rho$, since
  we can express every collection of rows $\tau$ lying in the row kernel of $\rho$
  as a linear combination (given by $u(\tau)$) of the rows in $\kappa$.
 \end{enumerate}
 But since these linear combinations do not have to be uniquely determined,
 we deal with \emph{weak} kernels here.
 Thus, the existence of weak kernels in $\Rows_R$ is 
 equivalent to
 finding a finite generating system for row kernels of matrices over $R$.
 A ring for which row kernels are finitely generated is called \textbf{(left-)coherent}.
\end{example}

\begin{remark}
 Algorithms to compute syzygies in $\Rows_R$ mainly rely on the theory of Gröbner bases.
 For the cases of quotients of commutative polynomial rings (both graded and non-graded), see, e.g., \cite{GP}.
 For non-commutative cases (including finite dimensional quotients of path algebras), see, e.g., \cite{Green99}.
\end{remark}

Our goal is to describe kernels in $\Freyd( \AC )$
with the help of weak kernels in $\AC$.
In order to be able to do so,
we need the construction of weak pullbacks from weak kernels.

\begin{definition}\label{definition:weak_pullbacks}
 Let $\AC$ be an additive category. Given a cospan $A \stackrel{\alpha}{\longrightarrow} B \stackrel{\gamma}{\longleftarrow} C$ in $\AC$,
 a \textbf{weak pullback} consists of the following data:
 \begin{enumerate}
  \item An object $\WeakPullback( \alpha, \gamma ) \in \AC$.
  \item Morphisms
  \[
   \firstproj{\alpha}{\gamma}: \WeakPullback( \alpha, \gamma ) \rightarrow A
  \]
   and 
   \[
   \secondproj{\alpha}{\gamma}: \WeakPullback( \alpha, \gamma ) \rightarrow C
   \]
   such that 
   \[\firstproj{\alpha}{\gamma} \cdot \alpha = \secondproj{\alpha}{\gamma} \cdot \gamma.\]
  \item An algorithm that computes for $T \in \AC$ and morphisms $p: T \rightarrow A$, $q: T \rightarrow C$ with $p \cdot \alpha = q \cdot \gamma$
        a morphism 
        \[
         \bmatrowind{\alpha}{p}{q}{\gamma}: T \rightarrow \WeakPullback( \alpha, \gamma )
        \]
        satisfying
        \begin{center}
        \begin{tabular}{ccccc}
         $p = \bmatrowind{\alpha}{p}{q}{\gamma} \cdot \firstproj{\alpha}{\gamma}$
         & & and & &
         $q = \bmatrowind{\alpha}{p}{q}{\gamma} \cdot \secondproj{\alpha}{\gamma}$.
        \end{tabular}
        \end{center}
 \end{enumerate}
 
\end{definition}

\begin{remark}
 The only difference between pullbacks and weak pullbacks lies
 in the uniqueness of the induced morphism, which is missing in the case of weak pullbacks.
\end{remark}

\begin{construction}\label{construction:weak_pullbacks_from_weak_kernels}
 We show how to construct weak pullbacks from weak kernels in an additive category $\AC$.
 Let
    \begin{center}
\begin{tikzpicture}[label/.style={postaction={
      decorate,
      decoration={markings, mark=at position .5 with \node #1;}}}]
      \coordinate (r) at (1.5,0);
      \coordinate (u) at (0,1.25);
      \node (A) {$A$};
      \node (B) at ($(A)+(r)$) {$B$};
      \node (C) at ($(B) +(u)$) {$C$};
      \draw[->,thick] (A) -- node[below]{$\alpha$} (B);
      \draw[->,thick] (C) -- node[right]{$\gamma$} (B);
\end{tikzpicture}
\end{center}
be a cospan.
We define the \bfindex{diagonal difference}
\[
 \delta \coloneqq \begin{pmatrix}
             \alpha \\ -\gamma
           \end{pmatrix}
           : A \oplus C \rightarrow B.
\]
Then, we may set
\begin{enumerate}
 \item the weak pullback object
 \[\WeakPullback( \alpha, \gamma )\coloneqq \weakkernel( \delta ),\]
 \item the first weak pullback projection
 \begin{center}
        \begin{tikzpicture}[label/.style={postaction={
                decorate,
                decoration={markings, mark=at position .5 with \node #1;}},
                mylabel/.style={thick, draw=none, align=center, minimum width=0.5cm, minimum height=0.5cm,fill=white}}]
                \coordinate (r) at (2.5,0);
                \coordinate (u) at (0,1.5);
                \node (K) {$\WeakPullback( \alpha, \gamma )$};
                \node (A) at ($(K)+3.2*(r)$) {$A \oplus C$};
                \node (B) at ($(A)+(r)$) {$A$,};
                \node (T) at ($(K)-1.1*(r)$) {$\firstproj{\alpha}{\gamma}$};
                \draw[->,thick] (K) --node[above]{$\WeakKernelEmbedding( \delta )$} (A);
                \draw[->,thick] (A) --node[above]{$\pmatcol{1}{0}$} (B);
                \draw[draw=none] (T) --node[]{$\coloneqq$} (K);
        \end{tikzpicture}
      \end{center}
 \item the second weak pullback projection
 \begin{center}
        \begin{tikzpicture}[label/.style={postaction={
                decorate,
                decoration={markings, mark=at position .5 with \node #1;}},
                mylabel/.style={thick, draw=none, align=center, minimum width=0.5cm, minimum height=0.5cm,fill=white}}]
                \coordinate (r) at (2.5,0);
                \coordinate (u) at (0,1.5);
                \node (K) {$\WeakPullback( \alpha, \gamma )$};
                \node (A) at ($(K)+3.2*(r)$) {$A \oplus C$};
                \node (B) at ($(A)+(r)$) {$C$.};
                \node (T) at ($(K)-1.1*(r)$) {$\secondproj{\alpha}{\gamma}$};
                \draw[->,thick] (K) --node[above]{$\WeakKernelEmbedding( \delta )$} (A);
                \draw[->,thick] (A) --node[above]{$\pmatcol{0}{1}$} (B);
                \draw[draw=none] (T) --node[]{$\coloneqq$} (K);
        \end{tikzpicture}
      \end{center}
\end{enumerate}
Moreover, for any pair $p: T \rightarrow A$, $q: T \rightarrow C$ such that $p \cdot \alpha = q \cdot \gamma$, we set
\begin{enumerate}
    \setcounter{enumi}{3}
      \item  the morphism into the weak pullback
         \begin{center}
    \begin{tikzpicture}[label/.style={postaction={
            decorate,
            decoration={markings, mark=at position .5 with \node #1;}},
            mylabel/.style={thick, draw=none, align=center, minimum width=0.5cm, minimum height=0.5cm,fill=white}}]
            \coordinate (r) at (3.5,0);
            \coordinate (u) at (0,2);
            \node (K) {$\WeakPullback( \alpha, \gamma )$};
            \node (A) at ($(K)+(r) - 0.5*(u)$) {$A \oplus C$};
            \node (B) at ($(A)+(r)$) {$B$.};
            \node (T) at ($(K) - (u)$) {$T$};
            \draw[->,thick] (K) --node[above]{} (A);
            \draw[->,thick] (A) --node[above]{$\delta$} (B);
            \draw[->,thick,dashed] (T) --node[left]{$\bmatrowind{\alpha}{p}{q}{\gamma} \coloneqq \WeakKernelLift\left( \delta, \pmatrow{p}{q} \right)$} (K);
            \draw[->,thick] (T) --node[below,xshift = 0.5em]{$\pmatrow{p}{q}$} (A);
    \end{tikzpicture}
  \end{center}
\end{enumerate}

\end{construction}
\begin{proof}[Correctness of the construction]
The equation $p \cdot \alpha = q \cdot \gamma$
is equivalent to 
$\begin{pmatrix} p & q \end{pmatrix} \cdot \delta = 0$.
\end{proof}

\begin{example}
 Let $R$ be a ring.
 Computing the weak pullback of two morphisms in $\Rows_R$,
 i.e., of two matrices $\alpha, \gamma$ over $R$ having the same number of columns,
 amounts to computing the syzygies of the stacked matrix
 \[
  \pmatcol{\alpha}{-\gamma}.
 \]
\end{example}

\begin{construction}[Kernels in Freyd categories]\label{construction:kernels_in_Freyd}
 Let $\AC$ be an additive category in which we can compute weak kernels.
 By Construction \ref{construction:weak_pullbacks_from_weak_kernels},
 this means that we are able to construct weak pullbacks.
 We will use these for the construction of kernels in the Freyd category.
 Given a morphism
 \begin{center}
     \begin{tikzpicture}[label/.style={postaction={
      decorate,
      decoration={markings, mark=at position .5 with \node #1;}},
      mylabel/.style={thick, draw=none, align=center, minimum width=0.5cm, minimum height=0.5cm,fill=white}},
      baseline=(base)]
      \coordinate (r) at (2.5,0);
      \coordinate (u) at (0,2);
      \node (A) {$A$};
      \node (B) at ($(A)+(r)$) {$R_A$};
      \node (C) at ($(A) - (u)$) {$B$};
      \node (D) at ($(B) - (u)$) {$R_B$};
      \node (base) at ($(A) - 0.5*(u)$) {};
      \draw[->,thick] (B) --node[above]{$\rho_A$} (A);
      \draw[->,thick] (D) --node[above]{$\rho_B$} (C);
      \draw[->,thick] (A) --node[left]{$\alpha$} (C);
      \draw[->,thick,dashed] (B) --node[right]{$\rho_{\alpha}$} (D);
    \end{tikzpicture}
  \end{center}
  in $\Freyd( \AC )$.
  Generalizing the idea given in Remark \ref{remark:freyd_kernels}, we
  can construct its kernel object and kernel embedding as
  \begin{center}
     \begin{tikzpicture}[label/.style={postaction={
      decorate,
      decoration={markings, mark=at position .5 with \node #1;}},
      mylabel/.style={thick, draw=none, align=center, minimum width=0.5cm, minimum height=0.5cm,fill=white}},
      baseline=(base)]
      \coordinate (r) at (7,0);
      \coordinate (u) at (0,2.5);
      \node (A) {$\WeakPullback( \rho_B, \alpha )$};
      \node (B) at ($(A)+(r)$) {$\WeakPullback( \kappa, \rho_A )$};
      \node (C) at ($(A) - (u)$) {$A$};
      \node (D) at ($(B) - (u)$) {$R_A$};
      \node (base) at ($(A) - 0.5*(u)$) {};
      \draw[->,thick] (B) --node[above]{$\firstproj{\kappa}{\alpha}$} (A);
      \draw[->,thick] (D) --node[above]{$\rho_A$} (C);
      \draw[->,thick] (A) --node[left]{$\kappa \coloneqq \secondproj{\rho_B}{\alpha}$} (C);
      \draw[->,thick,dashed] (B) --node[right]{$\secondproj{\kappa}{\alpha}$} (D);
    \end{tikzpicture}
  \end{center}
  If we have a test morphism
  \begin{center}
     \begin{tikzpicture}[label/.style={postaction={
      decorate,
      decoration={markings, mark=at position .5 with \node #1;}},
      mylabel/.style={thick, draw=none, align=center, minimum width=0.5cm, minimum height=0.5cm,fill=white}},
      baseline=(base)]
      \coordinate (r) at (2.5,0);
      \coordinate (u) at (0,2);
      \node (A) {$T$};
      \node (B) at ($(A)+(r)$) {$R_T$};
      \node (C) at ($(A) - (u)$) {$A$};
      \node (D) at ($(B) - (u)$) {$R_A$};
      \node (base) at ($(A) - 0.5*(u)$) {};
      \draw[->,thick] (B) --node[above]{$\rho_T$} (A);
      \draw[->,thick] (D) --node[above]{$\rho_A$} (C);
      \draw[->,thick] (A) --node[left]{$\tau$} (C);
      \draw[->,thick,dashed] (B) --node[right]{$\rho_{\tau}$} (D);
    \end{tikzpicture}
  \end{center}
  whose composition with our first morphism yields zero in $\Freyd( \AC )$, i.e.,
  there exists a lift
  \begin{center}
\begin{tikzpicture}[label/.style={postaction={
  decorate,
  decoration={markings, mark=at position .5 with \node #1;}},
  mylabel/.style={thick, draw=none, align=center, minimum width=0.5cm, minimum height=0.5cm,fill=white}}]
  \coordinate (r) at (3.5,0);
  \coordinate (u) at (0,2);
  \node (m) {$B$};
  \node (n) at ($(m)+(r)$) {$R_B$,};
  \node (r1) at ($(m) + (u)$) {$T$};
  \draw[->,thick] (n) --node[above]{$\rho_B$} (m);
  \draw[->,thick,dashed] (r1) --node[above]{$\sigma$} (n);
  \draw[->,thick] (r1) --node[left]{$\tau \cdot \alpha$} (m);
\end{tikzpicture}
\end{center}
then we can construct the kernel lift
\begin{center}
     \begin{tikzpicture}[label/.style={postaction={
      decorate,
      decoration={markings, mark=at position .5 with \node #1;}},
      mylabel/.style={thick, draw=none, align=center, minimum width=0.5cm, minimum height=0.5cm,fill=white}},
      baseline=(base)]
      \coordinate (r) at (7,0);
      \coordinate (u) at (0,2.5);
      \node (A) {$\WeakPullback( \rho_B, \alpha )$};
      \node (B) at ($(A)+(r)$) {$\WeakPullback( \kappa, \rho_A )$};
      \node (C) at ($(A) + (u)$) {$T$};
      \node (D) at ($(B) + (u)$) {$R_T$};
      \node (base) at ($(A) - 0.5*(u)$) {};
      \draw[->,thick] (B) --node[above]{$\firstproj{\kappa}{\alpha}$} (A);
      \draw[->,thick] (D) --node[above]{$\rho_A$} (C);
      \draw[->,thick] (C) --node[left]{$\bmatrowind{\rho_B}{\sigma}{\tau}{\alpha}$} (A);
      \draw[->,thick,dashed] (D) --node[right]{$\bmatrowind{\kappa}{ \bmatrowind{\rho_B}{\sigma}{\tau}{\alpha} }{ \rho_{\tau} }{ \rho_A }$} (B);
    \end{tikzpicture}
  \end{center}
\end{construction}
\begin{proof}[Correctness of the construction]
 See \cite[Section 3.2]{PosFreyd}.
\end{proof}

\subsubsection{The abelian case}

Knowing how to construct kernels and cokernels in Freyd categories
allows us to construct pullbacks and pushouts: for pullbacks, we can proceed analogously
to Construction \ref{construction:weak_pullbacks_from_weak_kernels}. For pushouts,
we can proceed dually.

The construction of kernels in $\Freyd( \AC )$ relies on having weak kernels in $\AC$.
However, even more can be computed once $\AC$ has weak kernels:

\begin{theorem}[\cite{FreydRep}]\label{theorem:Freyd_main}
 $\Freyd( \AC )$ is abelian if and only if $\AC$ has weak kernels.
\end{theorem}

Here is the definition of an abelian category
as it can be found in textbooks like \cite{weihom}:
an abelian category
is an additive category $\AC$
with kernels and cokernels
such that
\begin{enumerate}
 \item every mono is the kernel of its cokernel,
 \item every epi is the cokernel of its kernel.
\end{enumerate}

Let us unravel these new requirements from an algorithmic point of view.
The first statement tells us that whenever we are
given a monomorphism $\alpha \in  \Hom_{\AC}(A,B)$,
it should have the same categorical properties as the kernel embedding of
the morphism $\CokernelProjection( \alpha )$.
Since we are able to compute kernel lifts for a given kernel embedding,
we have to be able to compute such lifts for $\alpha$ as well.
Thus, an algorithmic rereading of the first statement is given as follows:
an abelian category comes equipped with an algorithm
that computes for a given monomorphism $\alpha \in  \Hom_{\AC}(A,B)$ and 
given morphism $\tau \in \Hom_{\AC}(T,B)$ such that $\tau \cdot \CokernelProjection( \alpha ) = 0$
the \textbf{lift along a monomorphism} $u \in \Hom_{\AC}( T,A )$ (i.e., $u \cdot \alpha = \tau$).

Dually, the second statement can be rephrased as:
an abelian category comes equipped with an algorithm
that computes for a given epimorphism $\alpha \in \Hom_{\AC}(A,B)$ and 
given morphism $\tau \in \Hom_{\AC}(A,T)$ such that $\KernelEmbedding( \alpha ) \cdot \tau = 0$
the \textbf{colift along an epimorphism} $u \in \Hom_{\AC}(B,T)$ (i.e., $\alpha \cdot u = \tau$).

We will show how to compute lifts along monomorphisms in $\Freyd( \AC )$.

\begin{remark}\label{remark:witness_mono}
 Suppose given a monomorphism
 \begin{center}
     \begin{tikzpicture}[label/.style={postaction={
      decorate,
      decoration={markings, mark=at position .5 with \node #1;}},
      mylabel/.style={thick, draw=none, align=center, minimum width=0.5cm, minimum height=0.5cm,fill=white}},
      baseline=(base)]
      \coordinate (r) at (2.5,0);
      \coordinate (u) at (0,2);
      \node (A) {$A$};
      \node (B) at ($(A)+(r)$) {$R_A$};
      \node (C) at ($(A) - (u)$) {$B$};
      \node (D) at ($(B) - (u)$) {$R_B$};
      \node (base) at ($(A) - 0.5*(u)$) {};
      \draw[->,thick] (B) --node[above]{$\rho_A$} (A);
      \draw[->,thick] (D) --node[above]{$\rho_B$} (C);
      \draw[->,thick] (A) --node[left]{$\alpha$} (C);
      \draw[->,thick,dashed] (B) --node[right]{$\rho_{\alpha}$} (D);
    \end{tikzpicture}
  \end{center}
  in $\Freyd( \AC )$.
  Then its kernel embedding (see Construction \ref{construction:kernels_in_Freyd})
  \begin{center}
     \begin{tikzpicture}[label/.style={postaction={
      decorate,
      decoration={markings, mark=at position .5 with \node #1;}},
      mylabel/.style={thick, draw=none, align=center, minimum width=0.5cm, minimum height=0.5cm,fill=white}},
      baseline=(base)]
      \coordinate (r) at (7,0);
      \coordinate (u) at (0,2.5);
      \node (A) {$\WeakPullback( \rho_B, \alpha )$};
      \node (B) at ($(A)+(r)$) {$\WeakPullback( \kappa, \rho_A )$};
      \node (C) at ($(A) - (u)$) {$A$};
      \node (D) at ($(B) - (u)$) {$R_A$};
      \node (base) at ($(A) - 0.5*(u)$) {};
      \draw[->,thick] (B) --node[above]{$\firstproj{\kappa}{\alpha}$} (A);
      \draw[->,thick] (D) --node[above]{$\rho_A$} (C);
      \draw[->,thick] (A) --node[left]{$\kappa \coloneqq \secondproj{\rho_B}{\alpha}$} (C);
      \draw[->,thick,dashed] (B) --node[right]{$\secondproj{\kappa}{\alpha}$} (D);
    \end{tikzpicture}
  \end{center}
  is zero in $\Freyd( \AC )$. We call a witness for this kernel embedding being zero,
  which is nothing but a lift
  \begin{center}
     \begin{tikzpicture}[label/.style={postaction={
      decorate,
      decoration={markings, mark=at position .5 with \node #1;}},
      mylabel/.style={thick, draw=none, align=center, minimum width=0.5cm, minimum height=0.5cm,fill=white}},
      baseline=(base)]
      \coordinate (r) at (7,0);
      \coordinate (u) at (0,2.5);
      \node (A) {$\WeakPullback( \rho_B, \alpha )$};
      \node (B) at ($(A)+(r)$) {};
      \node (C) at ($(A) - (u)$) {$A$};
      \node (D) at ($(B) - (u)$) {$R_A$};

      \node (base) at ($(A) - 0.5*(u)$) {};
      
      \draw[->,thick] (D) --node[above]{$\rho_A$} (C);
      \draw[->,thick] (A) --node[left]{$\secondproj{\rho_B}{\alpha}$} (C);
      \draw[->,thick,dashed] (A) --node[above]{$\sigma$} (D);
    \end{tikzpicture}
  \end{center}
  
  a \textbf{witness for being a monomorphism} of our original morphism.
\end{remark}

\begin{construction}[Lift along monomorphism in Freyd categories]
 Let
 \begin{center}
     \begin{tikzpicture}[label/.style={postaction={
      decorate,
      decoration={markings, mark=at position .5 with \node #1;}},
      mylabel/.style={thick, draw=none, align=center, minimum width=0.5cm, minimum height=0.5cm,fill=white}},
      baseline=(base)]
      \coordinate (r) at (2.5,0);
      \coordinate (u) at (0,2);
      \node (A) {$A$};
      \node (B) at ($(A)+(r)$) {$R_A$};
      \node (C) at ($(A) - (u)$) {$B$};
      \node (D) at ($(B) - (u)$) {$R_B$};
      \node (base) at ($(A) - 0.5*(u)$) {};
      \draw[->,thick] (B) --node[above]{$\rho_A$} (A);
      \draw[->,thick] (D) --node[above]{$\rho_B$} (C);
      \draw[->,thick] (A) --node[left]{$\alpha$} (C);
      \draw[->,thick,dashed] (B) --node[right]{$\rho_{\alpha}$} (D);
    \end{tikzpicture}
  \end{center}
  be a monomorphism in $\Freyd( \AC )$ together with a witness for being a monomorphism (see Remark \ref{remark:witness_mono})
  \[
   \sigma: \WeakPullback( \rho_B, \alpha ) \longrightarrow R_A.
  \]
  Moreover, let
  \begin{center}
     \begin{tikzpicture}[label/.style={postaction={
      decorate,
      decoration={markings, mark=at position .5 with \node #1;}},
      mylabel/.style={thick, draw=none, align=center, minimum width=0.5cm, minimum height=0.5cm,fill=white}},
      baseline=(base)]
      \coordinate (r) at (2.5,0);
      \coordinate (u) at (0,2);
      \node (A) {$T$};
      \node (B) at ($(A)+(r)$) {$R_T$};
      \node (C) at ($(A) - (u)$) {$B$};
      \node (D) at ($(B) - (u)$) {$R_B$};
      \node (base) at ($(A) - 0.5*(u)$) {};
      \draw[->,thick] (B) --node[above]{$\rho_T$} (A);
      \draw[->,thick] (D) --node[above]{$\rho_B$} (C);
      \draw[->,thick] (A) --node[left]{$\tau$} (C);
      \draw[->,thick,dashed] (B) --node[right]{$\rho_{\tau}$} (D);
    \end{tikzpicture}
  \end{center}
  be a test morphism, i.e., a morphism in $\Freyd( \AC )$ whose composition with the cokernel projection 
  \begin{center}
   \begin{tikzpicture}[label/.style={postaction={
      decorate,
      decoration={markings, mark=at position .5 with \node #1;}},
      mylabel/.style={thick, draw=none, align=center, minimum width=0.5cm, minimum height=0.5cm,fill=white}},
      baseline=(base)]
      \coordinate (r) at (2.5,0);
      \coordinate (u) at (0,2);
      \node (A) {$B$};
      \node (B) at ($(A)+(r)$) {$R_B$};
      \node (C) at ($(A) - (u)$) {$B$};
      \node (D) at ($(B) - (u)$) {$R_B \oplus A$};
      \node (base) at ($(A) - 0.5*(u)$) {};
      \draw[->,thick] (B) --node[above]{$\rho_B$} (A);
      \draw[->,thick] (D) --node[above]{$\begin{pmatrix}\rho_B \\ \alpha\end{pmatrix}$} (C);
      \draw[->,thick] (A) --node[left]{$\id_B$} (C);
      \draw[->,thick,dashed] (B) --node[right]{$\begin{pmatrix}\id_{R_B} & 0\end{pmatrix}$} (D);
    \end{tikzpicture}
  \end{center}
  of our monomorphism yields zero, which, in turn, is witnessed by a lift
  \begin{center}
     \begin{tikzpicture}[label/.style={postaction={
      decorate,
      decoration={markings, mark=at position .5 with \node #1;}},
      mylabel/.style={thick, draw=none, align=center, minimum width=0.5cm, minimum height=0.5cm,fill=white}},
      baseline=(base)]
      \coordinate (r) at (4.5,0);
      \coordinate (u) at (0,2.5);
      \node (A) {$T$};
      \node (B) at ($(A)+(r)$) {};
      \node (C) at ($(A) - (u)$) {$B$};
      \node (D) at ($(B) - (u)$) {$R_B \oplus A$.};
      \node (base) at ($(A) - 0.5*(u)$) {};
      
      \draw[->,thick] (D) --node[below]{$\begin{pmatrix}\rho_B \\ \alpha\end{pmatrix}$} (C);
      \draw[->,thick] (A) --node[left]{$\tau$} (C);
      \draw[->,thick,dashed] (A) --node[right,yshift=0.5em]{$\pmatrow{\tau_{R_B}}{\tau_A}$} (D);
    \end{tikzpicture}
  \end{center}
  Then, we can construct the lift along monomorphism as
  \begin{center}
     \begin{tikzpicture}[label/.style={postaction={
      decorate,
      decoration={markings, mark=at position .5 with \node #1;}},
      mylabel/.style={thick, draw=none, align=center, minimum width=0.5cm, minimum height=0.5cm,fill=white}},
      baseline=(base)]
      \coordinate (r) at (2.5,0);
      \coordinate (u) at (0,2);
      \node (A) {$T$};
      \node (B) at ($(A)+(r)$) {$R_T$};
      \node (C) at ($(A) - (u)$) {$A$};
      \node (D) at ($(B) - (u)$) {$R_A$};
      \node (base) at ($(A) - 0.5*(u)$) {};
      \draw[->,thick] (B) --node[above]{$\rho_T$} (A);
      \draw[->,thick] (D) --node[above]{$\rho_A$} (C);
      \draw[->,thick] (A) --node[left]{$\tau_A$} (C);
      \draw[->,thick,dashed] (B) --node[right]{$\bmatrowind{\rho_B}{\rho_{\tau} - \rho_T \cdot \tau_{R_B}}{\rho_T \cdot \tau_A}{\alpha} \cdot \sigma$} (D);
    \end{tikzpicture}
  \end{center}
\end{construction}
\begin{proof}[Correctness of the construction]
 See \cite[Section 3.3]{PosFreyd}.
\end{proof}

How to proceed for colifts along epimorphisms can be seen in \cite[Section 3.4]{PosFreyd}.

\subsubsection{Homomorphisms}\label{subsubsection:homomorphisms}

We end this first section with a discussion of how to compute sets of homomorphisms in Freyd categories,
since this enables us, among other things, to compute sets of natural transformations between finitely presented functors.

Let $\AC$ be an additive category
and let $(A \stackrel{\rho_A}{\longleftarrow} R_A)$ and $(B \stackrel{\rho_B}{\longleftarrow} R_B)$ be objects in $\Freyd( \AC )$.
Recall that a morphism between these two objects
\begin{center}
     \begin{tikzpicture}[label/.style={postaction={
      decorate,
      decoration={markings, mark=at position .5 with \node #1;}},
      mylabel/.style={thick, draw=none, align=center, minimum width=0.5cm, minimum height=0.5cm,fill=white}},
      baseline=(base)]
      \coordinate (r) at (2.5,0);
      \coordinate (u) at (0,2);
      \node (A) {$A$};
      \node (B) at ($(A)+(r)$) {$R_A$};
      \node (C) at ($(A) - (u)$) {$B$};
      \node (D) at ($(B) - (u)$) {$R_B$};
      \node (base) at ($(A) - 0.5*(u)$) {};
      \draw[->,thick] (B) --node[above]{$\rho_A$} (A);
      \draw[->,thick] (D) --node[above]{$\rho_B$} (C);
      \draw[->,thick] (A) --node[left]{$\alpha$} (C);
      \draw[->,thick,dashed] (B) --node[right]{$\rho_{\alpha}$} (D);
      \draw[->,thick,dashed] (A) --node[above]{$\lambda$} (D);
    \end{tikzpicture}
\end{center}
consists of an element $\alpha \in \Hom_{\AC}( A, B )$ considered up to addition with an element of the form
$\lambda \cdot \rho_B$ such that there exists $\rho_{\alpha}$ with $\rho_A \cdot \alpha = \rho_{\alpha} \cdot \rho_B$.
In other words, the abelian group
\[
\mathcal{H} \coloneqq \Hom_{\Freyd( \AC )}\big( (A \stackrel{\rho_A}{\longleftarrow} R_A), (B \stackrel{\rho_B}{\longleftarrow} R_B)\big)
\]
is given by a certain subquotient of the abelian group $\Hom_{\AC}( A, B )$
that fits into the following commutative diagram of abelian groups with exact rows and columns:
\begin{center}
\begin{minipage}{\textwidth}
\begin{center}
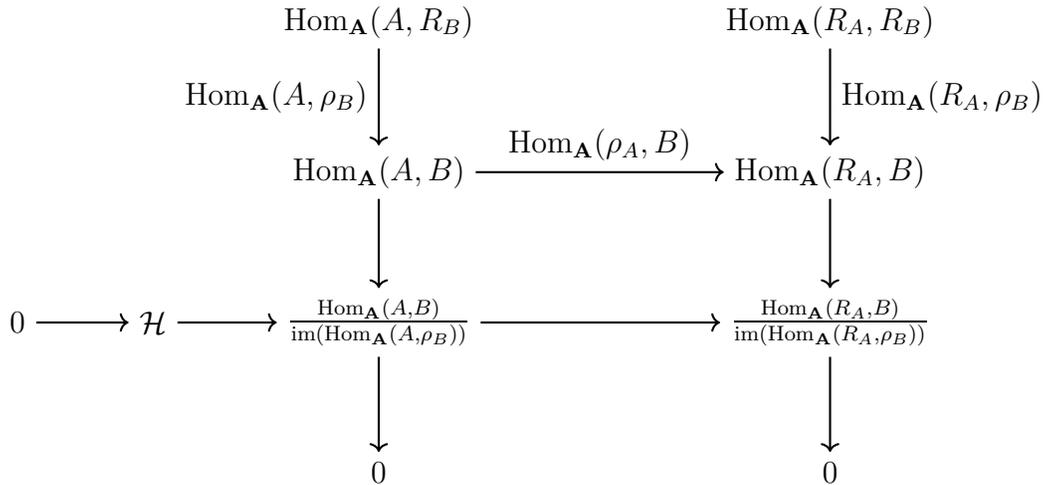

  \captionof{figure}{$\mathcal{H}$ as a subquotient of abelian groups.}
  \label{figure:hom_abel_grps}
     \begin{tikzpicture}[label/.style={postaction={
      decorate,
      decoration={markings, mark=at position .5 with \node #1;}},
      mylabel/.style={thick, draw=none, align=center, minimum width=0.5cm, minimum height=0.5cm,fill=white}}]
      \coordinate (r) at (6,0);
      \coordinate (u) at (0,2);
      \node (A) {$0$};
      \node (B) at ($(A)+0.3*(r)$) {$\mathcal{H}$};
      \node (C) at ($(B)+0.5*(r)$) {$\frac{\Hom_{\AC}(A,B)}{\im( \Hom_{\AC}( A, \rho_B ) ) }$};
      \node (D) at ($(C)+(r)$) {$\frac{\Hom_{\AC}(R_A,B)}{\im( \Hom_{\AC}( R_A, \rho_B ) ) }$};
      
      \node (C2) at ($(C)-(u)$) {$0$};
      \node (D2) at ($(D)-(u)$) {$0$};
      
      \node (C3) at ($(C)+(u)$) {$\Hom_{\AC}( A, B )$};
      \node (D3) at ($(D)+(u)$) {$\Hom_{\AC}( R_A, B )$};
      
      \node (C4) at ($(C3)+(u)$) {$\Hom_{\AC}( A, R_B )$};
      \node (D4) at ($(D3)+(u)$) {$\Hom_{\AC}( R_A, R_B )$};

      \draw[->,thick] (A) -- (B);
      \draw[->,thick] (B) -- (C);
      \draw[->,thick] (C) -- (D);
      
      \draw[->,thick] (C) -- (C2);
      \draw[->,thick] (D) -- (D2);
      
      \draw[->,thick] (C3) -- (C);
      \draw[->,thick] (D3) -- (D);
      
      \draw[->,thick] (C4) --node[left]{$\Hom_{\AC}( A, \rho_B )$} (C3);
      \draw[->,thick] (D4) --node[right]{$\Hom_{\AC}( R_A, \rho_B )$} (D3);
      
      \draw[->,thick] (C3) --node[above]{$\Hom_{\AC}( \rho_A, B )$} (D3);
    \end{tikzpicture}
\end{center}
\end{minipage}
\end{center}
Now, assume that $\AC$ has a $\BC$-homomorphism structure $(H, 1, \nu )$, where $\BC$ is an abelian category.
Then, inspired by the diagram of abelian groups above, we may construct a diagram with exact rows and columns in $\BC$:

\begin{center}
\begin{minipage}{\textwidth}
\begin{center}
  \captionof{figure}{Constructing a homomorphism structure for Freyd categories.}
  \label{figure:hom_freyd}
  \begin{tikzpicture}[label/.style={postaction={
      decorate,
      decoration={markings, mark=at position .5 with \node #1;}},
      mylabel/.style={thick, draw=none, align=center, minimum width=0.5cm, minimum height=0.5cm,fill=white}}]
      \coordinate (r) at (6,0);
      \coordinate (u) at (0,2);
      \node (A) {$0$};
      \node (B) at ($(A)+0.3*(r)$) {$\mathcal{H}'$};
      \node (C) at ($(B)+0.5*(r)$) {$\frac{H(A,B)}{\im( H( A, \rho_B ) }$};
      \node (D) at ($(C)+(r)$) {$\frac{H(R_A,B)}{\im( H( R_A, \rho_B ) }$};
      
      \node (C2) at ($(C)-(u)$) {$0$};
      \node (D2) at ($(D)-(u)$) {$0$};
      
      \node (C3) at ($(C)+(u)$) {$H( A, B )$};
      \node (D3) at ($(D)+(u)$) {$H( R_A, B )$};
      
      \node (C4) at ($(C3)+(u)$) {$H( A, R_B )$};
      \node (D4) at ($(D3)+(u)$) {$H( R_A, R_B )$};

      \draw[->,thick] (A) -- (B);
      \draw[->,thick] (B) -- (C);
      \draw[->,thick] (C) -- (D);
      
      \draw[->,thick] (C) -- (C2);
      \draw[->,thick] (D) -- (D2);
      
      \draw[->,thick] (C3) -- (C);
      \draw[->,thick] (D3) -- (D);
      
      \draw[->,thick] (C4) --node[left]{$H( A, \rho_B )$} (C3);
      \draw[->,thick] (D4) --node[right]{$H( R_A, \rho_B )$} (D3);
      
      \draw[->,thick] (C3) --node[above]{$H( \rho_A, B )$} (D3);
    \end{tikzpicture}
\end{center}
\end{minipage}
\end{center}

If $1 \in \BC$ is a projective object, then $\Hom_{\BC}( 1, - )$ is exact. Applying $\Hom_{\BC}( 1, - )$ to the diagram in Figure \ref{figure:hom_freyd} recovers the
diagram of abelian groups depicted in Figure \ref{figure:hom_abel_grps}. But this means
\[
 \Hom_{\BC}( 1, \mathcal{H}' ) \simeq \mathcal{H} \simeq \Hom_{\Freyd( \AC )}\big( (A \stackrel{\rho_A}{\longleftarrow} R_A), (B \stackrel{\rho_B}{\longleftarrow} R_B)\big).
\]
In other words, we used the $\BC$-homomorphism structure on $\AC$ to define a $\BC$-homomorphism structure on $\Freyd( \BC )$
(for more details, see \cite[Section 6.2]{PosFreyd}).

\subsection{Computing natural transformations}

As an application of the abstract algorithms that allow us to compute within Freyd categories,
we show how to compute sets of natural transformations between finitely presented functors.
Within this subsection, $R$ denotes a commutative coherent ring.

\begin{construction}\label{construction:hom_for_fp}
 Recall from Subsection \ref{subsection:Freyd_category} that
 the cascade of category constructors
 \[
 \Freyd( \Freyd( \AsCat(R)^{\oplus} )^{\op} )
\]
defines a category equivalent to finitely presented functors on the category of finitely presented modules over $R$.
We use the findings of the previous subsections to define an $\Freyd( \AsCat(R)^{\oplus} )$-homomorphism structure for this category.
\begin{enumerate}
 \item By Example \ref{example:hom_structure_for_R}, $\AsCat(R)$ has a $\AsCat(R)$-homomorphism structure.
 \item By Construction \ref{Construction:lift_of_hom_structure_to_additive_closure} and Remark \ref{remark:Ab_to_additive_for_hom_structure},
       we can extend this to a $\AsCat(R)^{\oplus}$-homo\-morphism structure for $\AsCat(R)^{\oplus}$.
 \item By applying the natural embedding $\AsCat(R)^{\oplus} \longrightarrow \Freyd( \AsCat(R)^{\oplus} )$, the category 
       $\AsCat(R)^{\oplus}$ has an $\Freyd( \AsCat(R)^{\oplus} )$-homomorphism structure.
 \item Since $R$ is coherent, $\Freyd( \AsCat(R)^{\oplus} )$ is
       abelian and the distinguished object of the homomorphism structure, corresponding to $R$, is projective.
       Thus, by the findings of Subsubsection \ref{subsubsection:homomorphisms}, we obtain an $\Freyd( \AsCat(R)^{\oplus} )$-homomorphism structure for $\Freyd( \AsCat(R)^{\oplus} )$.
 \item If an additive category $\AC$ has a $\BC$-homomorphism structure, then $\AC^{\op}$ has a $\BC$-homomorphism structure as well.
       In particular, $\Freyd( \AsCat(R)^{\oplus} )^{\op}$ has a $\Freyd( \AsCat(R)^{\oplus} )$-homomorphism structure.
 \item Last, we apply the findings of Subsubsection \ref{subsubsection:homomorphisms} again and arrive at the desired $\Freyd( \AsCat(R)^{\oplus} )$-homomorphism structure
       for $\Freyd( \Freyd( \AsCat(R)^{\oplus} )^{\op} )$.
\end{enumerate}
\end{construction}

We demonstrate how the algorithm for the computation of homomorphisms
that results from Construction \ref{construction:hom_for_fp} is carried out concretely.
For simplifying the notation we use the equivalence $\Freyd( \AsCat(R)^{\oplus} ) \simeq R\fpmodl$,
but keep in mind that computing kernels, cokernels, and homomorphisms for $R\fpmodl$
can all be carried out by means of the results in Subsection \ref{subsection:Freyd_category} on Freyd categories.
We start with a simple example.
\begin{example}
 Given the functors $\Hom_{\Z}( \Z/2\Z, - )$ and $\Ext_{\Z}^1( \Z/2\Z, - )$,
 we want to confirm computationally
 \[
  \Hom\big( \Hom_{\Z}( \Z/2\Z, - ), \Ext_{\Z}^1( \Z/2\Z, - ) \big) \simeq \Ext_{\Z}^1( \Z/2\Z, \Z/2\Z ) \simeq \Z/2\Z.
 \]
 The functor $\Hom_{\Z}( \Z/2\Z, - )$ considered as an object in $\Freyd( \Z\fpmodl^{\op} )$ is given by
 \[
  \Z/2\Z \longrightarrow 0.
 \]
 The functor $\Ext_{\Z}^1( \Z/2\Z, - )$ considered as an object in $\Freyd( \Z\fpmodl^{\op} )$ is given by
 \[
  \Z \stackrel{2}{\longrightarrow} \Z,
 \]
 see Example \ref{example:exts}.
 Now, plugging these data into the diagram in Figure \ref{figure:hom_freyd} and
 computing the cokernels, the induced morphism, and the kernel,
 we end up with the diagram
 \begin{center}
    \begin{tikzpicture}[label/.style={postaction={
      decorate,
      decoration={markings, mark=at position .5 with \node #1;}}},
      mylabel/.style={thick, draw=black, align=center, minimum width=0.5cm, minimum height=0.5cm,fill=white}]
      \coordinate (r) at (6,0);
      \coordinate (u) at (0,2);
      \node (A) {$0$};
      \node[mylabel] (B) at ($(A)+0.5*(r)$) {$\Z/2\Z$};
      \node (C) at ($(B)+0.5*(r)$) {$\Z/2\Z$};
      \node (D) at ($(C)+(r)$) {$0$};
      
      \node (C2) at ($(C)-(u)$) {$0$};
      \node (D2) at ($(D)-(u)$) {$0$};
      
      \node (C3) at ($(C)+(u)$) {$\Z/2\Z \simeq H( \Z, \Z/2\Z ) $};
      \node (D3) at ($(D)+(u)$) {$0 \simeq H( \Z, 0 )$};
      
      \node (C4) at ($(C3)+(u)$) {$\Z/2\Z \simeq H( \Z, \Z/2\Z )$};
      \node (D4) at ($(D3)+(u)$) {$0 \simeq H( \Z, 0 )$};

      \draw[shorten >=0.2em,shorten <=0.2em,->,thick] (A) -- (B);
      \draw[shorten >=0.2em,shorten <=0.2em,->,thick] (B) -- (C);
      \draw[->,thick] (C) -- (D);
      
      \draw[->,thick] (C) -- (C2);
      \draw[->,thick] (D) -- (D2);
      
      \draw[->,thick] (C3) -- (C);
      \draw[->,thick] (D3) -- (D);
      
      \draw[->,thick] (C4) --node[left]{$2$} (C3);
      \draw[->,thick] (D4) -- (D3);
      
      \draw[->,thick] (C3) -- (D3);
    \end{tikzpicture}
\end{center}
where we find our desired result inside the box.
\end{example}

Let $M$ be a finitely presented $R$-module.
In order to provide more complicated examples,
we show how to represent the functor $(M \otimes -)$ in $\Freyd( R\fpmodl^{\op} )$, see also
\cite[Lemma 6.1]{A}.
Let
\begin{center}
         \begin{tikzpicture}[label/.style={postaction={
          decorate,
          decoration={markings, mark=at position .5 with \node #1;}},
          mylabel/.style={thick, draw=none, align=center, minimum width=0.5cm, minimum height=0.5cm,fill=white}}]
          \coordinate (r) at (2.5,0);
          \coordinate (u) at (0,2);
          \node (A) {$R^{1 \times b}$};
          \node (B) at ($(A)-(r)$) {$R^{1 \times a}$};
          \node (C) at ($(B)-(r)$) {$M$};
          \node (D) at ($(C)-(r)$) {$0$};
          \draw[->,thick] (A) --node[above]{$\rho_M$} (B);
          \draw[->,thick] (B) -- (C);
          \draw[->,thick] (C) -- (D);
          \end{tikzpicture}
\end{center}
be a presentation of $M$.
The right exactness of the tensor product yields an exact sequence of functors
\begin{center}
         \begin{tikzpicture}[label/.style={postaction={
          decorate,
          decoration={markings, mark=at position .5 with \node #1;}},
          mylabel/.style={thick, draw=none, align=center, minimum width=0.5cm, minimum height=0.5cm,fill=white}}]
          \coordinate (r) at (4,0);
          \coordinate (u) at (0,2);
          \node (A) {$(R^{1 \times b} \otimes -)$};
          \node (B) at ($(A)-1.3*(r)$) {$(R^{1 \times a} \otimes -)$};
          \node (C) at ($(B)-(r)$) {$(M \otimes -)$};
          \node (D) at ($(C)-0.7*(r)$) {$0$};
          \draw[->,thick] (A) --node[above]{$\rho_M \otimes -$} (B);
          \draw[->,thick] (B) -- (C);
          \draw[->,thick] (C) -- (D);
          \end{tikzpicture}
\end{center}
where $\otimes$ is taken over $R$. 
For any free module $R^{1 \times c}$ where $c \in \N_0$, there are isomorphisms
\[
 R^{1 \times c} \otimes N \simeq N^{1 \times c} \simeq \Hom_R(R^{1 \times c}, N)
\]
natural in $N \in R\fpmodl$. Applied to the exact sequence above yields the presentation
\begin{center}
         \begin{tikzpicture}[label/.style={postaction={
          decorate,
          decoration={markings, mark=at position .5 with \node #1;}},
          mylabel/.style={thick, draw=none, align=center, minimum width=0.5cm, minimum height=0.5cm,fill=white}}]
          \coordinate (r) at (4,0);
          \coordinate (u) at (0,2);
          \node (A) {$(R^{1 \times b},-)$.};
          \node (B) at ($(A)-1.3*(r)$) {$(R^{1 \times a},-)$};
          \node (C) at ($(B)-(r)$) {$(M \otimes -)$};
          \node (D) at ($(C)-0.7*(r)$) {$0$};
          \draw[->,thick] (A) --node[above]{$(\rho_M^{\tr},-)$} (B);
          \draw[->,thick] (B) -- (C);
          \draw[->,thick] (C) -- (D);
          \end{tikzpicture}
\end{center}
Thus, $(M \otimes -)$ is given as an object in $\Freyd( R\fpmodl^{\op} )$ by
\begin{center}
         \begin{tikzpicture}[label/.style={postaction={
          decorate,
          decoration={markings, mark=at position .5 with \node #1;}},
          mylabel/.style={thick, draw=none, align=center, minimum width=0.5cm, minimum height=0.5cm,fill=white}}]
          \coordinate (r) at (4,0);
          \coordinate (u) at (0,2);
          \node (A) {$R^{1 \times b}$.};
          \node (B) at ($(A)-1.3*(r)$) {$R^{1 \times a}$};
          \draw[->,thick] (B) --node[above]{$\rho_M^{\tr}$} (A);
          \end{tikzpicture}
\end{center}
\begin{example}
 Let $R \coloneqq \Q[x,y]$ and let 
 \[
 M \coloneqq R^{1 \times 2}/ \langle \pmatrow{x}{y} \rangle. 
 \]
 We wish to compute
 \[
  \Hom\big( (M \otimes_{R} -), \Ext^1(M,-) \big).
 \]
 As seen above, the functor $(M \otimes_{R} -)$ considered as an object in $\Freyd( R\fpmodl^{\op} )$ is given by
 \[
  R^{1 \times 2} \stackrel{\pmatcol{x}{y}}{\longrightarrow} R^{1 \times 1}
 \]
 and the functor $\Ext^1_{R}( M, - )$ considered as an object in $\Freyd( R\fpmodl^{\op} )$ is given by
 \[
  R^{1 \times 1} \stackrel{\pmatrow{x}{y}}{\longrightarrow} R^{1 \times 2}.
 \]
 Again, we use the diagram in Figure \ref{figure:hom_freyd} for our computation
 \begin{center}
     \begin{tikzpicture}[label/.style={postaction={
      decorate,
      decoration={markings, mark=at position .5 with \node #1;}}},
      mylabel/.style={thick, draw=black, align=center, minimum width=0.5cm, minimum height=0.5cm,fill=white}]
      \coordinate (r) at (6.5,0);
      \coordinate (u) at (0,2);
      \node (A) {$0$};
      \node[mylabel] (B) at ($(A)+0.4*(r)$) {$({R}/{\langle x,y \rangle})^{1 \times 2}$};
      \node (C) at ($(B)+0.5*(r)$) {$({R}/{\langle x,y \rangle})^{1 \times 2}$};
      \node (D) at ($(C)+(r)$) {${R}/{\langle x,y \rangle}$};
      
      \node (C2) at ($(C)-(u)$) {$0$};
      \node (D2) at ($(D)-(u)$) {$0$};
      
      \node (C3) at ($(C)+(u)$) {$R^{1 \times 2}$};
      \node (D3) at ($(D)+(u)$) {$R$};
      
      \node (C4) at ($(C3)+(u)$) {$R^{2 \times 2}$};
      \node (D4) at ($(D3)+(u)$) {$R^{2 \times 1}$};

      \draw[shorten >=0.2em,shorten <=0.2em,->,thick] (A) -- (B);
      \draw[shorten >=0.2em,shorten <=0.2em,->,thick] (B) -- (C);
      \draw[->,thick] (C) --node[above]{$0$} (D);
      
      \draw[->,thick] (C) -- (C2);
      \draw[->,thick] (D) -- (D2);
      
      \draw[->,thick] (C3) -- (C);
      \draw[->,thick] (D3) -- (D);
      
      \draw[->,thick] (C4) --node[left]{$(A \mapsto \pmatrow{x}{y}A)$} (C3);
      \draw[->,thick] (D4) --node[right]{$(v \mapsto \pmatrow{x}{y}v)$} (D3);
      
      \draw[->,thick] (C3) --node[above]{$(w \mapsto w\pmatcol{x}{y})$} (D3);
    \end{tikzpicture}
\end{center}
from which we conclude
\[
  \Hom\big( (M \otimes_{R} -), \Ext^1(M,-) \big) \simeq ({R}/{\langle x,y \rangle})^{1 \times 2}.
\]
\end{example}

Last, the functors $\Tor_i(M,-)$ for $i > 0$ are also finitely presented and can thus be represented as objects in $\Freyd( R\fpmodl^{\op} )$, see also \cite[Theorem 10.2.35]{PrestPSL}.
For $\Tor_1(M,-)$, let 
\begin{center}
         \begin{tikzpicture}[label/.style={postaction={
          decorate,
          decoration={markings, mark=at position .5 with \node #1;}},
          mylabel/.style={thick, draw=none, align=center, minimum width=0.5cm, minimum height=0.5cm,fill=white}}]
          \coordinate (r) at (2.5,0);
          \coordinate (u) at (0,2);
          \node (A) {$R^{1 \times b}$};
          \node (A2) at ($(A)+(r)$) {$R^{1 \times c}$};
          \node (B) at ($(A)-(r)$) {$R^{1 \times a}$};
          \node (C) at ($(B)-(r)$) {$M$};
          \node (D) at ($(C)-(r)$) {$0$};
          \draw[->,thick] (A) --node[above]{$\iota$} (B);
          \draw[->,thick] (B) --node[above]{$\epsilon$} (C);
          \draw[->,thick] (C) -- (D);
          \draw[->,thick] (A2) --node[above]{$\rho$} (A);
          \end{tikzpicture}
\end{center}
be an exact sequence, and set
\[
 \Omega^1( M ) \coloneqq \kernel( \epsilon ) \simeq \image( \iota ) \simeq \cokernel( \rho ).
\]
We have an isomorphism
\[
  \Tor_1(M,N) \simeq \kernel\big( \Omega^1(M) \otimes N \rightarrow R^{1\times a} \otimes N \big)
\]
natural in $N \in R\fpmodl$, which means that $\Tor_1(M,-)$ can be computed as the kernel of
\begin{equation}\label{equation:nat_trans}
  (\Omega^1(M) \otimes -) \rightarrow (R^{1\times a} \otimes -).
\end{equation}
Thus, all we need to do is to translate this natural transformation to a morphism in $\Freyd( R\fpmodl^{\op} )$ and take its kernel.
Lifting the embedding $\Omega^1(M) \longrightarrow R^{1 \times a}$ to presentations is simply given
by the following commutative diagram with exact rows:
\begin{center}
         \begin{tikzpicture}[label/.style={postaction={
          decorate,
          decoration={markings, mark=at position .5 with \node #1;}},
          mylabel/.style={thick, draw=none, align=center, minimum width=0.5cm, minimum height=0.5cm,fill=white}}]
          \coordinate (r) at (2.5,0);
          \coordinate (u) at (0,2);
          \node (A) {$R^{1 \times c}$};
          \node (B) at ($(A)-(r)$) {$R^{1 \times b}$};
          \node (C) at ($(B)-(r)$) {$\Omega^1(M)$};
          \node (D) at ($(C)-(r)$) {$0$};
          
          \node (A2) at ($(A) - (u)$) {$0$};
          \node (B2) at ($(A2)-(r)$) {$R^{1 \times a}$};
          \node (C2) at ($(B2)-(r)$) {$R^{1 \times a}$};
          \node (D2) at ($(C2)-(r)$) {$0$};
          \draw[->,thick] (A) --node[above]{$\rho$} (B);
          \draw[->,thick] (B) -- (C);
          \draw[->,thick] (C) -- (D);
          
          \draw[->,thick] (A2) -- (B2);
          \draw[->,thick] (B2) -- (C2);
          \draw[->,thick] (C2) -- (D2);
          
          \draw[->,thick] (A) -- (A2);
          \draw[->,thick] (B) --node[left]{$\iota$} (B2);
          \draw[->,thick] (C) -- (C2);
          \end{tikzpicture}
\end{center}
The transposition of its right square is our desired representation of \eqref{equation:nat_trans} in $\Freyd( R\fpmodl^{\op} )$:
\begin{center}
 \begin{tabular}{ccccccc}
    \begin{tikzpicture}[label/.style={postaction={
          decorate,
          decoration={markings, mark=at position .5 with \node #1;}},
          mylabel/.style={thick, draw=none, align=center, minimum width=0.5cm, minimum height=0.5cm,fill=white}},baseline=(base)]
          \coordinate (r) at (3,0);
          \coordinate (u) at (0,2);
          \node (A) {$(\Omega^1(M) \otimes -)$};
          \node (C) at ($(A) - (u)$) {$(R^{1\times a} \otimes -)$};
          \node (base) at ($(A) - 0.5*(u)$) {};
          \draw[->,thick] (A) -- (C);
      \end{tikzpicture}
  & & & corresponds to & & &
    \begin{tikzpicture}[label/.style={postaction={
          decorate,
          decoration={markings, mark=at position .5 with \node #1;}},
          mylabel/.style={thick, draw=none, align=center, minimum width=0.5cm, minimum height=0.5cm,fill=white}},baseline=(base)]
          \coordinate (r) at (3,0);
          \coordinate (u) at (0,2);
          \node (A) {$R^{1 \times b}$};
          \node (B) at ($(A)+(r)$) {$R^{1 \times c}$};
          \node (C) at ($(A) - (u)$) {$R^{1 \times a}$};
          \node (D) at ($(B) - (u)$) {$0$};
          \node (base) at ($(A) - 0.5*(u)$) {};
          \draw[<-,thick] (B) --node[above]{$\rho^{\tr}$} (A);
          \draw[<-,thick] (D) --node[above]{$0$} (C);
          \draw[<-,thick] (A) --node[left]{$\iota^{\tr}$} (C);
          \draw[<-,thick,dashed] (B) -- (D);
      \end{tikzpicture}
 \end{tabular}
\end{center}
For the construction of its kernel, we apply Construction \ref{construction:kernels_in_Freyd} with $\AC = R\fpmodl^{\op}$.
Since pullbacks in abelian categories are in particular weak pullbacks, and since pullbacks and pushouts are dual concepts,
we end up with
\begin{center}
 \begin{tabular}{ccccccc}
    \begin{tikzpicture}[label/.style={postaction={
          decorate,
          decoration={markings, mark=at position .5 with \node #1;}},
          mylabel/.style={thick, draw=none, align=center, minimum width=0.5cm, minimum height=0.5cm,fill=white}},baseline=(A)]
          \coordinate (r) at (3,0);
          \coordinate (u) at (0,2);
          \node (A) {$(\Omega^1(M) \otimes -)$};
          \node (C) at ($(A) - (u)$) {$(R^{1\times a} \otimes -)$};
          \node (K) at ($(A) + (u)$) {$\Tor_1(M,-)$};
          \node (base) at ($(A) - 0.5*(u)$) {};
          \draw[->,thick] (A) -- (C);
          \draw[->,thick] (K) -- (A);
      \end{tikzpicture}
  & & & corresponds to & & &
    \begin{tikzpicture}[label/.style={postaction={
          decorate,
          decoration={markings, mark=at position .5 with \node #1;}},
          mylabel/.style={thick, draw=none, align=center, minimum width=0.5cm, minimum height=0.5cm,fill=white}},baseline=(A)]
          \coordinate (r) at (4,0);
          \coordinate (u) at (0,2);
          \node (A) {$R^{1 \times b}$};
          \node (B) at ($(A)+(r)$) {$R^{1 \times c}$};
          \node (C) at ($(A) - (u)$) {$R^{1 \times a}$};
          \node (D) at ($(B) - (u)$) {$0$};
          
          \node (E) at ($(A) + (u)$) {$\cokernel( \iota^{\tr} )$};
          \node (F) at ($(B) + (u)$) {$\cokernel( \iota^{\tr} ) \amalg_{R^{1 \times b}} R^{1 \times c}$};
          
          \node (base) at ($(A) - 0.5*(u)$) {};
          \draw[<-,thick] (B) --node[above]{$\rho^{\tr}$} (A);
          \draw[<-,thick] (D) --node[above]{$0$} (C);
          \draw[<-,thick] (A) --node[left]{$\iota^{\tr}$} (C);
          \draw[<-,thick,dashed] (B) -- (D);
          
          \draw[->,thick] (A) -- (E);
          \draw[->,thick,dashed] (B) -- (F);
          \draw[->,thick] (E) -- (F);
      \end{tikzpicture}
 \end{tabular}
\end{center}
where $\cokernel( \iota^{\tr} ) \amalg_{R^{1 \times b}} R^{1 \times c}$ denotes the pushout of the cokernel projection
$R^{1 \times b} \rightarrow \cokernel( \iota^{\tr} )$
and $\rho^{\tr}$.
For higher $\Tor$s, we simply need to replace $\Omega^1(M)$ with a higher syzygy object.

\begin{example}
 We set $R \coloneqq \Q[x,y]$ and again take a look at the module
 \[
 M \coloneqq R^{1 \times 2}/ \langle \pmatrow{x}{y} \rangle. 
 \]
 This time, we wish to compute
 \[
  \Hom\big( \Tor_1(M,-), \Ext^1(M,-) \big).
 \]
 Again, the functor $\Ext^1_{R}( M, - )$ considered as an object in $\Freyd( R\fpmodl^{\op} )$ is given by
 \[
  R^{1 \times 1} \stackrel{\pmatrow{x}{y}}{\longrightarrow} R^{1 \times 2}.
 \]
 Using the description preceding this example, we see that $\Tor_1(M,-)$ considered as an object in $\Freyd( R\fpmodl^{\op} )$ is given by
 \[
  R^{1 \times 1}/\langle x,y \rangle \longrightarrow 0.
 \]
 Again, we use the diagram in Figure \ref{figure:hom_freyd} for our computation
 \begin{center}
     \begin{tikzpicture}[label/.style={postaction={
      decorate,
      decoration={markings, mark=at position .5 with \node #1;}}},
      mylabel/.style={thick, draw=black, align=center, minimum width=0.5cm, minimum height=0.5cm,fill=white}]
      \coordinate (r) at (6.5,0);
      \coordinate (u) at (0,2);
      \node (A) {$0$};
      \node[mylabel] (B) at ($(A)+0.4*(r)$) {${R}/{\langle x,y \rangle}$};
      \node (C) at ($(B)+0.5*(r)$) {${R}/{\langle x,y \rangle}$};
      \node (D) at ($(C)+(r)$) {$0$};
      
      \node (C2) at ($(C)-(u)$) {$0$};
      \node (D2) at ($(D)-(u)$) {$0$};
      
      \node (C3) at ($(C)+(u)$) {$R/\langle x, y \rangle$};
      \node (D3) at ($(D)+(u)$) {$0$};
      
      \node (C4) at ($(C3)+(u)$) {$R/\langle x, y \rangle^{1 \times 2}$};
      \node (D4) at ($(D3)+(u)$) {$0$};

      \draw[shorten >=0.2em,shorten <=0.2em,->,thick] (A) -- (B);
      \draw[shorten >=0.2em,shorten <=0.2em,->,thick] (B) -- (C);
      \draw[->,thick] (C) -- (D);
      
      \draw[->,thick] (C) -- (C2);
      \draw[->,thick] (D) -- (D2);
      
      \draw[->,thick] (C3) --node[left]{$\id$} (C);
      \draw[->,thick] (D3) -- (D);
      
      \draw[->,thick] (C4) --node[left]{$0$} (C3);
      \draw[->,thick] (D4) -- (D3);
      
      \draw[->,thick] (C3) -- (D3);
    \end{tikzpicture}
\end{center}
from which we conclude
\[
  \Hom\big( \Tor_1(M,-), \Ext^1(M,-) \big) \simeq {R}/{\langle x,y \rangle}.
\]
\end{example}

%% file: IntroCCT_chap2.tex
\section{Constructive diagram chases}
Diagram chases are a powerful tool used in homological algebra 
for proving the existence of morphisms situated in some diagram of
prescribed shape.
In this section, we will demonstrate how to perform diagram chases constructively.
The main idea is to employ a calculus that replaces 
the morphisms in an abelian category $\AC$
with a more flexible notion, yielding a new category $\G( \AC )$,
analogous to the replacement of functions in the category of sets
with relations.
This idea has first been pursued in an axiomatic way by Brinkmann and Puppe in \cite{BPKorr} and \cite{PuppeKorr},
and rendered into an explicit calculus by Hilton in \cite{HiltonCorr}.
A calculus of relations in so-called regular categories,
which are more general than abelian categories, was given by Johnstone \cite{JohnstoneElephantI}.

The first algorithmic usage of this calculus in the context of spectral sequence computations is due to Barakat in \cite{barhabil}.
Here, the term \emph{generalized morphism} is coined for morphisms in $\G( \AC )$
and we will follow this convention. Other appropriate terms would be: relations, correspondences, or pseudo morphisms\footnote{Suggested by Jean Michel.}.

The presented material follows closely the presentation of generalized morphisms given in \cite{PosurDoktor},
especially Subsections \ref{subsection:gen_mors} and \ref{subsection:comp_rules}.
\subsection{Additive relations}
We start with the following diagram with exact rows
in the category of abelian groups $\Ab$:
\begin{center}
          \begin{tikzpicture}[transform shape,mylabel/.style={thick, draw=black, align=center, minimum width=0.5cm, minimum height=0.5cm,fill=white}]
            \coordinate (r) at (2.5,0);
            \coordinate (u) at (0,2);
            \node (A) {$A$};
            \node (B) at ($(A) + (r)$) {$B$};
            \node (C) at ($(B) + (r)$) {$C$};
            \node (0u) at ($(C) + (r)$) {$0$};
            \node (ker) at ($(C) + (u)$) {$\kernel( \gamma )$};
            \node (Ap) at ($(A) - (u)$) {$A'$};
            \node (Bp) at ($(Ap) + (r)$) {$B'$};
            \node (Cp) at ($(Bp) + (r)$) {$C'$};
            \node (0d) at ($(Ap) - (r)$) {$0$};
            
            \node (coker) at ($(Ap) - (u)$) {$\cokernel( \alpha )$};
            \draw[->,thick] (A) --node[above]{$\delta$} (B);
            \draw[->,thick] (B) --node[above]{$\epsilon$} (C);
            \draw[->,thick] (C) -- (0u);
            \draw[->,thick] (Ap) --node[above]{$\iota$} (Bp);
            \draw[->,thick] (Bp) --node[above]{$\nu$} (Cp);
            \draw[->,thick] (0d) -- (Ap);
            \draw[->,thick] (ker) --node[right]{$\eta\coloneqq \KernelEmbedding( \gamma )$} (C);
            \draw[->,thick] (Ap) --node[right]{$\zeta\coloneqq \CokernelProjection( \alpha )$} (coker);
            \draw[->,thick] (A) --node[right]{$\alpha$} (Ap);
            \draw[->,thick] (B) --node[right]{$\beta$} (Bp);
            \draw[->,thick] (C) --node[right]{$\gamma$} (Cp);
            
            \node (0A) at ($(A) - (r)$) {$0$};
            \node (0Cp) at ($(Cp) + (r)$) {$0$};
            \draw[->,thick] (0A) -- (A);
            \draw[->,thick] (Cp) -- (0Cp);
          \end{tikzpicture}
\end{center}
The famous snake lemma claims the existence of a morphism
\[
 \partial: \kernel( \gamma ) \longrightarrow \cokernel( \alpha )
\]
fitting into an exact sequence
\begin{center}
          \begin{tikzpicture}[transform shape,mylabel/.style={thick, draw=black, align=center, minimum width=0.5cm, minimum height=0.5cm,fill=white}]
            \coordinate (r) at (2.5,0);
            \coordinate (u) at (0,0.5);
            \node (A) {};
            \node (B) at ($(A) + (r)$) {};
            \node (C) at ($(B) + (r)$) {};
            \node (0u) at ($(C) + (r)$) {};
            \node (ker) at ($(C) + (u)$) {$\kernel( \gamma )$};
            \node (kerB) at ($(B) + (u)$) {$\kernel( \beta )$};
            \node (kerA) at ($(A) + (u)$) {$\kernel( \alpha )$};
            \node (Ap) at ($(A) - (u)$) {};
            \node (Bp) at ($(Ap) + (r)$) {};
            \node (Cp) at ($(Bp) + (r)$) {};
            \node (0d) at ($(Ap) - (r)$) {};
            
            \node (coker) at ($(Ap) - (u)$) {$\cokernel( \alpha )$};
            \node (cokerB) at ($(Bp) - (u)$) {$\cokernel( \beta )$};
            \node (cokerC) at ($(Cp) - (u)$) {$\cokernel( \gamma )$};
            
            \node (0A2) at ($(kerA) - (r)$) {$0$};
            \node (0Cp2) at ($(cokerC) + (r)$) {$0$};
            
            \draw[->,thick] (0A2) -- (kerA);
            \draw[->,thick] (kerA) -- (kerB);
            \draw[->,thick] (kerB) -- (ker);
            
            \draw[->,thick] (coker) -- (cokerB);
            \draw[->,thick] (cokerB) -- (cokerC);
            \draw[->,thick] (cokerC) -- (0Cp2);

            \draw (ker) [->,thick,out=-1,in=180-1,looseness=2] to node[mylabel]{$\partial$} (coker);
          \end{tikzpicture}
\end{center}
We will focus on the existence part of this lemma.
A description of $\partial$ can be given on the level of elements:
\begin{enumerate}
 \item Start with an element $c \in \kernel( \gamma )$.
 \item Regard it as an element $c \in C$.
 \item \textbf{Choose} an element $b \in \epsilon^{-1}(\{ c \})$.
 \item Map $b$ via $\beta$ and obtain $b' \coloneqq \beta( b ) \in B'$.
 \item Find the uniquely determined element $a' \in \iota^{-1}(\{b'\})$.
 \item Consider the residue class of $a'$ in $\cokernel( \alpha )$.
\end{enumerate}
It is quite easy to prove that each of these steps can actually be carried out
and that the resulting map
\[
 \kernel( \gamma ) \rightarrow \cokernel( \alpha ): c \mapsto a' + \image( \alpha )
\]
is a group homomorphism independent of the choice made in step $(3)$.

A common approach to prove the existence of $\partial$ not only in the category
of abelian groups but in every abelian category is to use embedding theorems \cite{Freyd}.
Such theorems reduce constructions in a small abelian category to the case of categories of modules
where one can happily perform element-wise constructions like the one we did above.

We are going to follow a more computer-friendly approach that will enable us to construct
$\partial$ only using operations within our given abelian category and without
passing to an ambient module category.
To see how this goal can be achieved, let us take a look at the most crucial
step within the construction of $\partial$ in the category of abelian groups above,
namely step $(3)$. It is highly uncanonical to choose just any preimage of $c$,
and in fact, every choice is just as good as every other choice.
A possible way to overcome this problem is by not making any choice at all, but to work
with the whole preimage $\epsilon^{-1}(\{ c \})$ instead.
Following this idea, the steps in the construction of $\partial$ above can be reformulated as follows:
\begin{enumerate}
 \item Start with an element $c \in \kernel( \gamma )$.
 \item Regard it as an element $c \in C$.
 \item Construct the whole preimage $b \coloneqq \epsilon^{-1}(\{ c \}) \subseteq B$.
 \item Construct the image $b' \coloneqq \beta( b ) \subseteq B'$.
 \item Construct the whole preimage $a' \coloneqq \iota^{-1}(\{b'\}) \subseteq A'$.
 \item Construct the image of $a'$ under the cokernel projection: $\{ x + \image( \alpha ) \mid x \in a' \}$.
       It will consist of a single element.
\end{enumerate}
We got rid of the uncanonical step in this set of instructions
and all we do is to take images and fibers of \emph{sets of elements} instead of single elements.
One possible way to formulate these new instructions in a more categorical way is given by replacing
the notion of a group homomorphism by the notion of an additive relation.

\begin{definition}
 An \textbf{additive relation} from an abelian group $A$ to an abelian group $B$
 is given by a subgroup $f \subseteq A \times B$.
\end{definition}

\begin{example}
 Every abelian group homomorphism $\alpha: A \rightarrow B$ in $\Ab$
 defines via its graph an additive relation
 \[
  [\alpha]\coloneqq \{ ( a, b ) \mid \alpha( a ) = b \} \subseteq A \times B.
 \]
\end{example}

\begin{example}
 If $f \subseteq A \times B$ is an additive relation, then so is
 its \textbf{pseudo-inverse}
 \[
  f^{-1} \coloneqq \{ ( b, a ) \mid (a,b) \in f \} \subseteq B \times A.
 \]
\end{example}

Additive relations $f \subseteq A \times B$ and $g \subseteq B \times C$
can be composed via
\[
 f \cdot g \coloneqq \{ ( a, c ) \mid \exists b \in B: (a,b) \in f, (b,c) \in g \} \subseteq A \times C.
\]
This composition turns abelian groups and additive relations into a category $\Rel( \Ab )$
with graphs of the identity group homomorphisms as its identities.
Mapping a group homomorphism to its graph lets us think of 
$\Ab$ as a non-full subcategory of $\Rel( \Ab )$.

Our reformulated set of instructions for computing $\partial$
can now conveniently be written as a simple composition of relations:
\[
 [ \partial ] = [ \eta ] \cdot [ \epsilon ]^{-1} \cdot [ \beta ] \cdot [ \iota ]^{-1} \cdot [ \zeta ].
\]
To sum it up, it can be said that performing constructions in $\Ab$ via diagram chases
boils down to calculations in $\Rel( \Ab )$.
Thus, it is our goal to find a calculus for working with relations 
in an arbitrary abelian category $\AC$.

\subsection{Category of generalized morphisms}\label{subsection:gen_mors}
From now on, we denote by $\AC$ an arbitrary abelian category.
Given two objects $A,B \in \AC$, a \textbf{span} $S$ (from $A$ to $B$)
is simply given by an object $C \in \AC$ together with
a pair of morphisms $(A \stackrel{\alpha}{\longleftarrow} C, C \stackrel{\beta}{\longrightarrow} B)$.
We depict a span as
\begin{center}
          \begin{tikzpicture}[transform shape,mylabel/.style={thick, draw=black, align=center, minimum width=0.5cm, minimum height=0.5cm,fill=white}]
            \coordinate (r) at (1.5,0);
            \coordinate (u) at (0,1.5);
            
            \node (A) {$A$};
            \node (B) at ($(A)+2*(r)$) {$B$};
            \node (C) at ($(A)+(r) -(u)$) {$C$};
            \draw[->,thick, dashed] (A) --node[mylabel]{$S$} (B);
            \draw[->,thick] (C) -- node[right]{$\beta$} (B);
            \draw[->,thick] (C) -- node[left]{$\alpha$} (A);
          \end{tikzpicture}
        \end{center}
  or as 
  \begin{center}
          \begin{tikzpicture}[transform shape,mylabel/.style={thick, draw=black, align=center, minimum width=0.5cm, minimum height=0.5cm,fill=white}]
            \coordinate (r) at (1.5,0);
            \coordinate (u) at (0,1.5);
            \node (A) {$A$};
            \node (B) at ($(A)+2*(r)$) {$B$.};
            \node (C) at ($(A)+(r)$) {$C$};
            \draw[->,thick] (C) -- node[above]{$\beta$} (B);
            \draw[->,thick] (C) -- node[above]{$\alpha$} (A);
          \end{tikzpicture}
  \end{center}
Note that we included a direction within our definition of a span in the sense that
swapping the order of the pair of morphisms defines a different span
(from $B$ to $A$).

\begin{definition}\label{definition:category_of_spans}
 The \bfindex{category of spans of $\AC$}, denoted by $\Span( \AC )$, is defined by the following data:
 \begin{enumerate}
  \item Objects are given by $\Obj_{\AC}$.
  \item Morphisms from $A$ to $B$ are spans from $A$ to $B$.
  \item Two spans 
        $(A \stackrel{\alpha}{\longleftarrow} C \stackrel{\beta}{\longrightarrow} B)$ 
        and $(A \stackrel{\alpha'}{\longleftarrow} C'  \stackrel{\beta'}{\longrightarrow}  B)$
        are considered to be \textbf{equal as spans} if 
        there exists an isomorphism $\iota: C \longrightarrow C'$ compatible with the spans,
        i.e., such that
        $
         \alpha = \iota \cdot \alpha'
        $
        and
        $
        \beta = \iota \cdot \beta'.
        $
  \item The identity of $A$ is given by $(A \stackrel{\id}{\longleftarrow} A \stackrel{\id}{\longrightarrow} A)$, 
  where $\id$ denotes the identity of $A$ regarded as an object in $\AC$.
  \item Composition of $(A \stackrel{\alpha}{\longleftarrow} D \stackrel{\beta}{\longrightarrow} B)$ 
        and $(B \stackrel{\gamma}{\longleftarrow} E  \stackrel{\delta}{\longrightarrow}  C)$
        is given by the outer span in the following diagram:
        \begin{center}
          \begin{tikzpicture}[label/.style={postaction={
            decorate,
            decoration={markings, mark=at position .5 with \node #1;}}}]
            \coordinate (r) at (1.5,0);
            \coordinate (u) at (0,1.5);
            
            \node (A) {$A$};
            \node (B) at ($(A)+2*(r)$) {$B$};
            \node (C) at ($(A)+4*(r)$) {$C$};
            \node (D) at ($(A) + (r) - (u)$) {$D$};
            \node (E) at ($(D)+2*(r)$) {$E$};
            \node (F) at ($(B)-2*(u)$) {$D \times_{B} E$};
            \draw[->,thick, dashed] (A) -- (B);
            \draw[->,thick, dashed] (B) -- (C);
            \draw[->,thick] (D) -- node[left]{$\alpha$} (A);
            \draw[->,thick] (D) -- node[right]{$\beta$} (B);
            \draw[->,thick] (E) -- node[left,inner sep = 0.5em]{$\gamma$} (B);
            \draw[->,thick] (E) -- node[right]{$\delta$} (C);
            \draw[->,thick] (F) -- node[left,inner sep = 0.5em]{$\gamma^{\ast}$} (D);
            \draw[->,thick] (F) -- node[right]{$\beta^{\ast}$} (E);
          \end{tikzpicture}
        \end{center}
 \end{enumerate}
\end{definition}

We have to check compatibility of composition and identities with our
notion of equality for spans.

\begin{lemma}\mbox{}
 \begin{enumerate}
  \item The identity in $\Span( \AC )$ acts like a unit up to equality of spans.
  \item Composition of morphisms in $\Span( \AC )$ is associative up to equality of spans.
 \end{enumerate}
\end{lemma}
\begin{proof}
 For the first assertion, let $(A \stackrel{\alpha}{\longleftarrow} D \stackrel{\beta}{\longrightarrow} B)$
 be a span. Composition with the identity $(B \stackrel{\id}{\longleftarrow} B \stackrel{\id}{\longrightarrow} B)$ from
 the right yields the diagram
 \begin{center}
          \begin{tikzpicture}[label/.style={postaction={
            decorate,
            decoration={markings, mark=at position .5 with \node #1;}}}]
            \coordinate (r) at (1.5,0);
            \coordinate (u) at (0,1.5);
            
            \node (A) {$A$};
            \node (B) at ($(A)+2*(r)$) {$B$};
            \node (C) at ($(A)+4*(r)$) {$B$};
            \node (D) at ($(A) + (r) - (u)$) {$D$};
            \node (E) at ($(D)+2*(r)$) {$B$};
            \node (F) at ($(B)-2*(u)$) {$D$};
            \draw[->,thick, dashed] (A) -- (B);
            \draw[->,thick, dashed] (B) -- (C);
            \draw[->,thick] (D) -- node[left]{$\alpha$} (A);
            \draw[->,thick] (D) -- node[right]{$\beta$} (B);
            \draw[->,thick] (E) -- node[left,inner sep = 0.5em]{$\id$} (B);
            \draw[->,thick] (E) -- node[right]{$\id$} (C);
            \draw[->,thick] (F) -- node[left,inner sep = 0.5em]{$\id$} (D);
            \draw[->,thick] (F) -- node[right]{$\beta$} (E);
          \end{tikzpicture}
        \end{center}
 This proves that the identity is a right unit. An analogous argument shows that it is also a left unit.
 For the second assertion, consider the following diagram of consecutive pullbacks:
 \begin{center}
          \begin{tikzpicture}
          [transform shape,mylabel/.style={thick, draw=black, align=center, minimum width=0.5cm, minimum height=0.5cm,fill=white}]
            \coordinate (r) at (1.5,0);
            \coordinate (u) at (0,1.5);
            
            \node (A) {$A$};
            \node (B) at ($(A)+2*(r)$) {$B$};
            \node (C) at ($(A)+4*(r)$) {$C$};
            \node (D) at ($(A)+6*(r)$) {$D$};
            \node (E) at ($(A)+(r)-(u)$) {$E$};
            \node (F) at ($(E) + 2*(r)$) {$F$};
            \node (G) at ($(F)+ 2*(r)$) {$G$};
            \node (H) at ($(B) - 2*(u)$) {$E \times_{B} F$};
            \node (I) at ($(H) + 2*(r)$) {$F \times_{C} G$};
            \node (J) at ($(F) -2*(u)$) {$(E \times_{B} F) \times_{F} (F \times_{C} G)$};
            \draw[->,thick, dashed] (A) --node[mylabel]{$S$} (B);
            \draw[->,thick, dashed] (B) --node[mylabel]{$T$} (C);
            \draw[->,thick, dashed] (C) --node[mylabel]{$U$} (D);
            \draw[->,thick] (E) --  (A);
            \draw[->,thick] (E) --  (B);
            \draw[->,thick] (F) --  (B);
            \draw[->,thick] (F) -- (C);
            \draw[->,thick] (G) --  (C);
            \draw[->,thick] (G) -- (D);
            \draw[->,thick] (H) -- (E);
            \draw[->,thick] (H) -- (F);
            \draw[->,thick] (I) -- (F);
            \draw[->,thick] (I) -- (G);
            \draw[->,thick] (J) -- (H);
            \draw[->,thick] (J) -- (I);
          \end{tikzpicture}
        \end{center}
  By transitivity of pullbacks, the rectangles with vertices $E,B,F \times_{C} G,(E \times_{B} F) \times_{F} (F \times_{C} G)$
  and $C,G,E \times_{B} F,(E \times_{B} F) \times_{F} (F \times_{C} G)$
  are also pullback squares. But this means that the outer span of the above diagram
  is isomorphic to both $S \cdot (T \cdot U)$ and $(S \cdot T) \cdot U$.
\end{proof}

\begin{definition}
 Given a span $(A \stackrel{\alpha}{\longleftarrow} C \stackrel{\beta}{\longrightarrow} B)$, we define its
 \bfindex{associated relation} as the image of the morphism
 \[
  (\alpha, \beta): C \longrightarrow A \oplus B.
 \]
 In particular, the associated relation of a span is a subobject of $A \oplus B$.
\end{definition}

\begin{definition}
 We say two spans from $A$ to $B$ are \bfindex{stably equivalent}
 if their associated relations are equal as subobjects of $A \oplus B$.
\end{definition}

\begin{remark}
 Being stably equivalent is coarser than being equal as spans.
\end{remark}

\begin{lemma}\label{lemma:epimorphism_stably_equivalent}
 Let $\epsilon: D \twoheadrightarrow C$ be an epimorphism in $\AC$.
 Every span of the form 
 \[
 (A \stackrel{\alpha}{\longleftarrow} C \stackrel{\beta}{\longrightarrow} B)
 \]
 is stably equivalent to the outer span in the diagram given by composition with $\epsilon$:
 \begin{center}
          \begin{tikzpicture}
            [transform shape,mylabel/.style={thick, draw=black, align=center, minimum width=0.5cm, minimum height=0.5cm,fill=white}]
            
            \coordinate (r) at (2,0);
            \coordinate (u) at (0,1);
            
            \node (A) {$A$};
            \node (B) at ($(A)+2*(r)$) {$B$};
            \node (C) at ($(A)-(u) + (r)$) {$C$};
            \node (D) at ($(C) - 2*(u)$) {$D$};
            \draw[->,thick, dashed] (A) -- (B);
            \draw[->,thick] (C) --node[below]{$\alpha$} (A);
            \draw[->,thick] (C) --node[below]{$\beta$} (B);
            \draw[->>,thick] (D) --node[mylabel]{$\epsilon$} (C);
            \draw (D) [->,thick, out = 150, in = 290] to (A);
            \draw (D) [->,thick, out = 30, in = 250] to (B);
          \end{tikzpicture}
        \end{center}
\end{lemma}
\begin{proof}
 We have $( \epsilon \cdot \alpha,  \epsilon \cdot \beta ) = \epsilon \cdot (\alpha, \beta)$, and in an
 abelian category, the image is not affected by epimorphisms.
 Thus, $\im\left( ( \epsilon \cdot \alpha,  \epsilon \cdot \beta ) \right) = \im\left( (\alpha, \beta) \right)$.
\end{proof}

\begin{theorem}\label{theorem:stably_equivalence}
 Being stably equivalent defines a congruence on $\Span( \AC )$.
\end{theorem}
\begin{proof}
 Let $S = (A \longleftarrow D \longrightarrow B)$ be a span
 and let $(\zeta, \eta): I \hookrightarrow B \oplus C$ be a monomorphism.
 Let $T = (B \longleftarrow E \longrightarrow C)$ be a span obtained by composing 
 $\zeta, \eta$ with an epimorphism $\epsilon: E \twoheadrightarrow I$.
 By transitivity of the pullback, we get $S \cdot T$ as the outer span in the following diagram:
 \begin{center}
          \begin{tikzpicture}
            [transform shape,mylabel/.style={thick, draw=black, align=center, minimum width=0.5cm, minimum height=0.5cm,fill=white}]
            
            \coordinate (r) at (1.5,0);
            \coordinate (u) at (0,1.2);
            
            \node (A) {$A$};
            \node (B) at ($(A)+2*(r)$) {$B$};
            \node (C) at ($(A)+ 4*(r)$) {$C$};
            \node (D) at ($(A) + (r) - (u)$) {$D$};
            \node (I) at ($(D) + 2*(r)$) {$I$};
            \node (F) at ($(D) + (r) - (u)$) {$D \times_{B} I$};
            \node (P) at ($(F) - 2*(u)$) {$(D \times_{B} I) \times_{I} E$};
            \node (E) at ($(I) - 2*(u)$) {$E$};
            \draw[->,thick, dashed] (A) --node[mylabel]{$S$} (B);
            \draw[->,thick, dashed] (B) -- (C);
            \draw[->,thick] (D) -- (A);
            \draw[->,thick] (D) -- (B);
            \draw[->,thick] (I) --node[below]{$\zeta$} (B);
            \draw[->,thick] (I) --node[below]{$\eta$} (C);
            \draw[->,thick] (F) -- (D);
            \draw[->>,thick] (E) --node[mylabel]{$\epsilon$} (I);
            \draw[->,thick] (F) -- (I);
            \draw[->>,thick] (P) --node[mylabel]{$\epsilon^{\ast}$} (F);
            \draw[->,thick] (P) -- (E);
            \draw (E) [->,thick, out = 30, in = 270] to (C);
            \draw (P) [->,thick, out = 150, in = 270] to (A);
          \end{tikzpicture}
        \end{center}
  In an abelian category the pullback of an epimorphism yields an epimorphism. Thus, $\epsilon^{\ast}$ is an epimorphism.
  Now, we apply Lemma \ref{lemma:epimorphism_stably_equivalent} to see that the
  stable equivalence class of $S \cdot T$ only depends on $(\zeta, \eta)$, which is the associated relation of $T$.
  Thus, if $T$ and $T'$ have the same associated relation, i.e., are stably equivalent, then so are $S \cdot T$ and $S \cdot T'$.
  By the symmetry of the situation, a similar statement holds for stably equivalent $S$, $S'$ and compositions $S \cdot T$, $S' \cdot T$.
  This shows the claim.
\end{proof}

Due to Theorem \ref{theorem:stably_equivalence}, we can now define the generalized morphism category.

\begin{definition}\label{definition:generalized_morphism_category}
 Let $\AC$ be an abelian category. The quotient category of $\Span( \AC )$
 modulo stable equivalences is called the \bfindex{generalized morphism category of $\AC$},
 and denoted by $\G( \AC )$.
 Concretely, it consists of the following data:
 \begin{enumerate}
  \item Objects are given by $\Obj_{\AC}$.
  \item Morphisms from $A$ to $B$ are spans from $A$ to $B$.
  \item Two spans are considered to be \textbf{equal as generalized morphisms} if and only if they are stably equivalent.
  \item Identity and composition are given as in Definition \ref{definition:category_of_spans}.
 \end{enumerate}
 We call a span from $A$ to $B$ a \bfindex{generalized morphism} when we regard it as a morphism in $\G( \AC )$.
\end{definition}

\subsection{Computation rules}\label{subsection:comp_rules}

We will see that computing within $\G( \AC )$ boils down to
computing compositions of morphisms and pseudo-inverses of morphisms in $\AC$.
Every morphism $\alpha: A \rightarrow B$ in $\AC$ gives rise to a morphism 
\begin{center}
 \begin{tikzpicture}[label/.style={postaction={
            decorate,
            decoration={markings, mark=at position .5 with \node #1;}}},
            mylabel/.style={thick, draw=black, align=center, minimum width=0.5cm, minimum height=0.5cm,fill=white}]
            \coordinate (r) at (2,0);
            \coordinate (u) at (0,1.3);
            
            \node (A) {$A$};
            \node (B) at ($(A)+2*(r)$) {$B$};
            \node (C) at ($(A)+(r) -(u)$) {$A$};
            \draw[->,thick, dashed] (A) --node[mylabel]{$[\alpha]$} (B);
            \draw[->,thick] (C) -- node[below]{$\alpha$} (B);
            \draw[->,thick] (C) -- node[below]{$\id_A$}(A);
 \end{tikzpicture}
\end{center}
in $\G( \AC )$.
Since the pullback of the identity can again be chosen as the identity, we actually have a functor
\[
 [-]: \AC \longrightarrow \G( \AC ).
\]
Moreover, assume that we have
$[\alpha] = [\alpha']$ for a given pair
$\alpha, \alpha': A \rightarrow B$.
Since the morphisms $(1, \alpha): A \longrightarrow A \oplus B$
and $(1, \alpha'): A \longrightarrow A \oplus B$ are monos,
it follows that $\alpha = \alpha'$. Thus, our functor $[-]$ is faithful,
and we can regard $\AC$ as a subcategory of $\G( \AC )$.
Any morphism in $\G(\AC)$ which is equal to a morphism of the form $[\alpha]$ 
for $\alpha \in \AC$ is called \textbf{honest}.

The most prominent feature of $\G( \AC )$ is the operation
of taking pseudo-inverses.

\begin{definition}\label{definition:pseudo_inverse}
 For a span $S = (A \stackrel{\alpha}{\longleftarrow} C \stackrel{\beta}{\longrightarrow} B)$ from $A$ to $B$, we call
 the span $(B \stackrel{\beta}{\longleftarrow} C \stackrel{\alpha}{\longrightarrow} A)$ from $B$ to $A$ its \bfindex{pseudo-inverse}
 and denote it by $S^{-1}$.
 \begin{center}
 \begin{tabular}{ccccccc}
          \begin{tikzpicture}[label/.style={postaction={
            decorate,
            decoration={markings, mark=at position .5 with \node #1;}}}, baseline = (D),
            mylabel/.style={thick, draw=black, align=center, minimum width=0.5cm, minimum height=0.5cm,fill=white}]
            \coordinate (r) at (2,0);
            \coordinate (u) at (0,1.3);
            
            \node (A) {$A$};
            \node (B) at ($(A)+2*(r)$) {$B$};
            \node (C) at ($(A)+(r) -(u)$) {$C$};
            \node (D) at ($(A) - 0.5*(u)$) {};
            \draw[->,thick, dashed] (A) --node[mylabel]{$S$} (B);
            \draw[->,thick] (C) -- node[below]{$\beta$} (B);
            \draw[->,thick] (C) -- node[below]{$\alpha$} (A);
          \end{tikzpicture}
        & & & $\longleftrightarrow$ & & &

          \begin{tikzpicture}[label/.style={postaction={
            decorate,
            decoration={markings, mark=at position .5 with \node #1;}}}, baseline = (D),
            mylabel/.style={thick, draw=black, align=center, minimum width=0.5cm, minimum height=0.5cm,fill=white}]
            \coordinate (r) at (2,0);
            \coordinate (u) at (0,1.3);
            
            \node (A) {$B$};
            \node (B) at ($(A)+2*(r)$) {$A$};
            \node (C) at ($(A)+(r) -(u)$) {$C$};
            \node (D) at ($(A) - 0.5*(u)$) {};
            \draw[->,thick, dashed] (A) --node[mylabel]{$S^{-1}$} (B);
            \draw[->,thick] (C) -- node[below]{$\alpha$} (B);
            \draw[->,thick] (C) -- node[below]{$\beta$} (A);
          \end{tikzpicture}
 \end{tabular}
 \end{center}
\end{definition}

\begin{remark}\label{remark:taking_pseudo_inverses_as_a_functor}
 Taking pseudo-inverses is compatible with stable equivalences.
 Thus, it defines an equivalence of categories
 \[
  (-)^{-1}: \G( \AC )^{\mathrm{op}} \rightarrow \G( \AC ).
 \]
\end{remark}

Now, we show that we may represent every generalized morphism
as a composition of a pseudo-inverse of an honest morphism
with another honest morphism.

\begin{lemma}\label{lemma:symbolic_computation_for_generalized_morphisms}
 Every span $(A \stackrel{\alpha}{\longleftarrow} C \stackrel{\beta}{\longrightarrow} B)$ is equal
 to $[\alpha]^{-1} \cdot [\beta]$
 as generalized morphisms.
\end{lemma}
\begin{proof}
 A square consisting of identities is a pullback square.
 Thus, we have an equation of generalized morphisms (even as spans):
 \begin{center}
 \begin{tabular}{ccccccc}
          \begin{tikzpicture}[label/.style={postaction={
            decorate,
            decoration={markings, mark=at position .5 with \node #1;}}}, baseline = (D)]
            \coordinate (r) at (1.5,0);
            \coordinate (u) at (0,1.5);
            \node (A) {$A$};
            \node (base) at ($(A) - 0.5*(u)$) {};
            \node (B) at ($(A)+2*(r)$) {$C$};
            \node (C) at ($(A)+4*(r)$) {$B$};
            \node (D) at ($(A) + (r) - (u)$) {$C$};
            \node (E) at ($(D)+2*(r)$) {$C$};
            \node (F) at ($(B)-2*(u)$) {$C$};
            \draw[->,thick, dashed] (A) -- (B);
            \draw[->,thick, dashed] (B) -- (C);
            \draw[->,thick] (D) -- node[left]{$\alpha$} (A);
            \draw[->,thick] (D) -- node[left,inner sep = 0.5em]{$\id$} (B);
            \draw[->,thick] (E) -- node[right,inner sep = 0.5em]{$\id$} (B);
            \draw[->,thick] (E) -- node[right]{$\beta$} (C);
            \draw[->,thick] (F) -- node[left,inner sep = 0.5em]{$\id$} (D);
            \draw[->,thick] (F) -- node[right]{$\id$} (E);
          \end{tikzpicture}
        & & & $=$ & & &
          \begin{tikzpicture}[label/.style={postaction={
            decorate,
            decoration={markings, mark=at position .5 with \node #1;}}}, baseline = (base)]
            \coordinate (r) at (1.5,0);
            \coordinate (u) at (0,1.5);
            
            \node (A) {$A$};
            \node (base) at ($(A) - 0.5*(u)$) {};
            \node (B) at ($(A)+2*(r)$) {$B$};
            \node (D) at ($(A) + (r) - (u)$) {$C$};
            \draw[->,thick, dashed] (A) -- (B);
            \draw[->,thick] (D) -- node[left]{$\alpha$} (A);
            \draw[->,thick] (D) -- node[right]{$\beta$} (B);    
          \end{tikzpicture}
 \end{tabular}
 \end{center}
\end{proof}

\begin{theorem}\label{theorem:epi_mono_split}
 Given a mono $\iota$ in $\AC$, then $[\iota]$ is split in $\G( \AC )$
 with its pseudo-inverse as a retraction.
 Dually, 
 given an epi $\epsilon$ in $\AC$, then $[\epsilon]$ is split in $\G(\AC)$
 with its pseudo-inverse as a section.
\end{theorem}
\begin{proof}
 The composition of $[\iota]$ with $[\iota]^{-1}$ yields the diagram
 \begin{center}
          \begin{tikzpicture}[label/.style={postaction={
            decorate,
            decoration={markings, mark=at position .5 with \node #1;}}}]
            \coordinate (r) at (1.5,0);
            \coordinate (u) at (0,1.5);
            
            \node (A) {$A$};
            \node (B) at ($(A)+2*(r)$) {$B$};
            \node (C) at ($(A)+4*(r)$) {$A$};
            \node (D) at ($(A) + (r) - (u)$) {$A$};
            \node (E) at ($(D)+2*(r)$) {$A$};
            \node (F) at ($(B)-2*(u)$) {$A$};
            \draw[->,thick, dashed] (A) -- (B);
            \draw[->,thick, dashed] (B) -- (C);
            \draw[->,thick] (D) -- node[left,yshift= -0.1em]{$\id$} (A);
            \draw[right hook->,thick] (D) -- node[right]{$\iota$} (B);
            \draw[right hook->,thick] (E) -- node[left,inner sep = 0.5em]{$\iota$} (B);
            \draw[->,thick] (E) -- node[right,xshift = 0em,yshift= -0.1em]{$\id$} (C);
            \draw[->,thick] (F) -- node[left,yshift = -0.1em]{$\id$} (D);
            \draw[->,thick] (F) -- node[right,yshift = -0.1em]{$\id$} (E);
          \end{tikzpicture}
        \end{center}
  The dual statement can be proved analogously.
\end{proof}

\begin{corollary}\label{corollary:liftcolift_formula}
 Given a commutative diagram
 \begin{center}
    \begin{tikzpicture}[label/.style={postaction={
            decorate,
            decoration={markings, mark=at position .5 with \node #1;}},
            mylabel/.style={thick, draw=none, align=center, minimum width=0.5cm, minimum height=0.5cm,fill=white}}]
            \coordinate (r) at (3.5,0);
            \coordinate (u) at (0,2);
            \node (A) {$A$};
            \node (B) at ($(A)+(r)$) {$B$};
            \node (C) at ($(A)-(u)$) {$C$};
            \node (D) at ($(C)+(r)$) {$D$};
            
            \draw[->,thick] (A) --node[above]{$\alpha$} (B);
            \draw[->,thick] (C) --node[above]{$\gamma$} (D);
            
            \draw[->>,thick] (C) --node[left]{$\epsilon$} (A);
            \draw[left hook->,thick] (B) --node[right]{$\iota$} (D);
    \end{tikzpicture}
  \end{center}
 in $\AC$ with $\epsilon$ epi and $\iota$ mono,
 we get a commutative diagram 
 \begin{center}
    \begin{tikzpicture}[label/.style={postaction={
            decorate,
            decoration={markings, mark=at position .5 with \node #1;}},
            mylabel/.style={thick, draw=none, align=center, minimum width=0.5cm, minimum height=0.5cm,fill=white}}]
            \coordinate (r) at (3.5,0);
            \coordinate (u) at (0,2);
            \node (A) {$A$};
            \node (B) at ($(A)+(r)$) {$B$};
            \node (C) at ($(A)-(u)$) {$C$};
            \node (D) at ($(C)+(r)$) {$D$};
            
            \draw[->,thick] (A) --node[above]{$[\alpha]$} (B);
            \draw[->,thick] (C) --node[above]{$[\gamma]$} (D);
            
            \draw[left hook->,thick] (A) --node[left]{$[\epsilon]^{-1}$} (C);
            \draw[->>,thick] (D) --node[right]{$[\iota]^{-1}$} (B);
    \end{tikzpicture}
  \end{center}
  in $\G( \AC )$, i.e., the equation
 \[
  [ \alpha ] = [ \epsilon ]^{-1} \cdot [ \gamma ] \cdot [ \iota ]^{-1}
 \]
 holds.
\end{corollary}
\begin{proof}
 We simply multiply the equation
 \[
  [ \epsilon ] \cdot [ \alpha ] \cdot [ \iota ] =  [ \gamma ]
 \]
 from the left with $[ \epsilon ]^{-1}$ and from the right with $[ \iota ]^{-1}$.
 Then we apply Theorem \ref{theorem:epi_mono_split}.
\end{proof}

\begin{theorem}\label{theorem:computation_rules}
 Given a pullback diagram
 \begin{center}
          \begin{tikzpicture}
          [transform shape,mylabel/.style={thick, draw=black, align=center, minimum width=0.5cm, minimum height=0.5cm,fill=white}]
            \coordinate (r) at (1.5,0);
            \coordinate (u) at (0,1.3);
            \node (B) {$B$};
            \node (D) at ($(B) - (r) - (u)$) {$A$};
            \node (E) at ($(D)+2*(r)$) {$C$};
            \node (F) at ($(B)-2*(u)$) {$A \times_{B} C$};
            \draw[->,thick] (D) -- node[left, inner sep = 0.5em]{$\alpha$} (B);
            \draw[->,thick] (E) -- node[right,inner sep = 0.5em]{$\gamma$} (B);
            \draw[->,thick] (F) -- node[left,yshift=-0.3em]{$\gamma^{\ast}$} (D);
            \draw[->,thick] (F) -- node[right,yshift=-0.3em]{$\alpha^{\ast}$} (E);
          \end{tikzpicture}
\end{center}
 the \bfindex{pullback computation rule}
 \[
 [\alpha] \cdot [\gamma]^{-1} = [\gamma^{\ast}]^{-1} \cdot [\alpha^{\ast}]  
 \]
 holds. Dually, given a pushout square
 \begin{center}
          \begin{tikzpicture}
          [transform shape,mylabel/.style={thick, draw=black, align=center, minimum width=0.5cm, minimum height=0.5cm,fill=white}]
            \coordinate (r) at (1.5,0);
            \coordinate (u) at (0,1.3);
            
            \node (A) {$A \amalg_B C$};
            \node (B) at ($(A)-(u)-(r)$) {$A$};
            \node (C) at ($(B) + 2*(r)$) {$C$};
            \node (D) at ($(A)-2*(u)$) {$B$};
            
            \draw[->,thick] (D) -- node[right,yshift=-0.3em]{$\gamma$} (C);
            \draw[->,thick] (D) -- node[left,yshift=-0.3em]{$\alpha$} (B);
            \draw[->,thick] (C) -- node[right]{$\alpha_{\ast}$} (A);
            \draw[->,thick] (B) -- node[left]{$\gamma_{\ast}$} (A);
          \end{tikzpicture}
\end{center}
the \bfindex{pushout computation rule}
\[
 [\alpha]^{-1} \cdot [\gamma] = [\gamma_{\ast}] \cdot [\alpha_{\ast}]^{-1}
\]
holds.
\end{theorem}
\begin{proof}
 From the diagram
\begin{center}
          \begin{tikzpicture}
          [transform shape,mylabel/.style={thick, draw=black, align=center, minimum width=0.5cm, minimum height=0.5cm,fill=white}]
            \coordinate (r) at (1.5,0);
            \coordinate (u) at (0,1.5);
            
            \node (A) {$A$};
            \node (B) at ($(A)+2*(r)$) {$B$};
            \node (C) at ($(A)+4*(r)$) {$C$};
            \node (D) at ($(A) + (r) - (u)$) {$A$};
            \node (E) at ($(D)+2*(r)$) {$C$};
            \node (F) at ($(B)-2*(u)$) {$A \times_{B} C$};
            \draw[->,thick, dashed] (A) --node[mylabel]{$[\alpha]$} (B);
            \draw[->,thick, dashed] (B) --node[mylabel]{$[\gamma]^{-1}$} (C);
            \draw[->,thick] (D) -- node[left]{$\id_A$} (A);
            \draw[->,thick] (D) -- node[right]{$\alpha$} (B);
            \draw[->,thick] (E) -- node[left,inner sep = 0.5em]{$\gamma$} (B);
            \draw[->,thick] (E) -- node[right,inner sep = 0.5em]{$\id_C$} (C);
            \draw[->,thick] (F) -- node[left,inner sep = 0.5em]{$\gamma^{\ast}$} (D);
            \draw[->,thick] (F) -- node[right]{$\alpha^{\ast}$} (E);
          \end{tikzpicture}
\end{center}
and Lemma \ref{lemma:symbolic_computation_for_generalized_morphisms}, 
we get the pullback computation rule.

Next, we consider the situation for the pushout computation rule.
Let \[\alpha_{\ast}^{\ast}: A \times_{A \amalg_B C} C \rightarrow A\]
 and \[\gamma_{\ast}^{\ast}: A \times_{A \amalg_B C} C \rightarrow C\]
 be the pullback projections of $\gamma_{\ast}$, $\alpha_{\ast}$:
 \begin{center}
          \begin{tikzpicture}
          [transform shape,mylabel/.style={thick, draw=black, align=center, minimum width=0.5cm, minimum height=0.5cm,fill=white}]
            \coordinate (r) at (1.5,0);
            \coordinate (u) at (0,1.3);
            
            \node (A) {$A \amalg_B C$};
            \node (B) at ($(A)-(u)-(r)$) {$A$};
            \node (C) at ($(B) + 2*(r)$) {$C$};
            \node (D) at ($(A)-2*(u)$) {$A \times_{A \amalg_B C} C$.};
            
            \draw[->,thick] (D) -- node[right,yshift=-0.3em]{$\gamma_{\ast}^{\ast}$} (C);
            \draw[->,thick] (D) -- node[left,yshift=-0.3em]{$\alpha_{\ast}^{\ast}$} (B);
            \draw[->,thick] (C) -- node[right]{$\alpha_{\ast}$} (A);
            \draw[->,thick] (B) -- node[left]{$\gamma_{\ast}$} (A);
          \end{tikzpicture}
\end{center}
 By the pullback computation rule, we have
 \[
  [\gamma_{\ast}] \cdot [\alpha_{\ast}]^{-1}  =  [\alpha_{\ast}^{\ast}]^{-1} \cdot [\gamma_{\ast}^{\ast}].
 \]
 But taking pushout followed by taking pullback yields a monomorphism
 \[
  ( \alpha_{\ast}^{\ast}, \gamma_{\ast}^{\ast} ): A \times_{A \amalg_B C} C \longrightarrow A \oplus C
 \]
 which identifies with
the image embedding
 of the morphism
 \[
  ( \alpha, \gamma ): B \longrightarrow A \oplus C,
 \]
 since images in abelian categories are defined as the kernel embeddings of cokernel projections.
 It follows that 
 \[
  [\alpha_{\ast}^{\ast}]^{-1} \cdot [\gamma_{\ast}^{\ast}]  = [\alpha]^{-1} \cdot [\gamma].
 \]
\end{proof}

\subsection{Cohomology}

Generalized morphisms are a convenient tool to write
down closed formulas for morphisms whose existence
is induced by some prescribed diagram.
We demonstrate this principle by means of a standard
example in homological algebra, namely the
induced morphism on cohomology.

\begin{theorem}
 Suppose given a commutative diagram in $\AC$ of the following form:
 \begin{center}
    \begin{tikzpicture}[label/.style={postaction={
            decorate,
            decoration={markings, mark=at position .5 with \node #1;}},
            mylabel/.style={thick, draw=none, align=center, minimum width=0.5cm, minimum height=0.5cm,fill=white}}]
            \coordinate (r) at (2.5,0);
            \coordinate (u) at (0,1.5);
            \node (A) {$A$};
            \node (B) at ($(A)+(r)$) {$B$};
            \node (C) at ($(B)+(r)$) {$C$};
            
            \node (Ap) at ($(A)-(u)$) {$A'$};
            \node (Bp) at ($(Ap)+(r)$) {$B'$};
            \node (Cp) at ($(Bp)+(r)$) {$C'$};
            
            \node (KB) at ($(B)+(u)$) {$\kernel( d_B )$};
            \node (HB) at ($(KB)+(u)$) {$\frac{\kernel(d_B)}{\image(d_A)}$};
            
            \node (KBp) at ($(Bp)-(u)$) {$\kernel( d_{B'} )$};
            \node (HBp) at ($(KBp)-(u)$) {$\frac{\kernel(d_{B'})}{\image(d_{A'})}$};
            
            \draw[->,thick] (A) --node[above]{$d_A$} (B);
            \draw[->,thick] (B) --node[above]{$d_B$} (C);
            
            \draw[->,thick] (Ap) --node[above]{$d_{A'}$} (Bp);
            \draw[->,thick] (Bp) --node[above]{$d_{B'}$} (Cp);
            
            \draw[left hook->,thick] (KB) --node[left]{$\iota_B$} (B);
            \draw[->>,thick] (KB) --node[left]{$\epsilon_B$} (HB);
            
            \draw[left hook->,thick] (KBp) --node[left]{$\iota_{B'}$} (Bp);
            \draw[->>,thick] (KBp) --node[left]{$\epsilon_{B'}$} (HBp);
            
            \draw[->,thick] (A) -- (Ap);
            \draw[->,thick] (B) --node[left]{$\beta$} (Bp);
            \draw[->,thick] (C) -- (Cp);
            
    \end{tikzpicture}
  \end{center}
  where we have $\im( d_A ) \subseteq \kernel( d_B )$,
  $\im( d_{A'} ) \subseteq \kernel( d_{B'} )$,
  and $\iota_B, \iota_{B'}$ are the kernel embeddings,
  and $\epsilon_B$, $\epsilon_{B'}$ are the natural projections.
  Then the induced morphism on cohomologies 
  \[
   \frac{\kernel(d_B)}{\image(d_A)} \longrightarrow \frac{\kernel(d_{B'})}{\image(d_{A'})}
  \]
  is given by the following composition of generalized morphisms:
  \[
   [ \epsilon_B ]^{-1} \cdot [ \iota_B ] \cdot [ \beta ] \cdot [ \iota_{B'} ]^{-1} \cdot [ \epsilon_{B'} ].
  \]
\end{theorem}
\begin{proof}
The induced morphism on cohomologies is constructed by the cokernel functor applied to the commutative square
\begin{center}
    \begin{tikzpicture}[label/.style={postaction={
            decorate,
            decoration={markings, mark=at position .5 with \node #1;}},
            mylabel/.style={thick, draw=none, align=center, minimum width=0.5cm, minimum height=0.5cm,fill=white}}]
            \coordinate (r) at (2.5,0);
            \coordinate (u) at (0,1.5);
            \node (A) {$\image( d_A )$};
            \node (B) at ($(A)+(r)$) {$\kernel( d_B )$};
            
            \node (Ap) at ($(A)-(u)$) {$\image( d_{A'} )$};
            \node (Bp) at ($(Ap)+(r)$) {$\kernel( d_{B'} )$};
            
            \draw[right hook->,thick] (A) -- (B);
            
            \draw[right hook->,thick] (Ap) -- (Bp);
            \draw[->,thick] (A) -- (Ap);
            \draw[->,thick] (B) -- (Bp);
            
    \end{tikzpicture}
  \end{center}
which itself is defined by restricting $\beta$.
Thus, we have a commutative diagram
\begin{center}
    \begin{tikzpicture}[label/.style={postaction={
            decorate,
            decoration={markings, mark=at position .5 with \node #1;}},
            mylabel/.style={thick, draw=none, align=center, minimum width=0.5cm, minimum height=0.5cm,fill=white}}]
            \coordinate (r) at (2.5,0);
            \coordinate (u) at (0,2);
            \node (A) {$B$};
            \node (B) at ($(A)+(r)$) {$\kernel( d_B )$};
            \node (H) at ($(B)+(r)$) {$\frac{\kernel(d_B)}{\image(d_A)}$};
            
            \node (Ap) at ($(A)-(u)$) {$B'$};
            \node (Bp) at ($(Ap)+(r)$) {$\kernel( d_{B'} )$};
            \node (Hp) at ($(Bp)+(r)$) {$\frac{\kernel(d_{B'})}{\image(d_{A'})}$};
            \draw[left hook->,thick] (B) --node[above]{$\iota_B$} (A);
            
            \draw[left hook->,thick] (Bp) --node[above]{$\iota_{B'}$} (Ap);
            \draw[->,thick] (A) --node[left]{$\beta$} (Ap);
            \draw[->,thick] (B) --node[left]{$\gamma$} (Bp);
            \draw[->>,thick] (B) --node[above]{$\epsilon_{B}$} (H);
            \draw[->>,thick] (Bp) --node[above]{$\epsilon_{B'}$} (Hp);
            \draw[->,thick,dashed] (H) --node[left]{$\delta$} (Hp);
    \end{tikzpicture}
  \end{center}
  where the dashed arrow $\delta$ is the induced morphism on cohomologies.
  
  Now, since $\epsilon_B$ is an epi, by Corollary \ref{corollary:liftcolift_formula} we have
  \[
   [\delta] = [ \epsilon_B ]^{-1} \cdot [ \gamma ] \cdot [ \epsilon_{B'} ].
  \]
  Moreover, since $\iota_{B'}$ is a mono, by Corollary \ref{corollary:liftcolift_formula} we have
  \[
   [\gamma] = [ \iota_B ] \cdot [ \beta ] \cdot [ \iota_{B'} ]^{-1}.
  \]
  Substituting the latter formula in the former yields the claim.
\end{proof}

\subsection{Snake lemma}

The induced morphism in the famous snake lemma
can also be constructed as a composition
of the obvious generalized morphisms.
For seeing this, we analyze the construction of
the snake following \cite{MLCWM} in the light of the theory of generalized morphisms.

The starting point of the snake lemma is a commutative diagram in $\AC$ with exact rows:
\begin{center}
          \begin{tikzpicture}[transform shape,mylabel/.style={thick, draw=black, align=center, minimum width=0.5cm, minimum height=0.5cm,fill=white}]
            \coordinate (r) at (2.5,0);
            \coordinate (u) at (0,2);
            \node (A) {$A$};
            \node (B) at ($(A) + (r)$) {$B$};
            \node (C) at ($(B) + (r)$) {$C$};
            \node (0u) at ($(C) + (r)$) {$0$};
            \node (ker) at ($(C) + (u)$) {$\kernel( \gamma )$};
            \node (Ap) at ($(A) - (u)$) {$A'$};
            \node (Bp) at ($(Ap) + (r)$) {$B'$};
            \node (Cp) at ($(Bp) + (r)$) {$C'$};
            \node (0d) at ($(Ap) - (r)$) {$0$};
            
            \node (coker) at ($(Ap) - (u)$) {$\cokernel( \alpha )$};
            \draw[->,thick] (A) --node[above]{$\delta$} (B);
            \draw[->,thick] (B) --node[above]{$\epsilon$} (C);
            \draw[->,thick] (C) -- (0u);
            \draw[->,thick] (Ap) --node[above]{$\iota$} (Bp);
            \draw[->,thick] (Bp) --node[above]{$\nu$} (Cp);
            \draw[->,thick] (0d) -- (Ap);
            \draw[->,thick] (ker) --node[right]{$\eta\coloneqq \KernelEmbedding( \gamma )$} (C);
            \draw[->,thick] (Ap) --node[right]{$\zeta\coloneqq \CokernelProjection( \alpha )$} (coker);
            \draw[->,thick] (A) --node[right]{$\alpha$} (Ap);
            \draw[->,thick] (B) --node[right]{$\beta$} (Bp);
            \draw[->,thick] (C) --node[right]{$\gamma$} (Cp);
            
            \node (0A) at ($(A) - (r)$) {$0$};
            \node (0Cp) at ($(Cp) + (r)$) {$0$};
            \draw[->,thick] (0A) -- (A);
            \draw[->,thick] (Cp) -- (0Cp);
          \end{tikzpicture}
\end{center}
In \cite{MLCWM}, Mac Lane constructs the snake morphism
\[
 \delta: \kernel( \gamma ) \longrightarrow \cokernel( \alpha )
\]
by first computing the pullback
\begin{center}
    \begin{tikzpicture}[label/.style={postaction={
            decorate,
            decoration={markings, mark=at position .5 with \node #1;}},
            mylabel/.style={thick, draw=none, align=center, minimum width=0.5cm, minimum height=0.5cm,fill=white}}]
            \coordinate (r) at (3.5,0);
            \coordinate (u) at (0,2);
            \node (A) {$\kernel( \gamma ) \times_C B$};
            \node (B) at ($(A)+(r)$) {$\kernel( \gamma )$};
            \node (C) at ($(A)-(u)$) {$B$};
            \node (D) at ($(C)+(r)$) {$C$};
            
            \draw[->,thick] (A) --node[above]{$\epsilon^{\ast}$} (B);
            \draw[->,thick] (C) --node[above]{$\epsilon$} (D);
            
            \draw[left hook->,thick] (A) --node[left]{$\eta^{\ast}$} (C);
            \draw[left hook->,thick] (B) --node[right]{$\eta$} (D);
    \end{tikzpicture}
\end{center}
and pushout
\begin{center}
    \begin{tikzpicture}[label/.style={postaction={
            decorate,
            decoration={markings, mark=at position .5 with \node #1;}},
            mylabel/.style={thick, draw=none, align=center, minimum width=0.5cm, minimum height=0.5cm,fill=white}}]
            \coordinate (r) at (3.5,0);
            \coordinate (u) at (0,2);
            \node (A) {$A'$};
            \node (B) at ($(A)+(r)$) {$B'$};
            \node (C) at ($(A)-(u)$) {$\cokernel( \alpha )$};
            \node (D) at ($(C)+(r)$) {$\cokernel( \alpha ) \amalg_{A'} B'$};
            
            \draw[->,thick] (A) --node[above]{$\iota$} (B);
            \draw[->,thick] (C) --node[above]{$\iota_{\ast}$} (D);
            
            \draw[->>,thick] (A) --node[left]{$\zeta$} (C);
            \draw[->>,thick] (B) --node[right]{$\zeta_{\ast}$} (D);
    \end{tikzpicture}
\end{center}
and second proving the existence of 
a unique morphism $\delta$ rendering the diagram
\begin{center}
    \begin{tikzpicture}[label/.style={postaction={
            decorate,
            decoration={markings, mark=at position .5 with \node #1;}},
            mylabel/.style={thick, draw=none, align=center, minimum width=0.5cm, minimum height=0.5cm,fill=white}}]
            \coordinate (r) at (6,0);
            \coordinate (u) at (0,2);
            \node (A) {$\kernel( \gamma )$};
            \node (B) at ($(A)+(r)$) {$\cokernel(\alpha)$};
            \node (C) at ($(A)-(u)$) {$\kernel( \gamma ) \times_C B$};
            \node (D) at ($(C)+(r)$) {$\cokernel( \alpha ) \amalg_{A'} B'$};
            
            \draw[->,thick,dashed] (A) --node[above]{$\delta$} (B);
            \draw[->,thick] (C) --node[above]{$\eta^{\ast} \cdot \beta \cdot \zeta_{\ast}$} (D);
            
            \draw[->>,thick] (C) --node[left]{$\epsilon^{\ast}$} (A);
            \draw[left hook->,thick] (B) --node[right]{$\iota_{\ast}$} (D);
    \end{tikzpicture}
  \end{center}
commutative.

Analyzing this process in the light of generalized morphisms,
the first step of taking the pullback/pushout 
can be interpreted as rewriting the generalized morphisms
\begin{equation}\label{equation:term1}
 [ \eta ] \cdot [ \epsilon ]^{-1} = [ \epsilon^{\ast} ]^{-1} \cdot [ \eta^{\ast} ]
\end{equation}
and
\begin{equation}\label{equation:term2}
 [ \iota ]^{-1} \cdot [ \zeta ] = [ \zeta_{\ast} ] \cdot [ \iota_{\ast} ]^{-1}
\end{equation}
employing the pullback/pushout computation rule.
From Corollary \ref{corollary:liftcolift_formula}, we know that
we can produce $\delta$ as the composition
\begin{equation}\label{equation:term3}
 [\delta] = [ \epsilon^{\ast} ]^{-1} \cdot [ \eta^{\ast} ] \cdot [ \beta ] \cdot [ \zeta_{\ast} ] \cdot [ \iota_{\ast} ]^{-1}.
\end{equation}
Substituting \eqref{equation:term1} and \eqref{equation:term2} in \eqref{equation:term3},
the equation
\[
 [\delta] = [ \eta ] \cdot [ \epsilon ]^{-1} \cdot [ \beta ] \cdot [ \iota ]^{-1} \cdot [ \zeta ]
\]
follows, which is nothing but straightforwardly following the arrows
regardless of their direction from $\kernel( \gamma )$ to $\cokernel( \alpha )$:
\begin{center}
          \begin{tikzpicture}[transform shape,mylabel/.style={thick, draw=black, align=center, minimum width=0.5cm, minimum height=0.5cm,fill=white}]
            \coordinate (r) at (2.5,0);
            \coordinate (u) at (0,2);
            \node (A) {};
            \node (B) at ($(A) + (r)$) {$B$};
            \node (C) at ($(B) + (r)$) {$C$};
            \node (0u) at ($(C) + (r)$) {};
            \node (ker) at ($(C) + (u)$) {$\kernel( \gamma )$};
            \node (Ap) at ($(A) - (u)$) {$A'$};
            \node (Bp) at ($(Ap) + (r)$) {$B'$};
            \node (Cp) at ($(Bp) + (r)$) {};
            \node (0d) at ($(Ap) - (r)$) {};
            
            \node (coker) at ($(Ap) - (u)$) {$\cokernel( \alpha )$};
            
            \draw[->,thick] (C) --node[above]{$[\epsilon]^{-1}$} (B);
            
            \draw[->,thick] (Bp) --node[above]{$[\iota]^{-1}$} (Ap);

            \draw[->,thick] (ker) --node[right]{$[\eta]$} (C);
            \draw[->,thick] (Ap) --node[right]{$[\zeta]$} (coker);
            
            \draw[->,thick] (B) --node[right]{$[\beta]$} (Bp);

          \end{tikzpicture}
\end{center}

\begin{remark}
 This is not a proof of the snake lemma, but a way to construct
 the connecting homomorphism once we know it exists.
 For a proof of the snake lemma using the language of generalized morphisms,
 see \cite[Lemma II.2.1]{PosurDoktor}
\end{remark}

\subsection{Generalized homomorphism theorem}\label{subsection:gen_hom_thm}

To any morphism $\alpha: A \longrightarrow B$ in an abelian category $\AC$,
we can associate two canonical subobjects:
its image $\im(\alpha)$ and its kernel $\kernel(\alpha)$.
The homomorphism theorem states that, using these canonical subobjects,
we get a commutative diagram
\begin{center}
    \begin{tikzpicture}[label/.style={postaction={
            decorate,
            decoration={markings, mark=at position .5 with \node #1;}},
            mylabel/.style={thick, draw=none, align=center, minimum width=0.5cm, minimum height=0.5cm,fill=white}}]
            \coordinate (r) at (3.5,0);
            \coordinate (u) at (0,2);
            \node (A) {$A$};
            \node (B) at ($(A)+(r)$) {$B$};
            \node (C) at ($(A)-(u)$) {$\frac{A}{\kernel( \alpha )}$};
            \node (D) at ($(C)+(r)$) {$\im( \alpha )$};
            
            \draw[->,thick] (A) --node[above]{$\alpha$} (B);
            \draw[->,thick] (C) --node[above]{$\widetilde{\alpha}$} node[below]{$\simeq$} (D);
            
            \draw[->>,thick] (A) -- (C);
            \draw[left hook->,thick] (D) -- (B);
    \end{tikzpicture}
\end{center}

Given a generalized morphism 
\begin{tikzpicture}[label/.style={postaction={
            decorate,
            decoration={markings, mark=at position .5 with \node #1;}},
            mylabel/.style={thick, draw=none, align=center, minimum width=0.5cm, minimum height=0.5cm,fill=white}},baseline=(A)]
            \coordinate (r) at (1.5,0);
            \coordinate (u) at (0,0.5);
            \node (A) {$A$};
            \node (B) at ($(A)+2*(r)$) {$B$};
            \node (C) at ($(A)-(u) + (r)$) {$C$};
            
            \draw[->,dashed] (A) --node[above]{$\alpha$} (B);
            \draw[->] (C) --node[left,yshift=-0.4em,xshift=0.2em]{$\lambda$} (A);
            \draw[->] (C) --node[right,yshift=-0.4em,xshift=-0.4em]{$\rho$} (B);   
\end{tikzpicture},
we have four canonical subobjects:
\begin{itemize}
 \setlength\itemsep{1em}
 \item \textbf{Domain:} \\$\domain( \alpha ) \coloneqq \image( \lambda ) \subseteq A$
 \item \textbf{Generalized kernel:} \\$\gkernel( \alpha ) \coloneqq \lambda( \kernel( \rho ) ) \subseteq A$
 \item \textbf{Generalized image:} \\$\generalizedimage( \alpha ) \coloneqq \image( \rho ) \subseteq B$
 \item \textbf{Defect:} \\$\defect( \alpha ) \coloneqq \rho( \kernel( \lambda ) ) \subseteq B$
\end{itemize}

We claim that
a generalized homomorphism theorem holds, namely,
the existence of a commutative diagram
\begin{center}
    \begin{tikzpicture}[label/.style={postaction={
            decorate,
            decoration={markings, mark=at position .5 with \node #1;}},
            mylabel/.style={thick, draw=none, align=center, minimum width=0.5cm, minimum height=0.5cm,fill=white}}]
            \coordinate (r) at (3.5,0);
            \coordinate (u) at (0,2);
            \node (A) {$A$};
            \node (B) at ($(A)+(r)$) {$B$};
            \node (C) at ($(A)-(u)$) {$\frac{\domain{(\alpha)}}{\gkernel( \alpha )}$};
            \node (D) at ($(C)+(r)$) {$\frac{\generalizedimage( \alpha )}{\defect( \alpha )}$};
            
            \draw[->,thick,dashed] (A) --node[above]{$\alpha$} (B);
            \draw[->,thick] (C) --node[above]{$\widetilde{\alpha}$} node[below]{$\simeq$} (D);
            
            \draw[->>,thick,dashed] (A) -- (C);
            \draw[left hook->,thick,dashed] (D) -- (B);
    \end{tikzpicture}
\end{center}

The two vertical arrows are simply given by the generalized subquotient projection
\[
 A \hookleftarrow \domain( \alpha ) \twoheadrightarrow \frac{\domain(\alpha)}{\gkernel(\alpha)},
\]
which is an epimorphism in $\G( \AC )$ by Theorem \ref{theorem:epi_mono_split},
and the generalized subquotient injection
\[
 \frac{\generalizedimage(\alpha)}{\defect(\alpha)} \twoheadleftarrow \generalizedimage(\alpha) \hookrightarrow B,
\]
which is a monomorphism in $\G( \AC )$ also by Theorem \ref{theorem:epi_mono_split}.

The validity of the generalized homomorphism theorem can be easily extracted from
the following commutative diagram and from the pushout computation rule:

\begin{center}
    \begin{tikzpicture}[label/.style={postaction={
            decorate,
            decoration={markings, mark=at position .5 with \node #1;}},
            mylabel/.style={thick, draw=none, align=center, minimum width=0.5cm, minimum height=0.5cm,fill=white}}]
            \coordinate (r) at (3.5,0);
            \coordinate (u) at (0,2);
            \node (A) {$A$};
            \node (B) at ($(A)+2*(r)$) {$B$};
            \node (C) at ($(A)-(u) + (r)$) {$C$};
            \node (Il) at ($(C)-(r)$) {$\image( \lambda )$};
            \node (Ir) at ($(C)+(r)$) {$\image( \rho )$};
            \node (P) at ($(C)-(u)$) {$\image( \lambda ) \amalg_{C} \image( \rho )$};
            \node (Ql) at ($(P)-(r)$) {$\frac{\image( \lambda )}{\lambda( \kernel( \rho ) )}$};
            \node (Qr) at ($(P)+(r)$) {$\frac{\image( \rho )}{\rho( \kernel( \lambda ) )}$};
            
            \draw[->,thick] (C) --node[above]{$\lambda$} (A);
            \draw[->,thick] (C) --node[above]{$\rho$} (B);
            \draw[->>,thick] (C) -- (Il);
            \draw[left hook->,thick] (Il) -- (A);
            \draw[->>,thick] (C) -- (Ir);
            \draw[right hook->,thick] (Ir) -- (B);
            
            \draw[->>,thick] (Il) -- (Ql);
            \draw[->>,thick] (Il) -- (P);
            
            \draw[->>,thick] (Ir) -- (Qr);
            \draw[->>,thick] (Ir) -- (P);
            
            \draw[->,thick] (Ql) --node[above,yshift=-0.1em] {$\simeq$} (P);
            \draw[->,thick] (P) --node[above,yshift=-0.1em] {$\simeq$} (Qr);
            \draw[->,thick] (P) --node[above,yshift=-0.1em] {$\simeq$} (Qr);
            
            \draw (Ql) [->,thick,out=-25,in=180+25] to node[below]{$\widetilde{\alpha}$} (Qr);

    \end{tikzpicture}
\end{center}

\subsection{Computing spectral sequences}
This subsection serves as an introduction to spectral sequences.
We use generalized morphisms as a fundamental tool in our
explanation. This has two advantages:
\begin{enumerate}
 \item The main idea behind spectral sequences becomes quite transparent when you already
 have generalized morphisms available as a tool.
 \item Instead of mere existence theorems, we will get explicit formulas for all the differentials
 within a spectral sequence.
\end{enumerate}

Let $\AC$ be an abelian category.
A spectral sequence is a lot of data that can naturally be associated to
a given filtered cochain complex, i.e., a cochain complex
\begin{center}
          \begin{tikzpicture}[transform shape,mylabel/.style={thick, draw=black, align=center, minimum width=0.5cm, minimum height=0.5cm,fill=white}]
            \coordinate (r) at (2.5,0);
            \coordinate (u) at (0,2);
            \node (dotsl) {$\dots$};
            \node (A) at ($(dotsl) + (r)$) {$M^{i}$};
            \node (B) at ($(A) + (r)$) {$M^{i+1}$};
            \node (C) at ($(B) + (r)$) {$M^{i+2}$};
            \node (D) at ($(C) + (r)$) {$M^{i+3}$};
            \node (dotsr) at ($(D) + (r)$) {$\dots$};
            
            \draw[->,thick] (dotsl) -- (A);
            \draw[->,thick] (A) --node[above]{$\partial^i$} (B);
            \draw[->,thick] (B) --node[above]{$\partial^{i+1}$} (C);
            \draw[->,thick] (C) --node[above]{$\partial^{i+2}$} (D);
            \draw[->,thick] (D) -- (dotsr);
          \end{tikzpicture}
\end{center}
in which each object $M^i$ is equipped with a chain of subobjects
\[
 M^i \supseteq \dots \supseteq F^{j}M^i \supseteq F^{j+1}M^i \supseteq F^{j+2}M^i \supseteq \dots
\]
compatible with the differentials, i.e., $\partial^i$ restricts to a morphism
\[
 F^j\partial^i: F^jM^i \longrightarrow F^jM^{i+1}
\]
for every $i,j \in \Z$.
To simplify our explanation, we will concentrate on a finite excerpt of
such a filtered cochain complex, and denote it as follows:
\begin{center}
          \begin{tikzpicture}[transform shape,mylabel/.style={thick, draw=black, align=center, minimum width=0.5cm, minimum height=0.5cm,fill=white}]
            \coordinate (r) at (2.5,0);
            \coordinate (u) at (0,2);
            \node (dotsl) {$\dots$};
            \node (A) at ($(dotsl) + (r)$) {$A$};
            \node (B) at ($(A) + (r)$) {$B$};
            \node (C) at ($(B) + (r)$) {$C$};
            \node (D) at ($(C) + (r)$) {$D$};
            \node (dotsr) at ($(D) + (r)$) {$\dots$};
            
            \draw[->,thick] (dotsl) -- (A);
            \draw[->,thick] (A) --node[above]{$\partial^A$} (B);
            \draw[->,thick] (B) --node[above]{$\partial^{B}$} (C);
            \draw[->,thick] (C) --node[above]{$\partial^{C}$} (D);
            \draw[->,thick] (D) -- (dotsr);
          \end{tikzpicture}
\end{center}
with chain of subobjects
\[
 A \supseteq \dots \supseteq A^{j} \supseteq A^{j+1} \supseteq A^{j+2} \supseteq \dots
\]
and likewise for $B$, $C$, and $D$. The restrictions of the differentials to the $j$-th subobjects are
denoted by adding an extra index, e.g., $\partial^{A,j}: A^j \longrightarrow B^j$.

For every $j \in \Z$, 
we can restrict our filtered cochain complex to its $j$-th graded part 
and again obtain a cochain complex:
\begin{center}
          \begin{tikzpicture}[transform shape,mylabel/.style={thick, draw=black, align=center, minimum width=0.5cm, minimum height=0.5cm,fill=white}]
            \coordinate (r) at (2.5,0);
            \coordinate (u) at (0,2);
            \node (dotsl) {$\dots$};
            \node (A) at ($(dotsl) + (r)$) {$\frac{A^{j}}{A^{j+1}}$};
            \node (B) at ($(A) + (r)$) {$\frac{B^{j}}{B^{j+1}}$};
            \node (C) at ($(B) + (r)$) {$\frac{C^{j}}{C^{j+1}}$};
            \node (D) at ($(C) + (r)$) {$\frac{D^{j}}{D^{j+1}}$};
            \node (dotsr) at ($(D) + (r)$) {$\dots$};
            
            \draw[->,thick] (dotsl) -- (A);
            \draw[->,thick] (A) --node[above]{$\overline{\partial^{A,j}}$} (B);
            \draw[->,thick] (B) --node[above]{$\overline{\partial^{B,j}}$} (C);
            \draw[->,thick] (C) --node[above]{$\overline{\partial^{C,j}}$} (D);
            \draw[->,thick] (D) -- (dotsr);
          \end{tikzpicture}
\end{center}
It is the common convention to arrange this $\Z$-indexed family of cochain complexes
between the graded parts as follows:
\begin{center}
          \begin{tikzpicture}[transform shape,mylabel/.style={thick, draw=black, align=center, minimum width=0.5cm, minimum height=0.5cm,fill=white}]
            \coordinate (r) at (2.5,0);
            \coordinate (u) at (0,2);
            \node (A11) {};
            \node (A12) at ($(A11) + (r)$) {$\dots$};
            \node (A13) at ($(A12) + (r)$) {};
            \node (A14) at ($(A13) + (r)$) {};
            \node (A15) at ($(A14) + (r)$) {};
            
            \node (A21) at ($(A11) - (u)$) {$\dots$};
            \node (A22) at ($(A21) + (r)$) {$\frac{D^j}{D^{j+1}}$};
            \node (A23) at ($(A22) + (r)$) {$\dots$};
            \node (A24) at ($(A23) + (r)$) {};
            \node (A25) at ($(A24) + (r)$) {};
            
            \node (A31) at ($(A21) - (u)$) {$\dots$};
            \node (A32) at ($(A31) + (r)$) {$\frac{C^j}{C^{j+1}}$};
            \node (A33) at ($(A32) + (r)$) {$\frac{D^{j+1}}{D^{j+2}}$};
            \node (A34) at ($(A33) + (r)$) {$\dots$};
            \node (A35) at ($(A34) + (r)$) {};
            
            \node (A41) at ($(A31) - (u)$) {$\dots$};
            \node (A42) at ($(A41) + (r)$) {$\frac{B^j}{B^{j+1}}$};
            \node (A43) at ($(A42) + (r)$) {$\frac{C^{j+1}}{C^{j+2}}$};
            \node (A44) at ($(A43) + (r)$) {$\frac{D^{j+2}}{D^{j+3}}$};
            \node (A45) at ($(A44) + (r)$) {$\dots$};
            
            \node (A51) at ($(A41) - (u)$) {$\dots$};
            \node (A52) at ($(A51) + (r)$) {$\frac{A^j}{A^{j+1}}$};
            \node (A53) at ($(A52) + (r)$) {$\frac{B^{j+1}}{B^{j+2}}$};
            \node (A54) at ($(A53) + (r)$) {$\frac{C^{j+2}}{C^{j+3}}$};
            \node (A55) at ($(A54) + (r)$) {$\dots$};
            
            \node (A61) at ($(A51) - (u)$) {};
            \node (A62) at ($(A61) + (r)$) {$\dots$};
            \node (A63) at ($(A62) + (r)$) {$\frac{A^{j+1}}{A^{j+2}}$};
            \node (A64) at ($(A63) + (r)$) {$\frac{B^{j+2}}{B^{j+3}}$};
            \node (A65) at ($(A64) + (r)$) {$\dots$};
            
            \node (A71) at ($(A61) - (u)$) {};
            \node (A72) at ($(A71) + (r)$) {};
            \node (A73) at ($(A72) + (r)$) {$\dots$};
            \node (A74) at ($(A73) + (r)$) {$\frac{A^{j+2}}{A^{j+3}}$};
            \node (A75) at ($(A74) + (r)$) {$\dots$};
            
            \node (A81) at ($(A71) - (u)$) {};
            \node (A82) at ($(A81) + (r)$) {};
            \node (A83) at ($(A82) + (r)$) {};
            \node (A84) at ($(A83) + (r)$) {$\dots$};
            \node (A85) at ($(A84) + (r)$) {};
            
            \draw[->,thick] (A22) -- (A12);
            \draw[->,thick] (A32) -- (A22);
            \draw[->,thick] (A42) -- (A32);
            \draw[->,thick] (A52) -- (A42);
            \draw[->,thick] (A62) -- (A52);
            
            \draw[->,thick] (A33) -- (A23);
            \draw[->,thick] (A43) -- (A33);
            \draw[->,thick] (A53) -- (A43);
            \draw[->,thick] (A63) -- (A53);
            \draw[->,thick] (A73) -- (A63);
            
            \draw[->,thick] (A44) -- (A34);
            \draw[->,thick] (A54) -- (A44);
            \draw[->,thick] (A64) -- (A54);
            \draw[->,thick] (A74) -- (A64);
            \draw[->,thick] (A84) -- (A74);
          \end{tikzpicture}
\end{center}

Let us take a closer look at the induced differentials $\overline{\partial^{A,j}}$.
They fit into a commutative diagram
\begin{center}
          \begin{tikzpicture}[transform shape,mylabel/.style={thick, draw=black, align=center, minimum width=0.5cm, minimum height=0.5cm,fill=white}]
            \coordinate (r) at (2.5,0);
            \coordinate (u) at (0,2);
            \node (A) {$A$};
            \node (B) at ($(A) + (r)$) {$B$};
            \node (Aj) at ($(A) - (u)$) {$A^j$};
            \node (Bj) at ($(B) - (u)$) {$B^j$};
            \node (Ajp) at ($(Aj) - (u)$) {$\frac{A^j}{A^{j+1}}$};
            \node (Bjp) at ($(Bj) - (u)$) {$\frac{B^j}{B^{j+1}}$};
            
            \draw[->,thick] (A) --node[above]{$\partial^A$} (B);
            \draw[->,thick] (Aj) --node[above]{$\partial^{A,j}$} (Bj);
            \draw[->,thick] (Ajp) --node[above]{$\overline{\partial^{A,j}}$} (Bjp);
            
            \draw[->>,thick] (Aj) --node[left]{$\epsilon^{A,j}$} (Ajp);
            \draw[->>,thick] (Bj) --node[right]{$\epsilon^{B,j}$} (Bjp);
            
            \draw[left hook->,thick] (Aj) --node[left]{$\iota^{A,j}$} (A);
            \draw[left hook->,thick] (Bj) --node[right]{$\iota^{B,j}$} (B);
          \end{tikzpicture}
\end{center}
which shows, using Corollary \ref{corollary:liftcolift_formula}, that
we may express $\overline{\partial^{A,j}}$ as a composition of generalized morphisms,
following the outer path from $\frac{A^j}{A^{j+1}}$ to $\frac{B^j}{B^{j+1}}$ in the diagram above:
\[
 [ \overline{\partial^{A,j}} ] 
 =
 [ \epsilon^{A,j} ]^{-1} \cdot [ \iota^{A,j} ] \cdot [ \partial^{A} ] \cdot [ \iota^{B,j} ]^{-1} \cdot [ \epsilon^{B,j} ].
\]
To simplify this expression, let us introduce
\[
 \emb^{A,j} \coloneqq  [ \epsilon^{A,j} ]^{-1} \cdot [ \iota^{A,j} ]: \frac{A^j}{A^{j+1}} \dashrightarrow A
\]
as notation for the \textbf{generalized subquotient embedding}
and
\[
 \proj^{B,j} \coloneqq  [ \iota^{B,j} ]^{-1} \cdot [ \epsilon^{B,j} ]: B \dashrightarrow \frac{B^j}{B^{j+1}}
\]
as notation for the \textbf{generalized subquotient projection}.
Then, the induced morphism between graded parts is literally given
by restricting $\partial^{A}: A \rightarrow B$ to the appropriate subquotients:
\[
 [ \overline{\partial^{A,j}} ] 
 =
 \emb^{A,j} \cdot [ \partial^{A} ] \cdot \proj^{B,j}.
\]
Now, the main idea behind spectral sequences is that too much information is lost
when we only focus on restrictions of $\partial^{A}$ to subquotients of the same index $j$,
and thus, we should try and see what happens if we increase the index of the projection by $1$:
\[
 \partial_1^{A,j}
 \coloneqq
 \emb^{A,j} \cdot [ \partial^{A} ] \cdot \proj^{B,j+1}.
\]
In general, we cannot expect this generalized morphism
to be honest anymore and so we depict it with a dashed arrow
\[
 \partial_1^{A,j}: \frac{A^j}{A^{j+1}} \dashrightarrow \frac{B^{j+1}}{B^{j+2}}.
\]
We can assemble these generalized differentials within 
a structure that we would like to call a \emph{generalized} cochain complex:

\begin{equation}\label{equation:term5}
\begin{tikzpicture}[baseline = (A)]
  \coordinate (r) at (2.5,0);
  \coordinate (u) at (0,2);
  \node (dotsl) {$\dots$};
  \node (A) at ($(dotsl) + (r)$) {$\frac{A^{j}}{A^{j+1}}$};
  \node (B) at ($(A) + (r)$) {$\frac{B^{j+1}}{B^{j+2}}$};
  \node (C) at ($(B) + (r)$) {$\frac{C^{j+2}}{C^{j+3}}$};
  \node (D) at ($(C) + (r)$) {$\frac{D^{j+3}}{D^{j+4}}$};
  \node (dotsr) at ($(D) + (r)$) {$\dots$};
  
  \draw[->,dashed,thick] (dotsl) -- (A);
  \draw[->,dashed,thick] (A) --node[above]{$\partial_1^{A,j}$} (B);
  \draw[->,dashed,thick] (B) --node[above]{$\partial_1^{B,j+1}$} (C);
  \draw[->,dashed,thick] (C) --node[above]{$\partial_1^{C,j+2}$} (D);
  \draw[->,dashed,thick] (D) -- (dotsr);
\end{tikzpicture} 
\end{equation}

\begin{definition}
We define a \textbf{generalized cochain complex} to be a $\Z$-indexed family
of objects $M^i$ together with a $\Z$-indexed family of \emph{generalized} morphisms
\[
  \partial^i: M^i \dashrightarrow M^{i+1}
\]
such that
\[
\generalizedimage( \partial^i ) \subseteq \gkernel( \partial^{i+1} ).
\]
\end{definition}

We show that two consecutive morphisms in \eqref{equation:term5},
e.g., $\partial_1^{A,j}$ and $\partial_1^{B,j+1}$, satisfy
\begin{equation}\label{equation:term4}
 \generalizedimage( \partial_1^{A,j} ) \subseteq \gkernel( \partial_1^{B,j+1} ).
\end{equation}
Indeed, we can calculate
\[
 \generalizedimage( \partial_1^{A,j} ) = \frac{ \big(\partial^A(A^j) \cap B^{j+1}\big) + B^{j+2}}{B^{j+2}}
\]
and
\[
\gkernel( \partial_1^{B,j+1} ) = \frac{ \big((\partial^B)^{-1}(C^{j+3}) \cap B^{j+1}\big) + B^{j+2}}{B^{j+2}}
\]
where we use standard notation for dealing with subobjects in abelian categories, i.e., $\cap$ and $(-)^{-1}$ are shorthand
for the corresponding pullbacks, and $+$ for the join of subobjects.
Since
\[
 \partial^A(A^j) \subseteq \image( \partial^A ) \subseteq \kernel(\partial^B) = (\partial^B)^{-1}(0) \subseteq (\partial^B)^{-1}(C^{j+3})
\]
we really get our desired inclusion \eqref{equation:term4}. Thus, \eqref{equation:term5}
forms a generalized cochain complex.

The whole collection of generalized cochain complexes that we get in this way may be depicted as follows:
\begin{center}
          \begin{tikzpicture}[transform shape,mylabel/.style={thick, draw=black, align=center, minimum width=0.5cm, minimum height=0.5cm,fill=white}]
            \coordinate (r) at (2.5,0);
            \coordinate (u) at (0,2);
            \node (A11) {};
            \node (A12) at ($(A11) + (r)$) {$\dots$};
            \node (A13) at ($(A12) + (r)$) {};
            \node (A14) at ($(A13) + (r)$) {};
            \node (A15) at ($(A14) + (r)$) {};
            
            \node (A21) at ($(A11) - (u)$) {$\dots$};
            \node (A22) at ($(A21) + (r)$) {$\frac{D^j}{D^{j+1}}$};
            \node (A23) at ($(A22) + (r)$) {$\dots$};
            \node (A24) at ($(A23) + (r)$) {};
            \node (A25) at ($(A24) + (r)$) {};
            
            \node (A31) at ($(A21) - (u)$) {$\dots$};
            \node (A32) at ($(A31) + (r)$) {$\frac{C^j}{C^{j+1}}$};
            \node (A33) at ($(A32) + (r)$) {$\frac{D^{j+1}}{D^{j+2}}$};
            \node (A34) at ($(A33) + (r)$) {$\dots$};
            \node (A35) at ($(A34) + (r)$) {};
            
            \node (A41) at ($(A31) - (u)$) {$\dots$};
            \node (A42) at ($(A41) + (r)$) {$\frac{B^j}{B^{j+1}}$};
            \node (A43) at ($(A42) + (r)$) {$\frac{C^{j+1}}{C^{j+2}}$};
            \node (A44) at ($(A43) + (r)$) {$\frac{D^{j+2}}{D^{j+3}}$};
            \node (A45) at ($(A44) + (r)$) {$\dots$};
            
            \node (A51) at ($(A41) - (u)$) {$\dots$};
            \node (A52) at ($(A51) + (r)$) {$\frac{A^j}{A^{j+1}}$};
            \node (A53) at ($(A52) + (r)$) {$\frac{B^{j+1}}{B^{j+2}}$};
            \node (A54) at ($(A53) + (r)$) {$\frac{C^{j+2}}{C^{j+3}}$};
            \node (A55) at ($(A54) + (r)$) {$\dots$};
            
            \node (A61) at ($(A51) - (u)$) {};
            \node (A62) at ($(A61) + (r)$) {$\dots$};
            \node (A63) at ($(A62) + (r)$) {$\frac{A^{j+1}}{A^{j+2}}$};
            \node (A64) at ($(A63) + (r)$) {$\frac{B^{j+2}}{B^{j+3}}$};
            \node (A65) at ($(A64) + (r)$) {$\dots$};
            
            \node (A71) at ($(A61) - (u)$) {};
            \node (A72) at ($(A71) + (r)$) {};
            \node (A73) at ($(A72) + (r)$) {$\dots$};
            \node (A74) at ($(A73) + (r)$) {$\frac{A^{j+2}}{A^{j+3}}$};
            \node (A75) at ($(A74) + (r)$) {$\dots$};
            
            \node (A81) at ($(A71) - (u)$) {};
            \node (A82) at ($(A81) + (r)$) {};
            \node (A83) at ($(A82) + (r)$) {};
            \node (A84) at ($(A83) + (r)$) {$\dots$};
            \node (A85) at ($(A84) + (r)$) {};
            
            \draw[->,dashed,thick] (A21) -- (A22);
            \draw[->,dashed,thick] (A22) -- (A23);
            
            \draw[->,dashed,thick] (A31) -- (A32);
            \draw[->,dashed,thick] (A32) -- (A33);
            \draw[->,dashed,thick] (A33) -- (A34);

            \draw[->,dashed,thick] (A41) -- (A42);
            \draw[->,dashed,thick] (A42) -- (A43);
            \draw[->,dashed,thick] (A43) -- (A44);
            \draw[->,dashed,thick] (A44) -- (A45);

            \draw[->,dashed,thick] (A51) -- (A52);
            \draw[->,dashed,thick] (A52) -- (A53);
            \draw[->,dashed,thick] (A53) -- (A54);
            \draw[->,dashed,thick] (A54) -- (A55);
            
            \draw[->,dashed,thick] (A62) -- (A63);
            \draw[->,dashed,thick] (A63) -- (A64);
            \draw[->,dashed,thick] (A64) -- (A65);
            
            \draw[->,dashed,thick] (A73) -- (A74);
            \draw[->,dashed,thick] (A74) -- (A75);
            
          \end{tikzpicture}
\end{center}
Increasing the index of the projection by $2$ would yield
the following picture (again of generalized cochain complexes):
\begin{center}
          \begin{tikzpicture}[transform shape,mylabel/.style={thick, draw=black, align=center, minimum width=0.5cm, minimum height=0.5cm,fill=white}]
            \coordinate (r) at (2.5,0);
            \coordinate (u) at (0,2);
            \node (A11) {};
            \node (A12) at ($(A11) + (r)$) {$\dots$};
            \node (A13) at ($(A12) + (r)$) {};
            \node (A14) at ($(A13) + (r)$) {};
            \node (A15) at ($(A14) + (r)$) {};
            
            \node (A21) at ($(A11) - (u)$) {$\dots$};
            \node (A22) at ($(A21) + (r)$) {$\frac{D^j}{D^{j+1}}$};
            \node (A23) at ($(A22) + (r)$) {$\dots$};
            \node (A24) at ($(A23) + (r)$) {};
            \node (A25) at ($(A24) + (r)$) {};
            
            \node (A31) at ($(A21) - (u)$) {$\dots$};
            \node (A32) at ($(A31) + (r)$) {$\frac{C^j}{C^{j+1}}$};
            \node (A33) at ($(A32) + (r)$) {$\frac{D^{j+1}}{D^{j+2}}$};
            \node (A34) at ($(A33) + (r)$) {$\dots$};
            \node (A35) at ($(A34) + (r)$) {};
            
            \node (A41) at ($(A31) - (u)$) {$\dots$};
            \node (A42) at ($(A41) + (r)$) {$\frac{B^j}{B^{j+1}}$};
            \node (A43) at ($(A42) + (r)$) {$\frac{C^{j+1}}{C^{j+2}}$};
            \node (A44) at ($(A43) + (r)$) {$\frac{D^{j+2}}{D^{j+3}}$};
            \node (A45) at ($(A44) + (r)$) {$\dots$};
            
            \node (A51) at ($(A41) - (u)$) {$\dots$};
            \node (A52) at ($(A51) + (r)$) {$\frac{A^j}{A^{j+1}}$};
            \node (A53) at ($(A52) + (r)$) {$\frac{B^{j+1}}{B^{j+2}}$};
            \node (A54) at ($(A53) + (r)$) {$\frac{C^{j+2}}{C^{j+3}}$};
            \node (A55) at ($(A54) + (r)$) {$\dots$};
            
            \node (A61) at ($(A51) - (u)$) {};
            \node (A62) at ($(A61) + (r)$) {$\dots$};
            \node (A63) at ($(A62) + (r)$) {$\frac{A^{j+1}}{A^{j+2}}$};
            \node (A64) at ($(A63) + (r)$) {$\frac{B^{j+2}}{B^{j+3}}$};
            \node (A65) at ($(A64) + (r)$) {$\dots$};
            
            \node (A71) at ($(A61) - (u)$) {};
            \node (A72) at ($(A71) + (r)$) {};
            \node (A73) at ($(A72) + (r)$) {$\dots$};
            \node (A74) at ($(A73) + (r)$) {$\frac{A^{j+2}}{A^{j+3}}$};
            \node (A75) at ($(A74) + (r)$) {$\dots$};
            
            \node (A81) at ($(A71) - (u)$) {};
            \node (A82) at ($(A81) + (r)$) {};
            \node (A83) at ($(A82) + (r)$) {};
            \node (A84) at ($(A83) + (r)$) {$\dots$};
            \node (A85) at ($(A84) + (r)$) {};
            \draw[->,dashed,thick] (A21) -- (A33);
            \draw[->,dashed,thick] (A22) -- (A34);
            
            \draw[->,dashed,thick] (A31) -- (A43);
            \draw[->,dashed,thick] (A32) -- (A44);
            \draw[->,dashed,thick] (A33) -- (A45);
            \draw[->,dashed,thick] (A41) -- (A53);
            \draw[->,dashed,thick] (A42) -- (A54);
            \draw[->,dashed,thick] (A43) -- (A55);
%
            \draw[->,dashed,thick] (A52) -- (A64);
            \draw[->,dashed,thick] (A53) -- (A65);
%
            \draw[->,dashed,thick] (A63) -- (A75);
%
%
          \end{tikzpicture}
\end{center}

It follows that we are able to construct for every integer $i \geq 0$,
and not only for the case $i = 0$,
a $\Z$-indexed family of \emph{generalized} cochain complexes

\begin{equation}\label{equation:generalized_complexes}
 \begin{tikzpicture}[baseline = (A)]
  \coordinate (r) at (2.5,0);
  \coordinate (u) at (0,2);
  \node (dotsl) {$\dots$};
  \node (A) at ($(dotsl) + (r)$) {$\frac{A^{j}}{A^{j+1}}$};
  \node (B) at ($(A) + (r)$) {$\frac{B^{j+i}}{B^{j+i+1}}$};
  \node (C) at ($(B) + (r)$) {$\frac{C^{j+2i}}{C^{j+2i + 1}}$};
  \node (D) at ($(C) + (r)$) {$\frac{D^{j+3i}}{D^{j+3i+1}}$};
  \node (dotsr) at ($(D) + (r)$) {$\dots$};
  
  \draw[->,dashed,thick] (dotsl) -- (A);
  \draw[->,dashed,thick] (A) --node[above]{$\partial_i^{A,j}$} (B);
  \draw[->,dashed,thick] (B) --node[above]{$\partial_i^{B,j+i}$} (C);
  \draw[->,dashed,thick] (C) --node[above]{$\partial_i^{C,j+2i}$} (D);
  \draw[->,dashed,thick] (D) -- (dotsr);
\end{tikzpicture}
\end{equation}
Next, we will see how to produce from a generalized cochain complex
an ordinary cochain complex having honest differentials.
Applying this process to the just created generalized cochain complexes
will then yield our desired spectral sequence.

So, let
\begin{center}
          \begin{tikzpicture}[transform shape,mylabel/.style={thick, draw=black, align=center, minimum width=0.5cm, minimum height=0.5cm,fill=white}]
            \coordinate (r) at (2.5,0);
            \coordinate (u) at (0,2);
            \node (dotsl) {$\dots$};
            \node (A) at ($(dotsl) + (r)$) {$M^{i}$};
            \node (B) at ($(A) + (r)$) {$M^{i+1}$};
            \node (C) at ($(B) + (r)$) {$M^{i+2}$};
            \node (D) at ($(C) + (r)$) {$M^{i+3}$};
            \node (dotsr) at ($(D) + (r)$) {$\dots$};
            
            \draw[->,thick,dashed] (dotsl) -- (A);
            \draw[->,thick,dashed] (A) --node[above]{$\partial^i$} (B);
            \draw[->,thick,dashed] (B) --node[above]{$\partial^{i+1}$} (C);
            \draw[->,thick,dashed] (C) --node[above]{$\partial^{i+2}$} (D);
            \draw[->,thick,dashed] (D) -- (dotsr);
          \end{tikzpicture}
\end{center}
be an arbitrary generalized cochain complex.
Since we have
\[
 \generalizedimage( \partial^i ) \subseteq \gkernel( \partial^{i+1} ),
\]
we also have
\[
 \defect( \partial^i ) \subseteq \generalizedimage( \partial^i ) \subseteq \gkernel( \partial^{i+1} ) \subseteq \domain( \partial^{i+1} ) .
\]
We apply the generalized homomorphism theorem (see Subsection \ref{subsection:gen_hom_thm})
to the generalized morphisms $\partial^i$
in order to produce honest morphisms $d^i$ fitting in the following
commutative diagram:
\begin{center}
          \begin{tikzpicture}[transform shape,mylabel/.style={thick, draw=black, align=center, minimum width=0.5cm, minimum height=0.5cm,fill=white}]
            \coordinate (r) at (1.9,0);
            \coordinate (d) at (0,-2);
            \node (M11) {$\frac{\domain( \partial^{i+1} )}{\defect( \partial^i )}$};
            \node (M14) at ($(M11) + 3*(r)$) {$\frac{\domain( \partial^{i+2} )}{\defect( \partial^{i+1} )}$};
            \node (M17) at ($(M14) + 3*(r)$) {$\frac{\domain( \partial^{i+3} )}{\defect( \partial^{i+2} )}$};
            
            \node (M22) at ($(M11) + 0.75*(r) + (d)$) {$\frac{\domain( \partial^{i+1} )}{\gkernel( \partial^{i+1} )}$};
            \node (M23) at ($(M14) - 0.75*(r) + (d)$) {$\frac{\generalizedimage( \partial^{i+1} )}{\defect( \partial^{i+1} )}$};
            \node (M25) at ($(M14) + 0.75*(r) + (d)$) {$\frac{\domain( \partial^{i+2} )}{\gkernel( \partial^{i+2} )}$};
            \node (M26) at ($(M17) - 0.75*(r) + (d)$) {$\frac{\generalizedimage( \partial^{i+2} )}{\defect( \partial^{i+2} )}$};
            
            \node (M32) at ($(M22) + (d)$) {$M^{i+1}$};
            \node (M33) at ($(M23) + (d)$) {$M^{i+2}$};
            \node (M35) at ($(M25) + (d)$) {$M^{i+2}$};
            \node (M36) at ($(M26) + (d)$) {$M^{i+3}$};
            
            \node (dotsl) at ($(M11) - (r)$) {$\dots$};
            \node (dotsr) at ($(M17) + (r)$) {$\dots$};
            
            \draw[->,thick] (M11) --node[above]{$d^{i+1}$} (M14);
            \draw[->,thick] (M14) --node[above]{$d^{i+2}$} (M17);
            
            \draw[->>,thick] (M11) -- (M22);
            \draw[->,thick] (M22) --node[above]{$\widetilde{\partial^{i+1}}$} (M23);
            \draw[right hook->,thick] (M23) -- (M14);
            
            \draw[->>,thick] (M14) -- (M25);
            \draw[->,thick] (M25) --node[above]{$\widetilde{\partial^{i+2}}$} (M26);
            \draw[right hook->,thick] (M26) -- (M17);
            
            \draw[->>,thick,dashed] (M32) -- (M22);
            \draw[->,thick,dashed] (M32) --node[above]{$\partial^{i+1}$} (M33);
            \draw[right hook->,thick,dashed] (M23) -- (M33);
            
            \draw[->>,thick,dashed] (M35) -- (M25);
            \draw[->,thick,dashed] (M35) --node[above]{$\partial^{i+1}$} (M36);
            \draw[right hook->,thick,dashed] (M26) -- (M36);
            
            \draw[->,thick] (dotsl) -- (M11);
            \draw[->,thick] (M17) -- (dotsr);
            
            \draw[draw = none] (M33) --node[yshift = -0.1em]{$=$} (M35);
            
            \draw[->,thick] (M23) --node[above]{$0$} (M25);
          \end{tikzpicture}
\end{center}
We can directly read off the equation
\[
 d^{i+1} \cdot d^{i+2} = 0.
\]
The collection of the $d^i$ is what we call the
\textbf{associated honest cochain complex} of the generalized cochain complex given by the $\partial^i$.
Note that the rectangles of the above diagram
\begin{center}
          \begin{tikzpicture}[transform shape,mylabel/.style={thick, draw=black, align=center, minimum width=0.5cm, minimum height=0.5cm,fill=white}]
            \coordinate (r) at (2,0);
            \coordinate (d) at (0,-2);
            \node (M11) {$\frac{\domain( \partial^{i+1} )}{\defect( \partial^i )}$};
            \node (M14) at ($(M11) + 3*(r)$) {$\frac{\domain( \partial^{i+2} )}{\defect( \partial^{i+1} )}$};
            
            \node (M22) at ($(M11) + 0.75*(r) + (d)$) {$\frac{\domain( \partial^{i+1} )}{\gkernel( \partial^{i+1} )}$};
            \node (M23) at ($(M14) - 0.75*(r) + (d)$) {$\frac{\generalizedimage( \partial^{i+1} )}{\defect( \partial^{i+1} )}$};
            
            \draw[->,thick] (M11) --node[above]{$d^{i+1}$} (M14);
            \draw[->>,thick] (M11) -- (M22);
            \draw[->,thick] (M22) --node[above]{$\widetilde{\partial^{i+1}}$} (M23);
            \draw[right hook->,thick] (M23) -- (M14);
          \end{tikzpicture}
\end{center}
are actually decompositions of the $d^{i+1}$ in the sense of the homomorphism theorem, since $\widetilde{\partial^{i+1}}$ is an isomorphism.
But then it follows that
\[
 \kernel( d^{i+1} ) = \frac{\gkernel( \partial^{i+1} )}{\defect( \partial^i )}
\]
and
\[
 \image( d^{i+1} ) = \frac{\generalizedimage( \partial^{i+1} )}{\defect( \partial^{i+1} )}.
\]
In particular, we can compute the cohomologies of the associated honest cochain complex $d^{\bullet}$
in terms of $\partial^{\bullet}$:
\begin{equation}\label{equation:term6}
 \CH^{i+1}( d^{\bullet} ) \simeq \frac{\gkernel( \partial^{i+1} )}{\generalizedimage( \partial^{i} )}.
\end{equation}

Now, let us go back to our generalized cochain complexes \eqref{equation:generalized_complexes}.
As we have learned in \eqref{equation:term6}, computing the cohomologies of their associated honest cochain complexes
boils down to the computation of generalized images and generalized kernels, for which we have:
\[
 \generalizedimage( \partial_i^{A,j} ) = \frac{ \big(\partial^A(A^j) \cap B^{j+i}\big) + B^{j+i+1}}{B^{j+i+1}}
\]
and
\[
\gkernel( \partial_i^{B,j+1} ) = \frac{ \big((\partial^B)^{-1}(C^{j+2i+1}) \cap B^{j+i}\big) + B^{j+i+1}}{B^{j+i+1}}.
\]
Computing the remaining two canonical subobjects can be performed analogously and yields
\[
 \defect( \partial_i^{A,j} ) = \frac{ \big(\partial^A(A^{j+1}) \cap B^{j+i}\big) + B^{j+i+1}}{B^{j+i+1}}
\]
and
\[
\domain( \partial_i^{B,j+1} ) = \frac{ \big((\partial^B)^{-1}(C^{j+2i}) \cap B^{j+i}\big) + B^{j+i+1}}{B^{j+i+1}}.
\]
But from this, we can deduce by a simple variable substitution
\[
 \gkernel( \partial_i^{B,j+1} ) = \domain( \partial_{i+1}^{B,j} )
\]
and
\[
 \generalizedimage( \partial_i^{A,j} ) = \defect( \partial_{i+1}^{A,j-1} ).
\]
In particular, we deduce
\[
 \frac{\gkernel( \partial_i^{B,j+1} )}{\generalizedimage( \partial_i^{A,j} )} \simeq \frac{\domain( \partial_{i+1}^{B,j} )}{\defect( \partial_{i+1}^{A,j-1} )}.
\]
Putting these information together, it follows that
the cohomologies of the $i$-th associated honest cochain complexes
determine the objects of the $(i+1)$-th associated honest cochain complexes.
This is exactly the defining feature of a spectral sequence,
which we are going to define now.

\begin{definition}\label{definition:category_of_spectral_sequences}
  A \bfindex{cohomological spectral sequence}
  (starting at $0$) consists of the following data: For all $p,q \in \Z$, $r \geq 0$,
  we have:
  \begin{enumerate}
   \item objects $E_r^{p,q} \in \AC$,
   \item morphisms $d_{r}^{p,q}: E_r^{p,q} \longrightarrow E_r^{p+r,q-(r-1)} \in \AC$,
   \item isomorphisms
         $\iota_r^{p,q}: E_{r+1}^{p,q} \xlongrightarrow{\sim} \frac{\kernel( d_{r}^{p,q} )}{\image( d_{r}^{p-r,q+(r-1)} )}$,
   \item the equation $ d_{r}^{p,q}\cdot d_{r}^{p+r,q-(r-1)}  = 0$ holds.
  \end{enumerate}
\end{definition}

From the discussion in this subsection, it follows that
if we are given a filtered cochain complex
\begin{center}
          \begin{tikzpicture}[transform shape,mylabel/.style={thick, draw=black, align=center, minimum width=0.5cm, minimum height=0.5cm,fill=white}]
            \coordinate (r) at (2.5,0);
            \coordinate (u) at (0,2);
            \node (dotsl) {$\dots$};
            \node (A) at ($(dotsl) + (r)$) {$M^{i}$};
            \node (B) at ($(A) + (r)$) {$M^{i+1}$};
            \node (C) at ($(B) + (r)$) {$M^{i+2}$};
            \node (D) at ($(C) + (r)$) {$M^{i+3}$};
            \node (dotsr) at ($(D) + (r)$) {$\dots$};
            
            \draw[->,thick] (dotsl) -- (A);
            \draw[->,thick] (A) --node[above]{$\partial^i$} (B);
            \draw[->,thick] (B) --node[above]{$\partial^{i+1}$} (C);
            \draw[->,thick] (C) --node[above]{$\partial^{i+1}$} (D);
            \draw[->,thick] (D) -- (dotsr);
          \end{tikzpicture}
\end{center}
then we can construct a spectral sequence by first defining the auxiliary data
\[
 E_0^{p,q} \coloneqq \frac{F^{p}M^{p+q}}{F^{p+1}M^{p+q}}
\]
and
\begin{center}
          \begin{tikzpicture}[transform shape,mylabel/.style={thick, draw=black, align=center, minimum width=0.5cm, minimum height=0.5cm,fill=white}]
            \coordinate (r) at (3.5,0);
            \coordinate (u) at (0,2);
            \node (dotsl) {$\partial_r^{p,q}$};
            \node (A) at ($(dotsl) + 0.5*(r)$) {$E_0^{p,q}$};
            \node (B) at ($(A) + (r)$) {$M^{p+q}$};
            \node (C) at ($(B) + (r)$) {$M^{p+q+1}$};
            \node (D) at ($(C) + (r)$) {$E_0^{p+r,q-(r-1)} $};
            
            \draw[draw = none] (dotsl) --node[]{$\coloneqq$} (A);
            \draw[right hook->,thick,dashed] (A) --node[above]{$\emb$} (B);
            \draw[->,thick] (B) --node[above]{$\partial^{p+q}$} (C);
            \draw[->>,thick,dashed] (C) --node[above]{$\proj$} (D);
          \end{tikzpicture}
\end{center}
and second constructing the data for the spectral sequence as
\[
 E_r^{p,q} \coloneqq \frac{\domain( \partial_r^{p,q})}{\defect( \partial_r^{p-r,q+(r-1)} )}
\]
and
\begin{center}
          \begin{tikzpicture}[transform shape,mylabel/.style={thick, draw=black, align=center, minimum width=0.5cm, minimum height=0.5cm,fill=white}]
            \coordinate (r) at (3.5,0);
            \coordinate (u) at (0,2);
            \node (dotsl) {$d_r^{p,q}$};
            \node (A) at ($(dotsl) + 0.5*(r)$) {$E_r^{p,q}$};
            \node (B) at ($(A) + (r)$) {$M^{p+q}$};
            \node (C) at ($(B) + (r)$) {$M^{p+q+1}$};
            \node (D) at ($(C) + (r)$) {$E_r^{p+r,q-(r-1)} $.};
            
            \draw[draw = none] (dotsl) --node[]{$\coloneqq$} (A);
            \draw[right hook->,thick,dashed] (A) --node[above]{$\emb$} (B);
            \draw[->,thick] (B) --node[above]{$\partial^{p+q}$} (C);
            \draw[->>,thick,dashed] (C) --node[above]{$\proj$} (D);
          \end{tikzpicture}
\end{center}

Note that all our constructions in this subsection were formulated purely in the language of generalized morphisms.
We have seen that computing with generalized morphisms only involves computations in the underlying abelian category
like taking pushouts and pullbacks.
It follows that we reached our second computational goal: computing the differentials on the pages of a spectral sequence associated to a filtered
cochain complex only with the help of direct computations in the underlying abelian category.

%% file: IntroCCT.bbl
\def\cprime{$'$} \def\cprime{$'$} \def\cprime{$'$} \def\cprime{$'$}
  \def\cprime{$'$}
\providecommand{\bysame}{\leavevmode\hbox to3em{\hrulefill}\thinspace}
\providecommand{\MR}{\relax\ifhmode\unskip\space\fi MR }
\providecommand{\MRhref}[2]{%
  \href{http://www.ams.org/mathscinet-getitem?mr=#1}{#2}
}
\providecommand{\href}[2]{#2}